\documentclass[12pt,a4paper]{amsart}
\usepackage[left=3cm,right=3cm,top=3cm,bottom=3cm]{geometry}

\usepackage[colorlinks=true,linktocpage=true,pagebackref=false,linkcolor=blue]{hyperref}
\usepackage{xcolor}
\usepackage{tikz, tikz-cd}
\usepackage{adjustbox}

\usepackage[cal=euler,scr=boondox]{mathalfa}

\usepackage{amssymb}
\usepackage{amsrefs}

\usepackage{bm}

\usepackage{etoolbox}
\apptocmd{\sloppy}{\hbadness 10000\relax}{}{}

\DefineSimpleKey{bib}{primaryclass}{}
\DefineSimpleKey{bib}{archiveprefix}{}

\BibSpec{arXiv}{%
  +{}{\PrintAuthors}{author}
  +{,}{ \textit}{title}
  +{}{ \parenthesize}{date}
  +{,}{ arXiv }{eprint}
}

\usepackage{booktabs}
\usepackage{multirow}
\usepackage{adjustbox}

\makeatletter
\@namedef{subjclassname@2020}{\textup{2020} Mathematics Subject Classification}
\makeatother

\newtheorem{thm}{Theorem}[section]
\newtheorem{prop}[thm]{Proposition}
\newtheorem{lemma}[thm]{Lemma}
\newtheorem{corol}[thm]{Corollary}
\newtheorem{conj}[thm]{Conjecture}
\theoremstyle{remark} \newtheorem{rmk}[thm]{Remark}
\theoremstyle{definition} \newtheorem{defn}[thm]{Definition}
\theoremstyle{definition} 
\theoremstyle{definition} \newtheorem{ej}[thm]{Example}

\DeclareMathOperator{\Aut}{Aut}
\DeclareMathOperator{\Out}{Out}
\DeclareMathOperator{\Int}{Int}
\DeclareMathOperator{\Hom}{Hom}
\DeclareMathOperator{\End}{End}

\DeclareMathOperator{\Spec}{Spec}
\DeclareMathOperator{\GL}{GL}
\DeclareMathOperator{\SL}{SL}
\DeclareMathOperator{\id}{id}
\DeclareMathOperator{\an}{an}
\DeclareMathOperator{\Ad}{Ad}
\DeclareMathOperator{\cl}{cl}
\DeclareMathOperator{\Env}{Env}

\DeclareMathOperator{\inv}{inv}
\DeclareMathOperator{\Higgs}{Higgs}

\newcommand{\git}{\mathbin{/\mkern-6mu/}}
\newcommand{\XX}{\mathrm{X}}
\newcommand{\GG}{\mathbb{G}}
\newcommand{\Aa}{\mathbb{A}}
\newcommand{\aA}{\mathcal{A}}
\newcommand{\CC}{\mathbb{C}}
\newcommand{\RR}{\mathbb{R}}

\newcommand{\ZZ}{\mathbb{Z}}
\newcommand{\HH}{\mathbb{H}}
\newcommand{\g}{\mathfrak{g}}

\newcommand{\Tt}{\mathfrak{t}}
\newcommand{\OO}{\mathcal{O}}
\newcommand{\Oo}{\mathscr{O}}

\newcommand{\B}{\mathbb{B}}
\newcommand{\BB}{\mathcal{B}}
\newcommand{\MM}{\mathcal{M}}
\newcommand{\Mm}{\mathbf{M}}
\newcommand{\eps}{\varepsilon}
\newcommand{\Cc}{\mathbf{C}}
\newcommand{\ad}{\mathrm{ad}}
\newcommand{\EE}{\mathcal{E}}
\newcommand{\Rr}{\mathrm{R}}
\newcommand{\PP}{\mathrm{P}}
\newcommand{\pr}{\mathrm{pr}}
\newcommand{\Ab}{\mathbf{A}}

\author[G. Gallego]{Guillermo Gallego}
\address{Facultad de Ciencias Matemáticas, UCM, Plaza Ciencias 3, 28040 Madrid, SPAIN}
\email{guigalle@ucm.es}
\author[O. Garc\'ia-Prada]{Oscar García-Prada}
\address{Instituto de Ciencias Matemáticas, CSIC-UAM-UC3M-UCM, Nicolás Cabrera, 13--15, 28049 Madrid, SPAIN}
\email{oscar.garcia-prada@icmat.es}

\thanks{The first author's research is supported by the UCM and Banco Santander under the contract CT63/19-CT64/19. This work of the second author is partially supported by the Spanish Ministry of Science and Innovation, through the ``Severo Ochoa Programme for Centres of Excellence in R\&D (CEX2019-000904-S)" and grant No. PID2019-109339GB-C31.}


\title{Multiplicative Higgs bundles and involutions}

\begin{document}
\begin{abstract}
	In this paper we generalize the theory of multiplicative $G$-Higgs bundles over a curve to pairs $(G,\theta)$, where $G$ is a reductive algebraic group and $\theta$ is an involution of $G$. This generalization involves the notion of a multiplicative Higgs bundle taking values in a symmetric variety associated to $\theta$, or in an equivariant embedding of it. We also study how these objects appear as fixed points of involutions of the moduli space of multiplicative $G$-Higgs bundles, induced by the involution $\theta$.
\end{abstract}
\maketitle

\section{Introduction}
\subsection{Multiplicative Higgs bundles} Let $k$ be an algebraically closed field of characteristic $0$, $G$ a reductive algebraic group over $k$ and $X$ a smooth algebraic curve over $k$. A \emph{multiplicative $G$-Higgs bundle} on $X$ is a pair $(E,\varphi)$ where $E$ is a principal $G$-bundle over $X$ and $\varphi$ is a section of the adjoint group bundle $E(G)$ over $X\setminus |D|$ the complement of a finite subset $|D|\subset X$. This is indeed a  ``multiplicative" version of an (ordinary, twisted)  \emph{$G$-Higgs bundle} on $X$, which is a pair $(E,\varphi)$ with $E$ a principal $G$-bundle over $X$ and $\varphi$ a section of the adjoint Lie algebra bundle $E(\g)$ twisted by some line bundle $L$ on $X$. As in the ordinary case, one wants to prescribe the singularities of $\varphi$. In the multiplicative case, instead of fixing the twisting line bundle $L$, one controls the singularity at each point $x\in |D|$ by fixing a dominant cocharacter $\lambda\in \XX_*(T)_+$, for a given choice $T\subset B\subset G$ of maximal torus $T$ and Borel subgroup $B$.

Just like for ordinary Higgs bundles \cite{hitchin}, one can define a Hitchin map in the multiplicative case, which is based on the ``multiplicative version"  of the Chevalley restriction map $G\rightarrow T/W$, for $W=N_G(T)/T$ the Weyl group of $T$. This leads to the \emph{multiplicative Hitchin fibration}. A general reference reviewing and explaining the several perspectives around multiplicative Higgs bundles and the multiplicative Hitchin fibration is presented by Elliot and Pestun \cite{elliot-pestun}.

Multiplicative Higgs bundles were originally introduced in algebraic geometry by Hurtubise and Markman \cite{hurtubise-markman}, although they had made previous appearances in the physics literature (see \cite{elliot-pestun} for references). Hurtubise and Markman considered the particular case where $X$ is an elliptic curve and, in this case, they constructed an algebraic symplectic form on the moduli space of \emph{simple} multiplicative Higgs bundles and proved that the multiplicative Hitchin fibration defines an algebraically completely integrable system. As explained in Elliot--Pestun \cite{elliot-pestun}, these results can be generalized to the case where $X$ is Calabi--Yau, that is, if $X$ is an elliptic curve, $\Aa^1$, or $\GG_m$.

The multiplicative Hitchin fibration was also studied independently by Frenkel and Ngô \cite{frenkel-ngo} in the context of geometrization of trace formulas. In particular, they suggest the use of Vinberg's enveloping monoid to describe the moduli stack of multiplicative Higgs bundles and construct the multiplicative Hitchin fibration. That paper was the starting point for a programme to study the multiplicative Hitchin fibration from the point of view of Vinberg's theory of reductive monoids \cite{vinberg}. This has been continued in several papers by Bouthier and J. Chi \cites{bouthier_Springer, bouthier_fibration, bouthier-chi,chi}, which study analogues of affine Springer fibres, and in the thesis and subsequent work of G. Wang \cite{wang}, dedicated to the Fundamental Lemma of Langlands--Shelstad for the spherical Hecke algebras, in the spirit of Ngô's work on the Fundamental Lemma for the Lie algebras \cite{ngo_lemme}.

When the base field is $k=\CC$ the field of complex numbers, multiplicative $G$-Higgs bundles on $X$ have also been related to singular $K$-monopoles on the $3$-manifold $X^{\an}\times S^1$, for $K\subset G$ a maximal compact subgroup, in the works of Charbonneau--Hurtubise \cite{charbonneau-hurtubise} and Smith \cite{smith}. The correspondence is given by constructing a multiplicative Higgs bundle as the \emph{scattering map} of a monopole. In the other direction, by ``gluing"  from a multiplicative Higgs bundle one constructs a holomorphic bundle on the $4$-manifold $X^{\an}\times S^1 \times S^1$, solves the Hermitian Yang--Mills equations on it and reduces to $X^{\an}\times S^1$ to obtain a monopole. These equations can be solved provided the stability condition on the holomorphic bundle on the $4$-manifold, which translates into an ad-hoc stability condition for the multiplicative Higgs bundle. The Charbonneau--Hurtubise--Smith correspondence is widely extended in Mochizuki's book \cite{mochizuki}. The moduli space of singular monopoles on $\CC\times S^1$ is known to be a hyperkähler manifold \cite{cherkis-kapustin}. Using Mochizuki's correspondence, Elliot and Pestun \cite{elliot-pestun} show that the holomorphic symplectic structure defined by that hyperkähler structure coincides with the Hurtubise--Markman algebraic symplectic structure on the moduli space of simple multiplicative Higgs bundles on $\Aa^1$. Furthermore, they study the twistor space for this hyperkähler manifold by introducing the notion of a \emph{$q$-difference connection}.

\subsection{Multiplicative Higgs bundles and involutions} The purpose of this paper is two-fold.

On the one hand, we introduce a generalization of multiplicative Higgs bundles from groups to pairs $(G,\theta)$ where $G$ is a reductive group and $\theta\in \Aut_2(G)$ is an involution of $G$. A \emph{multiplicative $(G,\theta)$-Higgs bundle} on $X$ is formed by a principal $G^\theta$-bundle $E\rightarrow X$ and a section $\varphi$ of the associated bundle $E(G/G^\theta)$ over the complement of a finite subset of $X$. Here, $G^\theta\subset G$ is the subgroup of fixed points of $\theta$ and $G/G^\theta$ is the corresponding symmetric variety. The group $G^\theta$ acts on $G/G^\theta$ by left multiplication and thus $E(G/G^\theta)\rightarrow X$ denotes the bundle of symmetric varieties associated to this action. 

Moreover, we generalize the description of the multiplicative Hitchin map in terms of the theory of reductive monoids, as developed in \cites{bouthier_Springer, bouthier_fibration, bouthier-chi,chi, frenkel-ngo,wang}, to this involutive situation. We do this by studying the theory of embeddings of symmetric varieties and using Guay's enveloping embedding \cite{guay}, and by extending the results of Richardson \cite{richardson}.

On the other hand, for any such pair $(G,\theta)$, we consider the involutions of the moduli space of (simple) multiplicative $G$-Higgs bundles given by
\begin{equation*}
\iota_{\pm}: (E,\varphi) \mapsto (\theta(E),\theta(\varphi)^{\pm 1})
\end{equation*} 
and study their fixed points. We show that the fixed points of $\iota_+$ define an algebraic symplectic submanifold of the moduli space of multiplicative Higgs bundles, while the fixed points of $\iota_-$ define a Lagrangian submanifold. We are specially interested in the involution $\iota_-$ and we show that multiplicative $(G,\theta)$-Higgs bundles appear as their fixed points, although also other components show up.

One of the reasons why multiplicative $(G,\theta)$-Higgs bundles are important is that when $k=\CC$ they are the multiplicative object analogous to Higgs bundles for \emph{real forms} of $G$. We recall that if $G_\RR$ is a real form of  $G$, a ($L$-twisted, for $L\rightarrow X$ a line bundle) $G_{\RR}$-Higgs bundle on $X$ is a pair $(E,\varphi)$, where $E$ is an $H$-bundle on  $X$ and $\varphi$ is a section of $E(\mathfrak{m})\otimes L$. Here, $H$ is the complexification of a maximal compact subgroup $H_\RR\subset G_\RR$ and  $\mathfrak{m}$ is the complexification of the space $\mathfrak{m}_\RR$ appearing in the Cartan decomposition $\g_\RR=\mathfrak{h}_\RR \oplus \mathfrak{m}_\RR$. 

There is a well known correspondence between real forms and involutions of $G$. If one regards a real form as the fixed points of an anti-holomorphic involution $\sigma:G\rightarrow G$, then one can find a commuting compact real form $\tau:G\rightarrow G$ and consider the involution $\theta=\sigma\circ \tau$. In this language, if $G_{\RR}=G^\sigma$, then $H=G^\theta$ and $\mathfrak{m}$ is the $(-1)$-eigenspace in the decomposition $\g=\mathfrak{h}\oplus \mathfrak{m}$. Note that $\mathfrak{m}$ is the tangent space of $G/G^\theta$ at $1\in G$, so a multiplicative $(G,\theta)$-Higgs bundle is indeed a multiplicative analogue of a $G_\RR$-Higgs bundle. In \cite{oscar-ramanan}, $G_{\RR}$-Higgs bundles receive the name of $(G^\theta,-)$-Higgs bundles and are shown to appear as fixed points of the involution $(E,\varphi)\mapsto (\theta(E),-\theta(\varphi))$. 

In any case, the definition of $G_{\RR}$-Higgs bundles originated from the nonabelian Hodge correspondence, through which they yield representations of the fundamental group of $X^{\an}$ in $G_\RR$. A similar correspondence is not known in the multiplicative situation. We expect to obtain some clarification of this by looking at multiplicative $(G,\theta)$-Higgs bundles or studying the fixed points of $\iota_-$ from the other side of the Charbonneau--Hurtubise--Smith correspondence. This is explored in the PhD Thesis of the first author \cite{tesis}, and will be the topic of a forthcoming paper.

Another reason to study the objects and involutions related to multiplicative Higgs bundles and pairs $(G,\theta)$ is to study \emph{branes}. In mirror symmetry, symplectic subvarieties of holomorphic symplectic manifolds are the support of \emph{$B$-branes}, while Lagrangian subvarieties are the support of \emph{$A$-branes}. In the language of hyperkähler geometry, holomorphic $B$-branes are $(B,B,B)$-branes, while holomorphic $A$-branes are $(B,A,A)$-branes.

For the classical Hitchin fibration, \emph{$S$-duality} is expected to give a correspondence between $B$-branes in the moduli space of $G$-Higgs bundles and $A$-branes in the moduli space of $\check{G}$-Higgs bundles, for $\check{G}$ the Langlands dual group of $G$. More precisely, since the Hitchin bases for $G$ and $\check{G}$ are canonically identified, the two moduli spaces fibre over the same space and Fourier--Mukai transforms on the fibres are expected to give the construction of $A$-branes on one of them from $B$-branes on the other. In particular, it is suggested that Higgs bundles for real forms $G_\RR$ of $G$, that give $A$-branes in the moduli space of $G$-Higgs bundles, should correspond to $\check{G}_{G_\RR}$-Higgs bundles, for $\check{G}_{G_\RR}\subset \check{G}$ the Nadler dual group of the real form $G_\RR$ \cite{nadler}. This conjectures arise from the work of Kapustin and Witten \cite{kapustin-witten} relating the geometric Langlands program with gauge theory, from the results on the duality of Hitchin systems given by Donagi and Pantev \cite{donagi-pantev}, and from a gauge-theoretical description of the construction of the Nadler group, given by Gaiotto and Witten \cite{gaiotto-witten}, and has been further studied in several works \cites{baraglia-schaposnik, oscar-biswas, oscar-biswas-hurtubise, branco, hitchin_hirzebruch}. The Nadler dual group coincides with $\check{G}_{G/G^\theta}$, the \emph{dual group} of the symmetric variety $G/G^\theta$, for $\theta$ the involution of $G$ corresponding to the real form $G_\RR$, introduced in the context of spherical varieties \cites{sake-venka, knop-schalke}. The dual group of a spherical variety plays an important role in the recent extension of the Langlands program (both in its classical and geometric flavours) to spherical varieties; this is the \emph{relative} Langlands program of Ben-Zvi, Sakellaridis and Venkatesh \cite{BZSV}. 
See Section \ref{section:dual-group} for more about dual groups. 

In \cite{elliot-pestun}*{Pseudo-Conjecture 3.9}, Elliot and Pestun suggest the existence of a similar correspondence between $A$-branes and $B$-branes in the moduli space of multiplicative $G$-Higgs bundles, when $X$ is Calabi--Yau and $G$ a Langlands self-dual group. If $G$ is a Langlands self-dual group, we can identify a maximal torus $T\subset G$ with its dual $\check{T}=\Spec k[e^{\XX_*(T)}]$, as well as the lattices of characters and cocharacters $\XX^*(T)\leftrightarrow \XX_*(T)$ and the roots and coroots $\Phi_T\leftrightarrow \Phi_T^\vee$. Moreover, in that case, the dual group $\check{G}_{G/G^\theta}$ of $G/G^\theta$ is naturally a subgroup of $G$. We thus suggest the following.

\begin{conj} \label{conjecture}
If $X$ is Calabi--Yau and $G$ is a Langlands self-dual group, for any involution $\theta\in \Aut_2(G)$, multiplicative $(G,\theta)$-Higgs bundles inside the moduli space of multiplicative $G$-Higgs bundles define the support of an $A$-brane dual to the $B$-brane given by multiplicative $\check{G}_{G/G^\theta}$-Higgs bundles.	
\end{conj}

\subsection{Outline of the paper} This paper consists of three different sections (apart from this introduction) and an appendix.

Section \ref{involutions} covers all the preliminary notions that we need about reductive groups with involutions, symmetric varieties and their embeddings, and their invariant theory. 
We begin by reviewing some general facts about reductive groups with involutions. We then introduce $\theta$-split tori and $\theta$-split parabolic subgroups, and recall the notion of a \emph{quasisplit} involution and the Iwasawa decomposition. We then recall the construction of the restricted root system and its associated root and weight lattices. We continue by describing the weight lattices and the weight semigroups of symmetric varieties and their embeddings, in the context of the theory of spherical varieties. We also recall the construction of the dual group of a symmetric variety, as introduced by Sakellaridis and Venkatesh \cite{sake-venka} and state its most relevant properties for us. We then review Guay's classification of a certain class of symmetric embeddings, which we call ``very flat" (in analogy with very flat monoids), and his construction of the ``enveloping" embedding \cite{guay}. Guay's enveloping embedding is an affine variety associated to a semisimple simply-connected group with involution that generalizes Vinberg's enveloping monoid. We also introduce the \emph{wonderful compactification} of a symmetric variety, and explain that the Guay embedding can be obtained as the spectrum of its Cox ring.  After this, we study the invariant theory of symmetric embeddings: first we recall Richardson's results on the invariant theory of a symmetric variety \cite{richardson}, and in Proposition \ref{invariant} we extend them to very flat symmetric embeddings. We finish the section by describing the loop parametrization of a symmetric variety, and extend it to the enveloping embedding in Proposition \ref{prop:extender_loops}, thus giving a generalization of \cite{chi}*{2.5}.

Section \ref{hitchin} introduces the main objects we are interested in, from several perspectives that we compare. 
In the first part of the section we emulate the construction of the multiplicative Hitchin fibration associated to a reductive monoid, as explained in \cite{wang}, in order to construct a multiplicative Hitchin map associated to a very flat symmetric embedding acted by a semisimple simply-connected group. In the second part of the section we introduce the notion of a multiplicative Higgs bundle associated to a reductive group with involution. We finish the section by relating the objects defined in the second part (with semisimple group) with the objects defined in the first part.

Finally, in Section \ref{invhiggs} we study a certain type of involutions in the moduli space of simple multiplicative Higgs bundles, how the objects introduced in Section \ref{hitchin} appear as their fixed points, and the relation between these involutions and the Hurtubise--Markman symplectic form.
More precisely, the section is dedicated to the study of the two involutions $\iota_{\pm}:(E,\varphi)\mapsto (\theta(E),\theta(\varphi)^{\pm 1})$, of the moduli space of simple multiplicative $G$-Higgs bundles, associated to an involution $\theta\in \Aut_2(G)$. We begin by recalling Hurtubise and Markman's \cite{hurtubise-markman} description of this moduli space of simple multiplicative $G$-Higgs bundles. We then show how multiplicative $(G,\theta)$-Higgs bundles are naturally mapped inside the fixed points of $\iota_-$ in Proposition \ref{inclusion} and give a detailed description of the fixed points in Theorem \ref{fixedpoints} and Corollary \ref{fixedpoints2}. We finish the section by reviewing the definition of the Hurtubise--Markman algebraic symplectic form in the moduli space of simple multiplicative $G$-Higgs bundles when the curve $X$ is assumed to be Calabi--Yau, and show in Theorem \ref{symplectic} how this symplectic form behaves under the above involutions. More precisely we show that the fixed points of $\iota_+$ define a symplectic submanifold, while the fixed points of $\iota_-$ give a Lagrangian submanifold. This is then a first result that motivates Conjecture \ref{conjecture} stated above.

At the end of the paper, we have included an Appendix (\ref{appendix}) to cover the basic facts that we use about root systems, and to establish our notation.

\subsection{Some notations} We establish now some notations regarding characters and cocharacters of algebraic tori. 
Given an algebraic torus $T$ we denote by $\XX^*(T)=\Hom(T,\GG_m)$ its lattice of characters and by $\XX_*(T)=\Hom(\GG_m,T)$ the lattice of cocharacters. Throughout the text we will use additive notation for characters and cocharacters. Thus, given an element $t\in T$ and a character $\chi\in \XX^*(T)$, we denote by $t^\chi\in \GG_m$ the image of the corresponding morphism $T\rightarrow \GG_m$. On the other hand, given an element $z\in \GG_m$ and a cocharacter $\lambda\in \XX_*(T)$, we denote by $z^\lambda \in T$ the image of the corresponding $1$-parameter subgroup $\GG_m\rightarrow T$. We consider the vector space $\EE_T=\XX^*(T)\otimes_{\ZZ}\RR$, endowed with the scalar product $(\cdot ,\cdot)$ induced by the choice of an invariant bilinear form on $T$ (which we also choose to be invariant under $\theta$ and under the Weyl group $W_T$ of $T$), and identify $\XX^*(T)$ and $\XX_*(T)$ as lattices inside $\EE_T$ and $\EE_T^*$, respectively. We can also consider the root system of $T$ inside $(\EE_T,(\cdot,\cdot))$; this is the set $\Phi_T\subset \XX^*(T)$ of the weights of the adjoint representation of $T$ in its Lie algebra $\mathfrak{t}$. See the Appendix \ref{appendix} for more notations regarding root systems.

\subsubsection*{Acknowledgements} We would like to thank Guillermo Barajas, Thomas Hameister, Jacques Hurtubise, Benedict Morrisey, Ngô Bao Châu and Griffin Wang for fruitful discussions, suggestions and references. We also thank the anonymous referee for their comments and corrections.

\section{Involutions, symmetric varieties and symmetric embeddings} \label{involutions}
\subsection{Some generalities on involutions} Let $G$ be a reductive algebraic group over $k$. By an \emph{involution} of $G$ we mean an order $2$ automorphism $\theta\in \Aut_2(G)$. To any such involution $\theta$ we can associate the subgroups
\begin{align*}
	G^\theta &= \left\{g\in G: \theta(g)=g\right\}	, \\
	G_\theta &= \left\{g\in G: g\theta(g)^{-1}\in Z_G\right\}	,
\end{align*}
where $Z_G\subset G$ is the centre of $G$.
The group $G^\theta$ is the group of fixed points of $\theta$ while $G_\theta$ can be shown to be its normalizer \cite{deconcini-procesi}*{1.7}. We can also consider $G_0^\theta=(G^\theta)^0$ the neutral connected component of $G^\theta$, and one can show that if $G$ is of simply-connected type then $G^\theta$ is connected, so $G_0^\theta=G^\theta$ \cite{steinberg_endomorphisms}*{Theorem 8.1}.

An involution $\theta$ of $G$ defines the \emph{$\theta$-twisted conjugation action} of $G$ on itself
\begin{align*}
G \times G & \longrightarrow G \\
(g,s) & \longmapsto g*_\theta s = g s \theta(g)^{-1}.
\end{align*} 
For each $s\in G$ we denote by $\tau^\theta_s:G\rightarrow G$ the map $\tau^\theta_s(g)=g*_\theta s$ sending $G$ to the $\theta$-twisted $G$-orbit $\tau^\theta_s(G)=G*_\theta s$. We also denote these $\theta$-twisted orbits by $M^\theta_s=\tau^\theta_s(G)$. In particular, we denote $\tau^\theta=\tau^\theta_1$ and $M^\theta=\tau^\theta(G)$. One can easily check that the isotropy subgroup of $s\in G$ is the fixed point subgroup $G^{\theta_s}$ of the automorphism
\begin{equation*}
\theta_s = \Int_s \circ \theta.
\end{equation*} 
Here, $\Int_s$ stands for the automorphism of $G$ given by conjugation by $s$; the elements of this form are the so-called \emph{inner automorphisms}, that form a subgroup $\Int(G)\subset\Aut(G)$. 

From the above, we conclude that the $\theta$-twisted orbits $M^\theta_s$ are homogeneous spaces of the form $G/G^{\theta_s}$, but not necessarily symmetric varieties. In particular, note that the orbit $M^\theta$ is naturally identified with the symmetric variety $G/G^\theta$. Moreover, this identification is a $G$-equivariant isomorphism, where $G$ acts on $M^\theta$ by $\theta$-twisted conjugation and on $G/G^\theta$ by left multiplication. We can ask ourselves for which $s\in G$ the automorphism $\theta_s$ is an involution. A simple computation shows that $\theta_s^2=\id_G$ if and only if $s\theta(s)$ is an element of the centre $Z_G$. The set of such $s$ is denoted by
\begin{equation*}
S_\theta = \left\{s\in G: s\theta(s)\in Z_G\right\}.
\end{equation*} 
For $s\in S_\theta$, the $\theta$-twisted orbit $M^\theta_s$ gets explicitly identified with $M^{\theta_s}$ as
\begin{equation*}
M^\theta_s=M^{\theta_s}s.
\end{equation*} 
In particular, note that if $s\in M^\theta$, then $M^{\theta_s}s=M^\theta_s=M^\theta$, so for $s\in M^\theta$ the symmetric varieties $G/G^\theta$ and $G/G^{\theta_s}$ can be identified. Moreover, if $s=g\theta(g)^{-1}$, we can identify $\theta_s=\Int_g \circ \theta \circ \Int_g^{-1}$ and $G^{\theta_s}=gG^\theta g^{-1}$.

It makes sense then to define an equivalence relation $\sim$ on $\Aut_2(G)$ :
\begin{equation*}
\text{$\theta\sim \theta'$ if and only if there exists $\alpha \in \Int(G)$ such that $\theta'=\alpha \circ \theta \circ \alpha^{-1}$,}
\end{equation*} 
and describe the quotient set $\Aut_2(G)/\sim$. We recall the description given in \cite{oscar-ramanan} in terms of what is called the ``clique map". 

The natural projection $\pi:\Aut_2(G)\rightarrow \Out_2(G):=\Aut_2(G)/\Int(G)$ factors through $\Aut_2(G)/\sim$. Indeed, this follows from the fact that
\begin{equation*}
\Int_g\circ \theta \circ \Int_g^{-1}=\Int_{g\theta(g)^{-1}}\circ \theta.
\end{equation*} 
We obtain a surjective map
\begin{align*}
\cl:\Aut_2(G)/\sim & \longrightarrow \Out_2(G), 
\end{align*} 
called the \emph{clique map}. The inverse image $\cl^{-1}(a)$ of a class $a\in \Out_2(G)$ is called the \emph{clique} of $a$. It is easy to check that for any $\theta\in \pi^{-1}(a)$ the clique $\cl^{-1}(a)$ is in bijection with the orbit set $S_\theta/(G\times Z_G)$, where $(G\times Z_G)$ acts on $S_\theta$ through the natural extension of the $\theta$-twisted conjugation action
\begin{align*}
	(G\times Z_G)\times G & \longrightarrow G \\
	((g,z),s) & \longmapsto zgs\theta(g)^{-1}.
\end{align*} 

The $\theta$-twisted conjugation action clearly preserves $S_\theta$ and also, for each $z\in Z_G$, the subset
\begin{equation*}
S^\theta_z=\left\{s\in G:s\theta(s)=z\right\}\subset S_\theta.
\end{equation*} 
In particular, it preserves the subset of ``anti-fixed"  points of $\theta$,
\begin{equation*}
S^\theta=S^\theta_{1}=\left\{s\in G: s=\theta(s)^{-1}\right\}.
\end{equation*} 
As explained in \cite{oscar-ramanan}, the quotient sets $S^\theta/G$ and $S_\theta/(G\times Z_G)$ admit an interpretation in terms of nonabelian group cohomology:
\begin{equation*}
S^\theta/G=H^1_\theta(\mathbb{Z}/2,G) \text{ and } S_\theta/(G\times Z_G)=H^1_\theta(\mathbb{Z}/2,G^\ad),
\end{equation*} 
where $G^\ad=G/Z_G$ and the subscript $\theta$ indicates that $\mathbb{Z}/2$ acts on $G$ and $G^\ad$ through $\theta$. Moreover, Richardson \cite{richardson2} shows that $H^1_\theta(\mathbb{Z}/2,G)$ (and thus $H^1_\theta(\ZZ/2,G)$) is finite.

\begin{rmk}
Consider the case where $k=\CC$. For any involution $\theta\in \Aut_2(G)$ we can find a maximal compact real form, defined by some conjugation $\sigma_c$ of $G$ commuting with $\theta$ and consider the composition $\sigma=\sigma_c\circ \theta$. This $\sigma$ defines a real form $G^\sigma$ of $G$. In \cite{adams-taibi}, Adams and Ta\"{\i}bi show that the nonabelian cohomology set $H^1_\theta(\ZZ/2,G)$ is isomorphic to the Galois cohomology $H^1_\sigma(\mathrm{Gal}(\CC/\RR),G)$. In particular, for $G^\ad$ one recovers Cartan's classification of real forms of $G$. Tables with $H^1_\theta(\ZZ/2,G)$ computed for semisimple simply-connected groups  can be found in \cite{adams-taibi}*{Section 10}. We also refer the reader to \cite{oscar-ramanan}*{Section 2.3} for more details about the correspondence between involutions and real forms.
\end{rmk}

\begin{ej}[The diagonal case] \label{ex:diagonal}
Groups with involution are in a certain sense a generalization of groups. Indeed, if $G$ is any linear algebraic group, we can consider the pair $(G\times G,\Theta)$, where $\Theta\in \Aut_2(G\times G)$ is the involution $\Theta(g_1,g_2)=(g_2,g_1)$. Let us denote by $\Delta:G\rightarrow G\times G$ the \emph{diagonal} map $\Delta(g)=(g,g)$ and by $\tilde{\Delta}$ the \emph{antidiagonal} $\tilde{\Delta}(g)=(g,g^{-1})$. The fixed point subgroup $(G\times G)^\Theta$ is the diagonal $\Delta(G)$, which can be naturally identified with $G$, whereas the orbit $M^{\Theta}$ is the antidiagonal $\tilde{\Delta}(G)$, which is again identified with $G$.	The set of anti-fixed points $S^\Theta$ is easily shown to be equal to  $M^\Theta$. We also have
\begin{align*}
	(G\times G)_\Theta  &= \left\{(zg,g):g\in G\right\}, \text{ and }
	S_\Theta = \left\{(zg,g^{-1}):g\in G\right\}.
\end{align*} 
Therefore, the orbit sets $S^\Theta/(G\times G)$ and $S_\Theta/(Z\times G\times G)$ are just singletons. The real form of $G\times G$ corresponding to  $\Theta$ is simply the group $G$ regarded as a real group.
\end{ej}

\begin{ej}[$G=\SL_n, n>2$] \label{ex:SLn}
In this case we have $\Out(G)\cong \ZZ/2$, so we distinguish two cases for $a\in \Out_2(G)$, $a=1$ and $a=-1$.	For $a=1$ we have 
\begin{equation*}
\cl^{-1}(1)=\left\{\theta_{p,q}: 0\leq p \leq q \leq n, p+q=n \right\},
\end{equation*} 
with $\theta_{p,q}(g)=I_{p,q}g I_{p,q}$, for
\begin{equation*}
I_{p,q} = 
\begin{pmatrix}
	I_p & 0 \\
	0 & -I_q
\end{pmatrix},
\end{equation*} 
where $I_p$ denotes the identity matrix of rank $p$. The corresponding groups of fixed points are
\begin{equation*}
G^{\theta_{p,q}}= \mathrm{S}(\GL_p\times \GL_q).
\end{equation*} 
Here, the $\mathrm{S}$ stands for taking the subset of matrices of determinant equal to $1$. Consider now the case $a=-1$. The clique  $\cl^{-1}(-1)$ consists of a single element if $n$ is odd and of two elements if $n=2m$ is even. In both cases we have the involution  $\theta_0(g)=(g^T)^{-1}$. 
The fixed points subgroup is 
\begin{equation*}
G^{\theta_0}= \mathrm{SO}_n.
\end{equation*} 
When $n=2m$ is even, we also have the involution $\theta_1(g)=J_m \theta_0(g)J_m^{-1}$, where $J_m$ is the symplectic matrix
\begin{equation*}
J_m =
\begin{pmatrix}
	0 & I_m \\
	-I_m & 0 
\end{pmatrix}.
\end{equation*} 
The fixed points subgroup is
\begin{equation*}
G^{\theta_1}= \mathrm{Sp}_{2m}.
\end{equation*} 

When $k=\CC$, we can also consider the real forms corresponding to these involutions. The compact real form of $\SL_n(\CC)$ is $\mathrm{SU}(n)$, which sits inside $\SL_n(\CC)$ as the fixed points of the conjugation $\sigma_K(g)=(g^\dagger)^{-1}$, where $^\dagger$ stands for taking transpose and complex-conjugation. Therefore, the real forms corresponding to the involution  $\theta_{p,q}$ are $\sigma_{p,q}(g)=I_{p,q}(g^\dagger)^{-1}I_{p,q}$, and $G^{\sigma_{p,q}}=\mathrm{SU}(p,q)$. Similarly, we have $\sigma_0(g)=\bar{g}$, so $G^{\sigma_0}=\SL_n(\RR)$, and, if $n=2m$, $\sigma_1(g)=J_m \bar{g} J_m^{-1}$, so $G^{\sigma_1}=\mathrm{SU}^*(2m)$.
\end{ej}

We refer to Table 26.3 in \cite{timashev} for a complete classification of the involutions of the simple groups.

\subsection{Split tori and split parabolics} \label{split}
Let $G$ be a reductive group over $k$ and $\theta\in \Aut_2(G)$ an involution. We say that a torus $A\subset G$ is \emph{$\theta$-split} if $\theta(a)=a^{-1}$ for every $a\in A$, and we say that it is \emph{maximal $\theta$-split} if it is maximal among $\theta$-split tori.

\begin{rmk}
The name ``$\theta$-split" comes from the correspondence with real groups when $k=\CC$. Indeed, a $\theta$-split torus is just the complexification of an $\RR$-split torus $A_{\RR}\subset G^\sigma$, for $\sigma$ the real form corresponding to $\theta$. 	
\end{rmk}

We now state some of the results of Vust \cite{vust_cones} on $\theta$-split tori.
\begin{prop}[Vust] \leavevmode
\begin{enumerate}
	\item Non-trivial $\theta$-split tori exist.	
	\item Any maximal torus $T$ containing a maximal $\theta$-split torus $A$ is \emph{$\theta$-stable}, meaning that $\theta(T)\subset T$.
	\item All maximal $\theta$-split tori are pairwise conjugated by elements of $G^\theta$.
	\item For any maximal $\theta$-split torus $A\subset G$, the group $G^\theta$ decomposes uniquely as $G^\theta=F^\theta G_0^\theta$, for
		\begin{equation*}
		F^\theta = A\cap G^\theta = \left\{a \in A: a^2=1\right\}.
		\end{equation*} 
\end{enumerate}	
\end{prop}

Related to point $(4)$ in the proposition above, we also have the following result of Richardson.

\begin{lemma}[\cite{richardson}*{Lemma 8.1.(a)}]
For any maximal $\theta$-split torus $A\subset G$, the group $G_\theta$ decomposes uniquely as $G_\theta=F_\theta G_0^\theta$, for
\begin{equation*}
F_\theta = A\cap G_\theta = \left\{a \in A: a^2 \in Z_G\right\}.
\end{equation*} 
\end{lemma}

A parabolic subgroup $P\subset G$ is \emph{$\theta$-split} if $P$ and $\theta(P)$ are opposite, that is, if $P\cap \theta(P)$ is a Levi subgroup of both $P$ and $\theta(P)$. If $P\subset G$ is a minimal $\theta$-split parabolic subgroup one can show \cite{vust_cones} that there exists a maximal $\theta$-split torus $A$ and a dominant cocharacter $\lambda \in \XX_*(A)_+$ such that $P$ is of the form
\begin{equation*}
P=P(\lambda) = \left\{p\in G: \lim_{t\rightarrow 0} t^\lambda p t^{-\lambda} \in G\right\}.
\end{equation*} 
Equivalently, if one takes $T\supset A$ a maximal torus, and considers $\Phi_T$ the associated root system and $\g = \Tt \oplus \bigoplus_{\alpha \in \Phi_\g} \g^\alpha$ the corresponding root space decomposition, then $P(\lambda)$ is the parabolic subgroup with Lie algebra
\begin{equation*}
\mathfrak{p}(\lambda) = \Tt \oplus \bigoplus_{\langle \alpha, \lambda \rangle \geq 0} \g^\alpha.
\end{equation*} 
This $P(\lambda)$ admits the Levi decomposition
\begin{equation*}
\mathfrak{l}=\Tt \oplus \bigoplus_{\langle \alpha, \lambda \rangle = 0} \g^\alpha=\mathrm{Lie}(Z_G(\lambda)) \ \ \text{ and } \ \ \mathfrak{p}(\lambda)^u = \bigoplus_{\langle \alpha, \lambda \rangle > 0} \g^\alpha,
\end{equation*} 
for $Z_G(\lambda)= Z_G(\left\{z^\lambda: z\in \CC^*\right\})$ the centralizer of the uniparametric subgroup induced by $\lambda$.
As a consequence, the Levi subgroup $P\cap \theta(P)$ is equal to
\begin{equation*}
P\cap \theta(P)= Z_G(\lambda) = Z_G(A).
\end{equation*} 

\begin{rmk}
	Again, the name ``$\theta$-split" comes from the correspondence with real groups when $k=\CC$. Indeed, a parabolic subgroup is $\theta$-split if and only if its corresponding real form is split. We can also consider parabolic subgroups such that their corresponding real form is quasisplit, which motivates the following.
\end{rmk}

\begin{defn}
An involution $\theta \in \Aut_2(G)$ is \emph{quasisplit} if there exists a $\theta$-split Borel subgroup $B\subset G$.	Equivalently, $\theta$ is quasisplit if there exists a maximal $\theta$-split torus $A$ such that $T=Z_G(A)$ is a maximal torus of $G$.
\end{defn}

We refer the reader to \cite{ABV} for the proof of the following fact.
\begin{lemma}\label{prop:quasisplit_existence}
Any class $a\in \Out_2(G)$ can be represented by a quasisplit involution.
\end{lemma}

A consequence of the existence of minimal $\theta$-split parabolics is the \emph{Iwasawa decomposition} (see \cite{timashev} for a proof).

\begin{lemma}[Iwasawa decomposition]
Let $P\subset G$ be a minimal $\theta$-split parabolic subgroup and $P^u$ its unipotent radical. Then the product $G^\theta_0 A P^u$ is an open subset of $G$.
\end{lemma}

\subsection{The restricted root system} \label{roots}
Let $A\subset G$ be a maximal $\theta$-split torus and $T\subset G$  a maximal $\theta$-stable torus containing it. Let $r$ denote the rank of $T$ and $l$ the rank of $A$. The number $l$ is called the \emph{rank of the symmetric variety} $G/G^\theta$. Let us consider $\Phi_T$ the root system of $T$. 

The involution $\theta$ naturally induces an involution $\chi \mapsto \chi^\theta$ on the characters $\XX^*(T)$. For each $\chi \in \XX^*(T)$ we define $\bar{\chi}=(\chi-\chi^\theta)/2$, which is a well defined element of $\EE_T$. Now, the lattice formed by the elements $\bar{\chi}$ of this form is naturally identified with the group of characters $\XX^*(A)$.

It is easy to check that the involution on characters of $T$ sends roots to roots. The elements of the set $\Phi_T^\theta$ of roots fixed under this involution are called the \emph{imaginary roots}. A choice of positive roots $\Phi_T^+\subset \Phi_T$ can be made in such a way that if $\alpha \in \Phi_T^+$ is not an imaginary root, then $\alpha^\theta$ is a negative root. Making this choice amounts to choosing a Borel subgroup $B\subset G$. Indeed, this is the Borel subgroup $B$ with Lie algebra
\begin{equation*}
\mathfrak{b}=\mathfrak{t} \oplus \sum_{\alpha \in \Phi_T^+} \g_\alpha.
\end{equation*} 
The following follows easily.
\begin{lemma}
If $\theta$ is quasisplit, then $\Phi_T^\theta = \varnothing$.	
\end{lemma}

We consider now the vector space $\EE_\theta=\XX^*(A)\otimes_\ZZ \RR$, which is naturally a subspace of $\EE_T$ and thus inherits the scalar product $(\cdot,\cdot)$.
\begin{defn}
The set of \emph{restricted roots} of $\theta$ is
\begin{equation*}
\Phi_\theta=\left\{\bar{\alpha}=\frac{\alpha- \alpha^\theta}{2} \in \EE_\theta : \alpha \in \Phi_T \setminus \Phi_T^\theta \right\}.
\end{equation*} 
\end{defn}
The following is well known (see, for example \cite{richardson}).

\begin{lemma}
The set $\Phi_\theta$ is a (possibly non-reduced) root system in $\EE_\theta$ with Weyl group the \emph{little Weyl group} $W_\theta=N_G(A)/Z_G(A)=N_{G_0^\theta}(A)/Z_{G_0^\theta}(A)$.
\end{lemma}

Simple roots for $\Phi_\theta$ can be constructed from the simple roots $\Delta_T$ of $\Phi_T$. Indeed, these roots can be split into two groups
\begin{equation*}
\Delta_T \cap \Phi_T \setminus \Phi_T^\theta = \left\{\alpha_1,\dots,\alpha_m \right\}, \ \ \Delta_T\cap \Phi_T^\theta = \left\{\beta_1,\dots,\beta_{r-m}\right\}.
\end{equation*} 
Now, one can order (see \cite{deconcini-procesi} for more details) the simple roots $\left\{\alpha_1,\dots,\alpha_m\right\}$ in such a way that the $\bar{\alpha}_i$ are mutually distinct for all $i\leq l$ and for each $j>l$ there exists some $i\leq l$ with $\bar{\alpha}_j=\bar{\alpha}_i$. Thus, we obtain the set of \emph{restricted simple roots}
\begin{equation*}
\Delta_\theta = \left\{\bar{\alpha}_1,\dots,\bar{\alpha}_l\right\}.
\end{equation*} 

There is a natural inclusion of the root lattice $\Rr_\theta:=\Rr(\Phi_\theta)=\ZZ\langle \Phi_\theta \rangle$ in the group of characters $\XX^*(A)$, and the cokernel of this inclusion is the group of characters of the subgroup
\begin{equation*}
\left\{a\in A: \bar{\alpha}(a)=1 \text{ for all }\bar{\alpha}\in \Phi_\theta\right\}.
\end{equation*} 
This group can be easily shown to be equal to $F_\theta=A\cap G_\theta$, and thus, if we denote  $A_{G_\theta}=A/F_\theta$, we get
\begin{equation*}
\Rr_\theta = \XX^*(A_{G_\theta}).
\end{equation*} 
Dually, we get $\PP^\vee_\theta=\PP^\vee(\Phi_\theta)=\XX_*(A_{G_\theta})$.

We describe now the weight lattice $\PP_\theta=\PP(\Phi_\theta)$ of $\Phi_\theta$. Let us denote by $\varpi_1,\dots,\varpi_l$ the corresponding fundamental dominant weights (that is, we want these weights to satisfy $\langle \bar{\alpha}_i^\vee, \varpi_j\rangle = \delta_{ij}$ for $i=1,\dots,l$). When $G$ is semisimple simply-connected, these $\varpi_i$ can be determined in terms of the fundamental dominant weights of $\Phi_T$. We can partition these into two sets
\begin{equation*}
\left\{\omega_1,\dots,\omega_m\right\}, \ \ \left\{\zeta_1,\dots,\zeta_{r-m}\right\}
\end{equation*} 
with $\langle \alpha_i^\vee, \omega_j \rangle= \delta_{ij}$, $\langle \beta_i^\vee,\zeta_j \rangle=\delta_{ij}$ and all the $\langle \alpha_i^\vee, \zeta_j \rangle = \langle \beta_i^\vee, \omega_j \rangle  =0$. 

Since a root $\alpha$ and its image ${\alpha}^\theta$ must have the same length, we can have three possible cases: (1) $\alpha^\theta=-\alpha$,  (2) $\langle \alpha^\vee, \alpha^\theta \rangle = 0$, and (3) $\langle \alpha^\vee,\alpha^\theta \rangle = 1$.
Note that if there is a root $\alpha$ of type (3) then $s_{\alpha^\theta}(\alpha)=\alpha-\alpha^\theta$ must be a root in $\Phi_T$, and this root restricts to itself, which is equal to $2\bar{\alpha}$. This implies that $\Phi_\theta$ is nonreduced and that we can take the restricted simple roots $\Delta_\theta$ coming just from simple roots of types (1) and (2).

Now, for any $\chi \in \XX^*(T)$ we have
\begin{equation*}
	\langle \bar{\alpha}^\vee, \chi-\chi^\theta \rangle = 
\begin{cases}
\langle \alpha^\vee, \chi \rangle & \text{ in case (1)} \\
	\langle \alpha^\vee, \chi-\chi^\theta \rangle & \text{ in case (2)} \\
	2\langle \alpha^\vee, \chi-\chi^\theta \rangle & \text{ in case (3)}.
\end{cases}
\end{equation*} 
As in \cite{deconcini-procesi}, it is easy to see that there exists some permutation $\sigma \in \mathfrak{S}_m$ of order $2$ such that for each $i=1,\dots, l$, we have that $\alpha_i^\theta + \alpha_{\sigma(i)}=\alpha_{\sigma(i)}^\theta + \alpha_i$ is a positive imaginary root. Moreover, we have $\omega_i^\theta = - \omega_{\sigma(i)}$. Then it is straightforward from the above that we get
\begin{equation*}
\varpi_i = 
\begin{cases}
	2\omega_i & \text{ if $\alpha_i$ is of type (1)} \\
	\omega_i+\omega_{\sigma(i)} & \text{ if $\alpha_i$ is of type (2) and $i\neq \sigma(i)$} \\
	\omega_i & \text{ if $\alpha_i$ is of type (2) and $i= \sigma(i)$}. 
\end{cases}
\end{equation*} 
Note that if $\Phi_T^\theta=\varnothing$, then $i=\sigma(i)$ if and only if $\alpha_i^\theta=-\alpha_{i}$, and thus we do not have the third situation in the formula above.

\begin{lemma}
When $G$ is semisimple simply connected, the weight lattice $\PP_\theta=\ZZ\langle \varpi_1,\dots,\varpi_l \rangle$ is equal to the group of characters $\XX^*(A_{G^\theta})$ of the torus $A_{G^\theta}=A/F^\theta$. Dually, we get $\Rr^\vee_\theta=\Rr^\vee(\Phi_\theta)=\XX_*(A_{G^\theta})$.
\end{lemma}
\begin{proof}
The lattice $\XX^*(A_{G^\theta})$ can be identified with $2\XX^*(A)=\left\{\chi-\chi^\theta: \chi \in \XX^*(T)\right\}$. It is clear from the above that for every $\chi\in \XX^*(T)$ and  $\bar{\alpha}\in \Phi_\theta$, the number $\langle \bar{\alpha}^\vee, \chi-\chi^\theta \rangle$ is an integer. On the other hand, $2\omega_i$ and $\omega_i+\omega_{\sigma(i)}$ are clearly of the form $\chi-\chi^\theta$, while clearly $\omega_i$ restricts to $1$ on $F^\theta=A\cap G^\theta$ if $\omega_i^\theta=-\omega_i$.
\end{proof}

To finish the section, note that
\begin{equation*}
\PP_{\theta,+} := \PP_+(\Delta_\theta)=\PP_\theta \cap C_{\Delta_\theta}^+ =\PP_\theta \cap C_{\Delta}^+= \XX^*(A_{G^\theta})\cap \XX^*(T)_+ =:\XX^*(A_{G^\theta})_+.
\end{equation*} 
By an analogous argument, we get $\PP_{\theta,+}^\vee=\XX_*(A_{G_\theta})\cap \XX_*(T)_+=\XX_*(A_{G_\theta})_+$. 

\begin{ej}[The diagonal case]
In the diagonal case from Example \ref{ex:diagonal}, a maximal $\Theta$-split torus is given by $A=\left\{(t^{-1},t):t\in T\right\}\subset T\times T$, with character lattice $\XX^*(A)=\left\{(-\chi,\chi):\chi \in \XX^*(T)\right\}$. The restricted roots are then identified with $\Phi_T$. 
\end{ej}

\begin{ej}[$\SL_n, n>2$]
Recall the classification of the involutions of $\SL_n$ given in Example \ref{ex:SLn}. We determine now the restricted root systems associated to each of these involutions. We begin by considering the inner involutions $\theta_{p,q}(g)=I_{p,q}gI_{p,q}$. Under a change of basis, we can write these involutions as $\theta_{p,q}(g)=I'_{p,q}g I'_{p,q}$, for
\begin{equation*}
I'_{p,q}=
\begin{pmatrix}
	0 & 0 & \tilde{I}_p \\	
	0 & I_{q-p} & 0 \\
	\tilde{I}_p & 0 & 0
\end{pmatrix},
\end{equation*} 
where $\tilde{I}_p$ is the $p\times p$ matrix
\begin{equation*}
\tilde{I}_p=
\begin{pmatrix}
	0 & \dots & 1 \\
	0 & \reflectbox{$\ddots$} & 0 \\
	1 &\dots & 0  
\end{pmatrix}.
\end{equation*} 
Now it is clear that the elements of the form $$a=\mathrm{diag}(a_1,\dots,a_p,1,\dots,1,a_p^{-1},\dots,a_1^{-1})$$ form a maximal $\theta_{p,q}$-split torus $A$. This torus sits inside the standard maximal torus $T$ of $\SL_n$, which is $\theta$-stable. The corresponding simple roots are $\Delta_{\mathfrak{sl}_n}=\left\{\alpha_1,\dots,\alpha_{n-1}\right\}$, with
\begin{equation*}
\alpha_i=\varepsilon_i - \varepsilon_{i+1},
\end{equation*} 
for $\eps_i(\mathrm{diag}(a_1,\dots,a_n)) = a_i$.
The imaginary roots are those of the form $\eps_i-\eps_j$, for  $p<i\neq j\leq q$; thus, $\theta_{p,q}$ is quasisplit if and only if $q=p$ or $q=p+1$. The restricted root system is
 \begin{equation*}
\Phi_\theta = \left\{\pm\bar{\eps}_i \pm \bar{\eps}_j, \pm 2 \bar{\eps}_i , \pm \bar{\eps}_i: 1\leq i\neq j \leq p\right\},
\end{equation*} 
when $p\neq n/2$, and  
 \begin{equation*}
\Phi_\theta = \left\{\pm\bar{\eps}_i \pm \bar{\eps}_j, \pm 2 \bar{\eps}_i: 1\leq i\neq j \leq p\right\},
\end{equation*} 
when $p=n/2$. Therefore, the restricted root system is nonreduced, of type BC$_p$ if $p\neq n/2$, and reduced of type  C$_p$ if  $p=n/2$. 

For the outer involution $\theta_0$, the standard torus is maximal $\theta_0$-split and $\Phi_\theta=\Phi_{T}$. On the other hand, if $n=2m$ is even, we can choose a basis such that $\theta_1(g)=I_{m,m} (g^{\tilde{T}})^{-1} I_{m,m}$, where $\tilde{T}$ denotes transposition with respect to the secondary diagonal. It is now easy to see that the elements of the form
\begin{equation*}
a= \mathrm{diag}(a_1,a_2,\dots,a_2,a_1),
\end{equation*} 
with $a_1a_2\cdots a_{m}=1$ form a maximal $\theta$-split torus $A$. The restricted root system now is
 \begin{equation*}
\Phi_\theta = \left\{\bar{\eps}_i-\bar{\eps}_j : 1\leq i\neq j \leq m\right\},
\end{equation*} 
so it has type A$_{m-1}$. Table \ref{tab:SLn} summarizes this example and Example \ref{ex:SLn}.

\begin{table}[h!]
\caption{Involutions of $\mathrm{SL}_n$, for $n>2$. Notation for real forms following Helgason \cite{helgason}.}
\label{tab:SLn}
\begin{adjustbox}{max width=.9\textwidth}
\begin{tabular}{@{}cccccc@{}}
\toprule
 &
  Involution &
  $G^\theta$ &
  Real form (when $k=\CC$) &
  Split or quasisplit? &
  $\Phi_\theta$ \\ \midrule
Inner &
  $\theta_{p,q}(g)=I_{p,q} g I_{p,q}$ &
  $\mathrm{S}(\mathrm{GL}_p\times \mathrm{GL}_q)$ &
  $\mathrm{SU}(p,q)$ &
  \begin{tabular}[c]{@{}c@{}}Quasisplit iff \\ $q=p$ or $q=p+1$\end{tabular} &
  \begin{tabular}[c]{@{}c@{}}BC$_p$ if $p\neq q$ \\ C$_p$ if $p=q$\end{tabular} \\ \midrule
\multirow{2}{*}{Outer} &
  $\theta_0(g)=(g^T)^{-1}$ &
  $\mathrm{SO}_n$ &
  $\mathrm{SL}_n(\mathbb{R})$ &
  Split &
  A$_{n-1}$ \\ \cmidrule(l){2-6}
 &
  \begin{tabular}[c]{@{}c@{}}$\theta_1(g)=J_m(g^T)^{-1}J_m^{-1}$\\ (only if $n=2m$)\end{tabular}&
  $\mathrm{Sp}_{2m}$ &
  $\mathrm{SU}^*(n)$ &
  No &
  A$_{m-1}$ \\ \midrule 
\end{tabular}
\end{adjustbox}
\end{table}
\end{ej}

\subsection{Symmetric pairs and symmetric varieties} \label{section_pairs}
\begin{defn}
Any closed subgroup $H\subset G$ of a reductive group $G$ with 
\begin{equation*}
	G_0^\theta \subset H \subset G_\theta
\end{equation*} 
for some involution $\theta\in \Aut_2(G)$ is called a \emph{symmetric subgroup} (associated to $\theta$). A pair $(G,H)$ where $G$ is a reductive group and $H\subset G$ is a symmetric subgroup is called a \emph{symmetric pair}, while the corresponding algebraic homogeneous space $G/H$ is called a \emph{symmetric variety}.
	
An \emph{embedding of a symmetric variety} or \emph{symmetric embedding} is a normal $G$-variety $\Sigma$ with a $G$-equivariant open embedding $O_\Sigma\hookrightarrow \Sigma$, where $O_\Sigma=G/H$ is a symmetric variety.
\end{defn}

\begin{rmk}
	Note that the Lie algebra of any symmetric subgroup associated to $\theta$ is equal to $\g^\theta$, the fixed points in $\g$ of the involution induced by $\theta$.
\end{rmk}

If $G$ is a reductive group, we can decompose it as $G=G'Z_{G}^0$, for $G'$ the derived subgroup and $Z_{G}^0$ the connected centre of $G$, which is a torus. Now, any involution $\theta\in \Aut_2(G)$ preserves the centre and its connected components, so $\theta$ acts on $G'$ as the involution $\theta'=\theta|_{G'}\in \Aut_2(G')$. Moreover, there exists a torus $Z$ such that any symmetric variety $G/H$ associated to $\theta$ is isomorphic to one associated to the involution $(g,z)\mapsto (\theta(g),z^{-1})$ on $G'\times Z$. 

Thus, from now on we will assume that any symmetric variety is of the form $G_Z/H_Z$ where $G_Z$ is the reductive group $G_Z=G\times Z$, with $G$ a semisimple group and $Z$ a torus, and that $G_Z/H_Z$ is associated to the involution $(g,z)\mapsto (\theta(g),z^{-1})$ for $\theta\in \Aut_2(G)$. The group $H=H_Z\cap (G\times \left\{1\right\})$ is a symmetric subgroup of $G$ associated to $\theta$. We denote $(G_Z/H_Z)'=G/H$ and call it the \emph{semisimple part} of $G_Z/H_Z$.

\subsection{Symmetric varieties as spherical varieties}
Let $\Sigma$ be a normal $G$-variety. We denote by $k[\Sigma]$ its ring of regular functions and by $k(\Sigma)$ its function field. For any subgroup $H\subset G$ we denote by $k[\Sigma]^H$ (respectively $k(\Sigma)^H$) the corresponding ring (resp. field) of $H$-invariant functions; that is, of functions $f$ with
\begin{equation*}
f(h\cdot x) = f(x) \text{ for any } h\in H \text{ and } x\in \Sigma.
\end{equation*} 
We denote by $k[\Sigma]^{(H)}$ (respectively $k(\Sigma)^{(H)}$) the corresponding ring (resp. field) of $H$-\emph{semiinvariants} of $\Sigma$; these are the functions $f$ such that there exists a character $\chi_f \in \XX^*(H)$ with
\begin{equation*}
f(h\cdot x) = h^{\chi_f} f(x) \text{ for any } h \in H.
\end{equation*} 

Let $B\subset G$ be a Borel subgroup. The characters that arise from $B$-semiinvariants of $\Sigma$ are called the \emph{weights} of $\Sigma$. More precisely, we define the \emph{weight lattice} and the \emph{weight semigroup} of $\Sigma$ to be, respectively
\begin{align*}
	\PP(\Sigma)&=\left\{\chi \in \XX^*(B): \exists f\in k(\Sigma)^{(B)}, \chi_f = \chi \right\}, \\
	\PP_+(\Sigma)&=\left\{\chi \in \XX^*(B): \exists f\in k[\Sigma]^{(B)}, \chi_f = \chi \right\}.
\end{align*} 
The rank of the weight lattice $\PP(\Sigma)$ is called the \emph{rank} of $\Sigma$.
When $\Sigma$ is affine, $\PP(\Sigma)$ can be determined in terms of $\PP_+(\Sigma)$ as $\PP(\Sigma)=\ZZ \PP_+(\Sigma)$.

A normal $G$-variety $\Sigma$ is \emph{spherical} if there exists a Borel subgroup $B\subset G$ and an open $B$-orbit in $\Sigma$. Such a $B$-orbit is of the form $Bx_0$, for some $x_0\in \Sigma$, and it is birational to $\Sigma$, so $k(\Sigma)^B=k(Bx_0)^B=k^\times$. Moreover, if $H\subset G$ is the stabilizer of  $x_0$, we have $Bx_0\cong B/H$ and thus  $k(\Sigma)^{(B)}\cong k(Bx_0)^{(B)}\cong k(B/H)^{(B)}$, so 
\begin{equation*}
\PP(\Sigma)=\XX^*(B/(B\cap H)).
\end{equation*} 
Restricting characters to a maximal torus $T\subset B$, we conclude that $\PP(\Sigma)=\XX^*(T_\Sigma)$, where $T_\Sigma$ is the torus $T_\Sigma=T/(T\cap H)$. If $\Sigma$ is affine, its coordinate ring is naturally a $G$-module, and as such it decomposes as a sum of irreducible representations of $G$ 
\begin{equation*}
k[\Sigma]=\bigoplus_{\chi \in \XX^*(T)_+}m_\chi V_\chi,
\end{equation*} 
for certain multiplicities $m_\chi$. This allows us to identify
\begin{equation*}
\PP_+(\Sigma)=\left\{\chi \in \XX^*(T)_+: m_\chi \geq 1\right\}.
\end{equation*} 
Moreover, if $\Sigma$ is spherical, then one can show that all the multiplicities $m_\chi$ are $\leq 1$, so $k[\Sigma]$ decomposes as
\begin{equation*}
k[\Sigma]=\bigoplus_{\chi \in \PP_+(\Sigma)} V_\chi.
\end{equation*} 
Finally, one can also prove that the weight semigroup of an affine spherical variety $\Sigma$ can be recovered as the lattice points on the cone generated by it; that is,
\begin{equation*}
\PP_+(\Sigma)=\mathbb{Q}_+\PP_+(\Sigma) \cap \ZZ \PP_+(\Sigma).
\end{equation*} 

It follows from the Iwasawa decomposition that for any involution $\theta\in \Aut_2(G)$ and any Borel subgroup $B$ of a minimal $\theta$-split parabolic subgroup $P\subset G$, there is an open $G$-equivariant embedding $B\hookrightarrow G/G^\theta_0$. This implies that all symmetric varieties and their embeddings are spherical. A first consequence of this is that, if $\Sigma$ is a symmetric embedding with $O_\Sigma=G/H$ a symmetric variety associated to an involution $\theta \in \Aut_2(G)$, then 
\begin{equation*}
\PP(\Sigma)=\XX^*(T_H),
\end{equation*} 
where $T_H=T/(T\cap H)$ for any $\theta$-stable maximal torus $T\subset G$. Note that if $A\subset G$ is a maximal $\theta$-split torus contained in $T$, then we also have 
\begin{equation*}
T_H=A_H = A/(A\cap H).
\end{equation*} 
In particular, we have
\begin{equation*}
\PP(G/G_\theta) = \XX^*(A_{G_\theta}) = \Rr_\theta.
\end{equation*} 
Moreover, we have
\begin{equation*}
\PP(G/G^\theta) = \XX^*(A_{G^\theta})
\end{equation*} 
which, if $G$ is semisimple simply-connected, is equal to the weight lattice $\PP_\theta$.

In general, one can recover the weight semigroup of a symmetric variety $G/H$ as the dominant weights of the corresponding torus $A_H$. More precisely, we have
 \begin{equation*}
\mathbb{Q}_+\PP_+(G/H)=C^+_{\Delta_\theta}.
\end{equation*} 
Therefore,
\begin{equation*}
\PP_+(G/H)=\XX^*(A_H)\cap C^+_{\Delta_\theta}=\XX^*(A_H)\cap C^+_\Delta =: \XX^*(A_H)_+.
\end{equation*}

\subsection{The dual group} \label{section:dual-group}

\begin{table}[hb!]
	\centering
\caption{Simple groups with their involutions, real forms and associated dual groups. Notation for types and real forms following Helgason \cite{helgason}.}
\label{tab:dual-groups}
\begin{adjustbox}{max width=\textwidth}
\begin{tabular}{@{}cccccccccc@{}}
\toprule
Type & $\mathfrak{g}$ &$\check{\mathfrak{g}}$  & $\mathfrak{g}_\RR$ (when $k=\CC$) & $\mathfrak{g}^\theta$ & $\Delta_\theta$ & $\Delta_\theta^{\mathrm{red}}$ & $\Delta_{G/G^\theta}$ & $\check{\mathfrak{g}}_{G/G^\theta}$ & Notes \\ 
\toprule
AI &
  $\mathfrak{sl}_n$ &
  $\mathfrak{sl}_n$ &
  $\mathfrak{sl}_n(\mathbb{R})$ &
  $\mathfrak{so}_n$ &
  A$_{n-1}$ &
  A$_{n-1}$ &
  A$_{n-1}$ &
  $\mathfrak{sl}_n$ & split \\
  \midrule
  AII & $\mathfrak{sl}_{2n}$  & $\mathfrak{sl}_{2n}$ & $\mathfrak{su}^*(2n)$ & $\mathfrak{sp}_{2n}$  & A$_{n-1}$ & A$_{n-1}$ & A$_{n-1}$ & $\mathfrak{sl}_n$ &  \\
  \midrule
  AIII & $\mathfrak{sl}_n$ & $\mathfrak{sl}_n$ & $\mathfrak{su}(p,q)$ & $\mathfrak{s}(\mathfrak{gl}_p \oplus \mathfrak{gl}_q)$  & BC$_p$  & B$_p$  & B$_p$ & $\mathfrak{sp}_{2p}$ & \begin{tabular}[c]{@{}c@{}} $p+q=n$, $p<q$ \\
  QS if $q=p+1$\end{tabular}\\
  \midrule
  AIV & $\mathfrak{sl}_{2n}$ & $\mathfrak{sl}_{2n}$ & $\mathfrak{su}(n,n)$ & $\mathfrak{s}(\mathfrak{gl}_n \oplus \mathfrak{gl}_n)$ & C$_n$  & C$_n$  & B$_n$ & $\mathfrak{sp}_{2n}$ & \begin{tabular}[c]{@{}c@{}} quasisplit \\
$\Delta_\theta^{\mathrm{red}}\neq \Delta_{G/G^\theta}$\end{tabular}  \\
\midrule
  BI & $\mathfrak{so}_{2n+1}$  & $\mathfrak{sp}_{2n}$ & $\mathfrak{so}(p,q+1)$ & $\mathfrak{so}_p\oplus \mathfrak{so}_q$ & B$_p$  & B$_p$ & B$_p$ & $\mathfrak{sp}_{2p}$ & $p+q=2n$, $p<q$ \\
	  \midrule
  BII & $\mathfrak{so}_{2n+1}$  & $\mathfrak{sp}_{2n}$ & $\mathfrak{so}(n,n+1)$ & $\mathfrak{so}_n\oplus \mathfrak{so}_{n+1}$ & B$_n$  & B$_n$ & B$_n$ & $\mathfrak{sp}_{2n}$ & split  \\
\midrule
  CI & $\mathfrak{sp}_{2n}$ & $\mathfrak{so}_{2n+1}$ & $\mathfrak{sp}_{2n}(\mathbb{R})$  & $\mathfrak{gl}_n$ & C$_n$ & C$_n$ & C$_n$ & $\mathfrak{so}_{2n+1}$ & split \\
  \midrule
  \multirow{2}{*}{CII} & $\mathfrak{sp}_{2n}$ & $\mathfrak{so}_{2n+1}$ & $\mathfrak{sp}(2p,2q)$  & $\mathfrak{sp}_{2p}\oplus\mathfrak{sp}_{2q}$ & BC$_p$ & B$_p$ & B$_p$ & $\mathfrak{sp}_{2p}$ & $p+q=n$, $p<q$ \\
\cmidrule(l){2-10}
   & $\mathfrak{sp}_{4n}$ & $\mathfrak{so}_{4n+1}$ & $\mathfrak{sp}(2n,2n)$  & $\mathfrak{sp}_{2n}\oplus\mathfrak{sp}_{2n}$ & C$_n$ & C$_n$ & B$_n$ & $\mathfrak{sp}_{2n}$ & $\Delta_\theta^{\mathrm{red}}\neq \Delta_{G/G^\theta}$ \\
  \midrule
  DI & $\mathfrak{so}_{2n}$ & $\mathfrak{so}_{2n}$ & $\mathfrak{so}(n,n)$  & $\mathfrak{so}_{n}\oplus\mathfrak{so}_{n}$ & D$_n$ & D$_n$ & D$_n$ & $\mathfrak{so}_{2n}$ & split \\
  \midrule
  DII & $\mathfrak{so}_{2n}$ & $\mathfrak{so}_{2n}$ & $\mathfrak{so}(p,q)$  & $\mathfrak{so}_{p}\oplus\mathfrak{so}_{q}$ & B$_p$ & B$_p$ & C$_p$ & $\mathfrak{so}_{2p+1}$ &  \begin{tabular}[c]{@{}c@{}}$p+q=2n$, $p<q$ \\  QS if $q=p+2$ \\
  $\Delta_\theta^{\mathrm{red}}\neq \Delta_{G/G^\theta}$\end{tabular} \\
  \midrule
  \multirow{2}{*}{DIII} & $\mathfrak{so}_{2n}$ & $\mathfrak{so}_{2n}$ & $\mathfrak{so}^*(2n)$  & $\mathfrak{gl}_{n}$ & BC$_{(n-1)/2}$ & B$_{(n-1)/2}$ & B$_{(n-1)/2}$ & $\mathfrak{sp}_{n-1}$ &  $n$ odd \\
\cmidrule(l){2-10}
   & $\mathfrak{so}_{2n}$ & $\mathfrak{so}_{2n}$ & $\mathfrak{so}^*(2n)$  & $\mathfrak{gl}_{n}$ & C$_{n/2}$ & C$_{n/2}$ & B$_{n/2}$ & $\mathfrak{sp}_{n}$ &  \begin{tabular}[c]{@{}c@{}} $n$ even \\ $\Delta_\theta^{\mathrm{red}}\neq \Delta_{G/G^\theta}$ \end{tabular}\\
  \midrule
  EI & $\mathfrak{e}_{6}$ & $\mathfrak{e}_{6}$ & $\mathfrak{e}_{6(6)}$  & $\mathfrak{sp}_{8}$ & E$_{6}$ & E$_{6}$ & E$_{6}$ & $\mathfrak{e}_{6}$ &  split \\
  \midrule
  EII & $\mathfrak{e}_{6}$ & $\mathfrak{e}_{6}$ & $\mathfrak{e}_{6(2)}$  & $\mathfrak{sl}_{6}\oplus \mathfrak{sl}_2$ & F$_{4}$ & F$_{4}$ & F$_{4}$ & $\mathfrak{f}_{4}$ &  quasisplit \\
  \midrule
  EIII & $\mathfrak{e}_{6}$ & $\mathfrak{e}_{6}$ & $\mathfrak{e}_{6(-14)}$  & $\mathfrak{so}_{11}\oplus \CC$ & BC$_{2}$ & B$_{2}$ & C$_{2}$ & $\mathfrak{so}_{5}$ &  $\Delta_\theta^{\mathrm{red}}\neq \Delta_{G/G^\theta}$ \\
  \midrule
  EIV & $\mathfrak{e}_{6}$ & $\mathfrak{e}_{6}$ & $\mathfrak{e}_{6(-26)}$  & $\mathfrak{f}_{4}$ & A$_{2}$ & A$_{2}$ & A$_{2}$ & $\mathfrak{sl}_{3}$ &   \\
  \midrule
  EV & $\mathfrak{e}_{7}$ & $\mathfrak{e}_{7}$ & $\mathfrak{e}_{7(7)}$  & $\mathfrak{sl}_{8}$ & E$_{7}$ & E$_{7}$ & E$_{7}$ & $\mathfrak{e}_{7}$ & split  \\
  \midrule
  EVI & $\mathfrak{e}_{7}$ & $\mathfrak{e}_{7}$ & $\mathfrak{e}_{7(-5)}$  & $\mathfrak{so}_{6}\oplus \mathfrak{sl}_2$ & F$_{4}$ & F$_{4}$ & F$_{4}$ & $\mathfrak{f}_{4}$ &    \\
  \midrule
  EVII & $\mathfrak{e}_{7}$ & $\mathfrak{e}_{7}$ & $\mathfrak{e}_{7(-25)}$  & $\mathfrak{e}_{6}\oplus \CC$ & C$_{3}$ & C$_{3}$ & B$_{3}$ & $\mathfrak{sp}_{6}$ &  $\Delta_\theta^{\mathrm{red}}\neq \Delta_{G/G^\theta}$  \\
  \midrule
  EVIII & $\mathfrak{e}_{8}$ & $\mathfrak{e}_{8}$ & $\mathfrak{e}_{8(8)}$  & $\mathfrak{so}_{16}$ & E$_{8}$ & E$_{8}$ & E$_{8}$ & $\mathfrak{e}_{8}$ &  split  \\
  \midrule
  EIX & $\mathfrak{e}_{8}$ & $\mathfrak{e}_{8}$ & $\mathfrak{e}_{8(-24)}$  & $\mathfrak{e}_{7}\oplus \mathfrak{sl}_2$ & F$_{4}$ & F$_{4}$ & F$_{4}$ & $\mathfrak{f}_{4}$ &   \\
  \midrule
  FI & $\mathfrak{f}_{4}$ & $\mathfrak{f}_{4}$ & $\mathfrak{f}_{4(4)}$  & $\mathfrak{sp}_{6}\oplus \mathfrak{sl}_2$ & F$_{4}$ & F$_{4}$ & F$_{4}$ & $\mathfrak{f}_{4}$ & split  \\
  \midrule
  FII & $\mathfrak{f}_{4}$ & $\mathfrak{f}_{4}$ & $\mathfrak{f}_{4(-20)}$  & $\mathfrak{so}_{9}$ & BC$_{1}$ & A$_{1}$ & A$_{1}$ & $\mathfrak{sl}_{2}$ &   \\
  \midrule
  G & $\mathfrak{g}_{2}$ & $\mathfrak{g}_{2}$ & $\mathfrak{g}_{2(2)}$  & $\mathfrak{sl}_{2}\oplus \mathfrak{sl}_2$ & G$_{2}$ & G$_{2}$ & G$_{2}$ & $\mathfrak{g}_{2}$ & split  \\ \bottomrule
\end{tabular}
\end{adjustbox}
\end{table}
Sakellaridis and Venkatesh \cite{sake-venka} describe a canonical way of associating a \emph{dual group} $\check{G}_\Sigma$ to any spherical $G$-variety $\Sigma$, conjecturally with a natural homomorphism $\check{G}_\Sigma\rightarrow \check{G}$ to the Langlands dual group $\check{G}$ of $G$. Knop and Schalke \cite{knop-schalke} later proved that, indeed, $\check{G}_\Sigma$ is contained in $\check{G}$ up to isogeny. Their construction consists of taking the weight lattice $\PP(\Sigma)$ on $\Sigma$ and a natural set of simple roots $\Delta_\Sigma$ associated to $\Sigma$, in such a way that $(\PP(\Sigma),\Delta_\Sigma,\PP(\Sigma)^\vee,\Delta_\Sigma^\vee)$ is a based root datum. The dual group $\check{G}_\Sigma$ is defined as the reductive group with based root datum equal to $(\PP(\Sigma)^\vee,\Delta_\Sigma^\vee,\PP(\Sigma),\Delta_\Sigma)$.
The simple roots $\Delta_\Sigma$ are obtained by taking the generators of intersections of the weight lattice with the negative of the dual of the \emph{central valuation cone}, and then applying some process of normalization, consisting in further intersecting the real span of each generator with the root lattice of $G$, as explained in \cites{sake-venka, knop-schalke}.

If $\theta \in \Aut_2(G)$ is an involution, there is a natural reduced root system associated to $\theta$, namely $$\Phi_\theta^{\mathrm{red}}=\left\{\alpha \in \Phi_\theta: \alpha/2 \not\in \Phi_\theta \right\},$$ the reduced version of the restricted root system obtained by taking the shortest roots. If $A$ is the corresponding maximal $\theta$-split torus, the resulting based root datum is $(\XX^*(A),\Delta_\theta^{\mathrm{red}},\XX_*(A),(\Delta_\theta^{\mathrm{red}})^\vee)$. When $k=\CC$, the group $\tilde{G}_\theta$ corresponding to this based root datum is the complexification of a \emph{maximal $\RR$-split subalgebra} $\tilde{\g}_\RR \subset \g_\RR$ inside the Lie algebra $\g_\RR$ of the real form of $G$ corresponding to $\theta$. See \cite{ana-oscar-ramanan}*{Section 2.2} for more details.

In most cases $\Delta_{G/G^\theta}=\Delta_\theta^{\mathrm{red}}$, but in some cases there are normalization factors involved, that change some roots by multiples of them by factors of $1/2$ or $2$. In specific cases, we can find that $\Delta_\theta^{\mathrm{red}}$ is of type $B_m$ or $C_m$, whereas $\Delta_{G/G^\theta}$ is of type $C_m$ or $B_m$, respectively. In Table \ref{tab:dual-groups}, we recollect the different involutions of the simple Lie algebras with their fixed points, their associated real forms and root systems, and the Lie algebra of the corresponding dual groups. The reader can compare Table \ref{tab:dual-groups} with Table 1 in \cite{ana-oscar-ramanan} and with Table 1 in \cite{nadler}. For us, the most relevant properties of the dual group $\check{G}_{G/G^\theta}$ are the following:
\begin{enumerate}
	\item $\check{G}_{G/G^\theta}$ contains $\check{A}_{G^\theta} = \Spec(k[e^{\XX_*(A_{G^\theta})}])$ as a maximal torus,
	\item the Weyl group of $\check{G}_{G/G^\theta}$ is the little Weyl group $W_\theta$,	
	\item the dominant Weyl chamber of $\check{G}_{G/G^\theta}$ coincides with $\XX_*(A_{G^\theta})_+$.
\end{enumerate}

The dual group $\check{G}_{G/G^\theta}$ of the symmetric variety $G/G^\theta$ was previously discovered by Nadler \cite{nadler} as a group $\check{G}_{G_\RR}$ naturally associated to the corresponding real form $G_\RR$, obtained as the Tannaka group of a certain neutral Tannakian category of perverse sheaves in the \emph{real loop Grassmannian} of $G_\RR$, thus giving a generalization of the geometric Satake correspondence. More generally, by a similar procedure, Gaitsgory and Nadler \cite{gaitsgory-nadler} associated a reductive group $\check{G}_{\Sigma,\mathrm{GN}}$ to any spherical $G$-variety $\Sigma$, with a natural inclusion $\check{G}_{\Sigma,\mathrm{GN}}\subset \check{G}$. The equality of $\check{G}_{\Sigma,\mathrm{GN}}$ and $\check{G}_{\Sigma}$ remains conjectural in the general case of a spherical variety \cite{BZSV}*{Page 75}.

\subsection{Very flat symmetric embeddings and the Guay embedding}
A $G$-variety $\Sigma$ is \emph{simple} if it has only one closed $G$-orbit. Let $\Sigma$ be a simple affine symmetric embedding. Recall from our previous discussions that the corresponding symmetric variety $O_\Sigma$ can be written in the form $O_\Sigma=G_Z/H_Z$, for $G_Z$ a reductive group  $G_Z=G\times Z$, with $G$ semisimple and $Z$ a torus and $G_Z/H_Z$ associated to the involution $\vartheta:(g,z)\mapsto (\theta(g),z^{-1})$ for $\theta\in \Aut_2(G)$. 

Consider the natural projections $\pr_1:G_Z\rightarrow G$ and $\pr_2:G_Z\rightarrow Z$. 
We can now define the tori
\begin{equation*}
A_\Sigma = O_\Sigma/G =Z/\pr_2(H_Z),
\end{equation*} 
and
\begin{equation*}
Z_\Sigma = Z/(\pr_2(H_Z)\cap Z_2), \text{ for } Z_2=\left\{z\in Z:z^2=1\right\}.
\end{equation*} 
Naturally, we have a projection $Z_\Sigma \rightarrow A_\Sigma$ and an inclusion $\XX^*(A_\Sigma)\rightarrow \XX^*(Z_\Sigma)$.

Let $A\subset G$ be a maximal $\theta$-split torus and $T\subset G$ a maximal torus containing it. Then, $A_Z=A\times Z$ is a maximal $\vartheta$-split torus of $G_Z$ and $T_Z$ a maximal $\vartheta$-stable torus containing it. Recall that we have another torus associated to $\Sigma$ 
\begin{equation*}
T_\Sigma = T_Z / (T_Z\cap H_Z) = A_Z / (A_Z \cap H_Z)
\end{equation*} 
such that $\PP(\Sigma)=\XX^*(T_\Sigma)$. We denoted $H=H_Z \cap (G\times \left\{1\right\})$. Let us put now $\tilde{H}=\pr_1(H_Z)$. Then,
\begin{equation*}
T_Z \cap H_Z \subset (T\cap \tilde{H}) \times \pr_2(H_Z),
\end{equation*} 
so we have natural projections
\begin{equation*}
	\pr_1:T_\Sigma \rightarrow A/(A\cap \tilde{H})=: A_{\tilde{H}} \text{ and } \pr_2:T_\Sigma \rightarrow Z/\pr_2(H_Z) = A_\Sigma.
\end{equation*} 
and inclusions $i_1:\XX^*(A_{\tilde{H}})\hookrightarrow \XX^*(T_\Sigma)$ and $i_2:\XX^*(A_{\Sigma}) \hookrightarrow \XX^*(T_\Sigma)$. Moreover, the kernel of the projection $\pr_2:T_\Sigma\rightarrow A_\Sigma$ is the torus
 \begin{equation*}
T_\Sigma \cap (G\times \left\{1\right\}) = T/(T\cap H) = A/(A\cap H) = A_H.
\end{equation*} 
Therefore, we obtain a natural projection $p:\XX^*(T_\Sigma)\rightarrow \XX^*(A_H)$.
Note that since $\Sigma$ is affine we also have a natural inclusion $\PP_+(\Sigma)\subset \PP_+(O_\Sigma)=\XX^*(T_\Sigma)_+.$ 

\begin{defn}
The GIT quotient $\Ab_\Sigma:=\Sigma \git G$ is an $A_\Sigma$-toric variety, which we call the \emph{abelianization} of $\Sigma$. The natural projection $\alpha_\Sigma:\Sigma\rightarrow \Ab_\Sigma$ is called the \emph{abelianization map} of $\Sigma$.	
\end{defn}

\begin{rmk}
The abelianization $\Ab_\Sigma$ is simply the toric variety
\begin{equation*}
\Ab_\Sigma = \Spec \left( \bigoplus_{\chi \in \PP_{+}(\Ab_\Sigma)} k[e^\chi]\right),
\end{equation*} 
where $\PP_{+}(\Ab_\Sigma)$ is the intersection
\begin{equation*}
\PP_{+}(\Ab_\Sigma)=\PP_+(\Sigma)\cap i_2(\XX^*(A_\Sigma)).
\end{equation*} 
\end{rmk}

\begin{defn}
A \emph{very flat symmetric embedding} is a simple affine symmetric embedding $\Sigma$ such that the abelianization map $\alpha_\Sigma:\Sigma \rightarrow \Ab_\Sigma$ is dominant, flat and with integral fibres.	
\end{defn}

We assume for the rest of the section that $G$ is simply-connected and fix $\theta\in \Aut_2(G)$. As in the previous section, we fix a maximal $\theta$-split torus $A\subset G$. We are interested in classifying very flat symmetric embeddings $\Sigma$ such that the semisimple part of $O_\Sigma$ is equal to $G/G^\theta$. These have been determined by Guay \cite{guay}*{Proposition 7}.

\begin{thm}[Guay] \label{thm:guayclassification}
A simple affine symmetric embedding $\Sigma$ with $O_\Sigma'=G/G^\theta$ is very flat if and only if there exists a homomorphism
\begin{align*}
\psi: \XX^*(A_{G^\theta}) & \longrightarrow \XX^*(Z_\Sigma), 
\end{align*} 
such that, for any $(a,z)\in H_Z \cap A_Z\subset (\tilde{H}\cap A) \times \pr_2(H_Z)$ and for any $\chi \in \XX^*(A_{G^\theta})$, we have
\begin{equation*}
a^\chi = z^{-\psi(\chi)};
\end{equation*} 
and such that the weight semigroup $\PP_+(\Sigma)$ is of the form
\begin{equation*}
\PP_+(\Sigma)=\left\{(\chi,\psi(\chi)+\eta):\chi \in \XX^*(A_{G^\theta})_+, \eta \in \PP_{+}(\Ab_\Sigma)\right\}.
\end{equation*} 
\end{thm}

\begin{rmk} \label{coordenv}
When $\Sigma$ is very flat any function $f\in k[G/G^\theta]$ with weight $\chi$ can be extended	to a function $f_+\in k[\Sigma]$ with weight $(\chi,\psi(\chi))$. Indeed, just define
\begin{equation*}
f_+(g,z)=z^{\psi(\chi)}f(g).
\end{equation*} 
It is clear that $f_+$ has the desired weight. Now, since $\tilde{H}$ is a symmetric subgroup for $\theta$, we have that any element $h\in \tilde{H}$ is of the form $h=ah_0$ for $h_0\in G^\theta$ and $a\in A\cap \tilde{H}$. Thus, if $(h,s)\in H_Z$, we have, for any $(g,z)\in G_Z$,
\begin{equation*}
f_+(hg,zs)=s^{\psi(\chi)}z^{\psi(\chi)}f(hg)=s^{\psi(\chi)}a^\chi f_+(g,z)=f_+(g,z).
\end{equation*} 

This allows us to describe $\Sigma$ more explicitly. For each fundamental dominant weight $\varpi_i \in \XX^*(A_{G^\theta})$, consider the $G$-submodule $k[G/G^\theta]_i$ of $k[G/G^\theta]$ formed by functions with weight $\varpi_i$, and take $f_i^1,\dots,f_i^{m_i}$ a basis of $k[G/G^\theta]_i$ as a $k$-vector space. On the other hand, let $\gamma_1,\dots,\gamma_s$ be generators of $\PP_+(\Ab_\Sigma)\subset \XX^*(Z)$. We can now define the map
\begin{align*}
	(f_+,\alpha_+):O_\Sigma & \longrightarrow \left( \bigoplus_{i=1}^l \Aa^{m_i} \right) \oplus \Aa^s\\
	(g,z) H_Z & \longmapsto ((f_{i+}^1(g,z),\dots,f_{i+}^{m_i}(g,z)),(z^{\gamma_1},\dots,z^{\gamma_s})).
\end{align*} 
The symmetric embedding $\Sigma$ is then identified with the closure of the image of this map, $\Sigma=\overline{(f_+,\alpha_+)(O_\Sigma)}$.
\end{rmk}

Amongst all simple affine symmetric embeddings, those having a fixed point will be of particular interest. By a fixed point of an $\Sigma$ of a symmetric variety $G/H$ we mean a point $0\in \Sigma$ such that $g\cdot 0 = 0$ for any $g\in G$. Since by assumption $\Sigma$ is simple, the set $\left\{0\right\}\subset \Sigma$ is its unique closed orbit and thus a fixed point, if it exists, is unique. From Guay's classification we can obtain necessary and sufficient conditions for the existence of a fixed point.

\begin{prop}
A very flat symmetric embedding $\Sigma$ with $O_\Sigma'=G/G^\theta$ has a fixed point if and only if $\PP_+(\Ab_\Sigma)$ is pointed (meaning that if $\chi$ and $-\chi$ are in $\PP_+(\Ab_\Sigma)$, then $\chi=0$) and the only element in $\PP_+(\Sigma)$ of the form $(\chi,\psi(\chi))$ for $\chi\in \XX^*(A_{G^\theta})_+$ is $(0,0)$.
\end{prop}

\begin{proof}
A fixed point $\Sigma$ corresponds to a maximal ideal $I\subset k[\Sigma]$ which is fixed under the action of $G_Z$. If such an ideal exists, it must be of the form
 \begin{equation*}
I=\bigoplus_{(\chi,\psi(\chi)+\eta) \in \PP_+(\Sigma)\setminus \left\{(0,0)\right\}} k[O_\Sigma]_{(\chi,\psi(\chi)+\eta)},
\end{equation*} 
and this is only an ideal if for every $(\chi,\psi(\chi)+\eta)\in \PP_+(\Sigma)\setminus \left\{(0,0)\right\}$, $(0,0)$ is not a highest weight of $k[O_\Sigma]_{(\chi,\psi(\chi)+\eta)}$. That is the case and $(0,0)$ is a highest weight of $k[O_\Sigma]_{(\chi,\psi(\chi)+\eta)}$ if and only if both $\eta$ and $-\eta$ are contained in $\PP_+(\Ab_\Sigma)$. 
\end{proof}

By the universal property of categorical quotients, a morphism $\Sigma_1\rightarrow \Sigma_2$ of simple affine symmetric embeddings induces a commutative square
\begin{center}
\begin{tikzcd}
\Sigma_1 \ar{r} \ar{d} & \Sigma_2 \ar{d} \\
\Ab_1 \ar{r} & \Ab_2.
\end{tikzcd}
\end{center}

\begin{defn}
A morphism $\Sigma_1\rightarrow \Sigma_2$ of simple affine symmetric embeddings is \emph{excellent} if the induced square is Cartesian.	

We denote by $\mathcal{VF}(G/G^\theta)$ the category with very flat symmetric embeddings with semisimple part $G/G^\theta$ as objects and excellent morphisms as arrows. We can also consider the subcategory $\mathcal{VF}_0(G/G^\theta)$ formed by very flat symmetric embeddings with a fixed point.
\end{defn}

Suppose that $\theta$ does not have imaginary roots (that is $\Phi_T^\theta=\varnothing$, for $T\subset G$ a $\theta$-stable maximal torus). For example, we can assume that $\theta$ is quasisplit. In that case, a versal object of $\mathcal{VF}(G/G^\theta)$ was constructed by Guay \cite{guay}. We recall here his construction.

We begin by taking $A\subset T \subset G$ a maximal $\theta$-split torus of $G$ and putting $G_A=G\times A$ and
\begin{equation*}
H_A=\left\{(nh,an^{-1}): h\in G^\theta, a \in F^\theta, n\in F_\theta\right\},
\end{equation*} 
where we recall that we denoted $F^\theta = G^\theta \cap A$ and  $F_\theta=G_\theta \cap A$. The space $(G/G^\theta)_+:=G_A/H_A$ is a symmetric variety associated to the involution $\vartheta:(g,a)\mapsto (\theta(g),a^{-1})$ with semisimple part isomorphic to $G/G^\theta$. Indeed, we have
 \begin{equation*}
H=H_A\cap (G\times \left\{1\right\}) = G^\theta, \ \ \text{ and } \ \ \tilde{H}=\pr_1(H_A)=G_\theta.
\end{equation*} 
A maximal $\vartheta$-split torus of $G_A$ is given by $A\times A$.

 \begin{defn}
We define the (Guay) \emph{enveloping embedding} of $G/G^\theta$ as the symmetric embedding $\Env(G/G^\theta)$ with $O_{\Env(G/G^\theta)}=(G/G^\theta)_+$ and determined by the weight semigroup
\begin{equation*}
\PP_+(\Env(G/G^\theta)) = \left\{(\chi,w_0\chi+\eta): \chi \in \XX^*(A_{G^\theta})_+, \eta\in -\ZZ_+\langle \Delta_\theta \rangle \right\} \cup \left\{(0,0)\right\},
\end{equation*} 
for $\Delta_\theta=\left\{\bar{\alpha}_1,\dots,\bar{\alpha}_l\right\}$ the simple restricted roots associated to $\theta$.
\end{defn}

Note that $Z_{\Env(G/G^\theta)}=A/F^\theta=A_{G^\theta}$ and $A_{\Env(G/G^\theta)}=A/F_\theta=A_{G_\theta}$. Now, we can define
\begin{align*}
\psi:\XX^*(A_{G^\theta}) & \longrightarrow  \XX^*(A_{G^\theta})\\
\chi & \longmapsto w_0\chi.
\end{align*} 
An element of $H_A\cap (A\times A)$ is of the form $(na_1,a_2n^{-1})$ for $a_1,a_2\in F^\theta$ and  $n\in F_\theta$. Therefore,
 \begin{equation*}
	 (na_1)^{-w_0\chi}=n^{-w_0\chi}=(n^{-1})^{w_0\chi} = (a_2n^{-1})^\chi.
\end{equation*} 
We conclude that $\Env(G/G^\theta)$ is a very flat symmetric embedding. Moreover, one checks easily that it has a fixed point $0\in \Env(G/G^\theta)$. 

The main result of Guay's paper is the following \cite{guay}*{Theorem 3}.
\begin{thm}[Guay]
When $\Phi_T^\theta=\varnothing$, the enveloping embedding  $\Env(G/G^\theta)$ is a versal object of the category $\mathcal{VF}(G/G^\theta)$ and a universal object of $\mathcal{VF}_0(G/G^\theta)$. That is, for any very flat symmetric embedding $\Sigma$ with semisimple part $G/G^\theta$ there exists an excellent morphism $\Sigma\rightarrow \Env(G/G^\theta)$, which is unique if $\Sigma$ has a fixed point.
\end{thm}

\subsection{The wonderful compactification}
The geometry of the Guay embedding can be better understood by means of the \emph{wonderful compactification} $\overline{G/G_\theta}$ of $G/G_\theta$. A \emph{wonderful} $G$-variety $\Sigma$ is, by definition, a $G$-equivariant dense open embedding of some homogeneous space $G/H$ such that
\begin{enumerate}
	\item $\Sigma$ is smooth and projective,
	\item $\Sigma\setminus (G/H)$ is a divisor with normal crossings (i.e., its components $D_1,\dots,D_l$ are smooth and intersect transversally), and
	\item the closures of the $G$-orbits of $\Sigma$ are given by the intersections $D_{i_1}\cap D_{i_2}\cap \dots \cap D_{i_k}$, for $1\leq i_1\leq i_2 \leq \cdots \leq i_k\leq l$.
\end{enumerate}
It follows from the theory of spherical varieties that, given a spherical homogeneous space $G/H$, a wonderful $G$-variety $\Sigma$ with a $G$-equivariant dense open embedding $G/H\hookrightarrow \Sigma$, if it exists, it is unique. When it does exist, it is called the \emph{wonderful compactification} of $G/H$.

A wonderful compactification does not exist for a general symmetric variety $G/H$, but it does always exist for $G/G_\theta$, and we denote it by $\overline{G/G_\theta}$. This is a celebrated result of De Concini and Procesi \cite{deconcini-procesi}. An important ingredient of their proof is the \emph{local structure theorem} \cite{deconcini-procesi}*{2.3}, which implies that every element of $\overline{G/G_\theta}$ is in the $G$-orbit of an element of a toric variety $\Ab$, defined as
\begin{equation*}
\Ab = \Spec \left( \bigoplus_{\bar{\alpha} \in \Delta_\theta} k[e^{-2\bar{\alpha}}] \right),
\end{equation*} 
which can be identified with the $l$-dimensional affine space $\Aa^l$ in such a way that the embedding $A\hookrightarrow \Ab$ is given explicitly as $a\mapsto (a^{-2\bar{\alpha}_1},\dots,a^{-2\bar{\alpha}_l})$.  

When $G$ is semisimple simply-connected the Guay embedding $\Env(G/G^\theta)$ can be obtained from the wonderful compactification $\overline{G/G_\theta}$ by means of a construction introduced by Brion \cite{brion}. This is a procedure that can be done for any wonderful variety $\Sigma$. A very important property of wonderful varieties is that their Picard group is generated by a finite set $\mathcal{D}$ of prime divisors called its \emph{colors}. That is, $\text{Pic}(\Sigma)\cong \mathbb{Z}^{\mathcal{D}}$. Now, we can consider the \emph{Cox ring} of $\Sigma$
 \begin{equation*}
\text{Cox}(\Sigma) = \bigoplus_{(n_D) \in \mathbb{Z}^{D}} H^0(\Sigma, \Oo_\Sigma(\sum_{D\in \mathcal{D}}n_D D)).
\end{equation*} 
The \emph{Brion-Cox variety} of $\Sigma$ is the spectrum $\text{BC}(\Sigma)=\Spec(\text{Cox}(\Sigma))$.

\begin{prop}[Guay]
The Brion-Cox variety $\mathrm{BC}(\overline{G/G_\theta})$ is isomorphic to the Guay embedding $\Env(G/G^\theta)$.	
\end{prop}

The proposition above is the content of Section 4 in \cite{guay}, and the relation to Cox rings has been previously remarked in \cite{ADHL}. In particular, $\text{BC}(\overline{G^{\text{ad}}})$ is isomorphic to the Vinberg monoid $\Env(G)$, as was originally noticed by Brion \cite{brion}. A consequence of the above result is that there exists a dense open smooth subvariety $\Env^0(G/G^\theta)$ such that the GIT quotient $\Env^0(G/G^\theta)\git A_{G^\theta}$ is isomorphic to the wonderful compactification $\overline{G/G_\theta}$. By means of the construction explained in Remark \ref{coordenv}, $\Env^0(G/G^\theta)$ can also be obtained as the closure of the intersection of $(f_+.\alpha_+)((G/G^\theta)_+)$ with $(\bigoplus_{i=1}^l \mathbb{A}^{m_i}\setminus \left\{0\right\})\oplus \mathbb{A}^s$.

\begin{ej}
Let $G=\SL_{n+1}$, for some $n\geq 1$, and $\theta=\theta_0$ the involution  $\theta_0(A)=(A^T)^{-1}$. The fixed point subgroup is $G^{\theta}=\mathrm{SO}_{n+1}$, while
\begin{equation*}
G_\theta= \left\{A \in \SL_{n+1}: A=\lambda (A^T)^{-1} \text{ for some }\lambda \in k, \lambda^{n+1}=1\right\}.
\end{equation*} 
This group $G_\theta$ is the stabilizer of the non-degenerate quadric of $\mathbb{P}^n$ with equation $-X_0+X_1+\dots+X_n=0$ under the natural action of  $\SL_{n+1}$ by coordinate change. Therefore, the symmetric variety $G/G_\theta$ parametrizes non-degenerate quadrics in $\mathbb{P}^n$. The space of all quadrics in $\mathbb{P}^n$ is the projectivization of the space of degree $2$ homogeneous polynomials in $n+1$ variables, that is
$\mathbb{P}(k[X_0,\dots,X_n]_2) = \mathbb{P}(\mathrm{Sym}^2 (k^{n+1}))$,
a projective space of dimension ${n+2 \choose 2} -1$. This space can be regarded as a subvariety of a bigger projective space through the \emph{Veronese embedding}
\begin{align*}
 \mathbb{P}(\mathrm{Sym}^2 (k^{n+1}))& \longrightarrow \mathbb{P}(\mathrm{Sym}^2 \wedge^2 (k^{n+1})) \times \cdots \times \mathbb{P}(\mathrm{Sym}^2 \wedge^n (k^{n+1})) \\
Z & \longmapsto (\wedge^2 Z, \dots, \wedge^n Z).
\end{align*} 
The wonderful compactification $\overline{G/G_\theta}$ is isomorphic to the closure of the image of $G/G_\theta \hookrightarrow \mathbb{P}(\mathrm{Sym}^2 (k^{n+1}))$ through this map. It can be constructed explicitly by performing a series of blowing ups, as explained in \cite{corniani-massarenti}.

For $n=1$, we have that $G/G_\theta$ parametrizes pairs of points in  $\mathbb{P}^1$, the wonderful compactification $\overline{G/G_\theta}$ is equal to $\mathbb{P}^2$, and thus the Guay embedding is just
\begin{equation*}
\Env(\SL_2/\mathrm{SO}_2) = \Spec(\bigoplus_{m\in \mathbb{Z}} H^0(\mathbb{P}^2,\Oo_{\mathbb{P}^2}(m)))=\Spec(k[X_0,X_1,X_2])=\mathbb{A}^3.
\end{equation*} 
Of course, we recover the wonderful compactification as $\mathbb{P}^2=(\mathbb{A}^3\setminus \left\{0\right\})/\mathbb{G}_m$.

For $n=2$, we have that  $G/G_\theta$ parametrizes non-degenerate conics in  $\mathbb{P}^2$, and the wonderful compactification $\overline{G/G_\theta}$ is the blow up of $\mathbb{P}^5$ along the Veronese surface. Clearly its Picard group is generated by two elements, the one coming from $\mathbb{P}^5$ and the exceptional divisor of the blow up. We would finally obtain $\Env(\SL_3/\mathrm{SO}_3)$ by computing the Cox ring of this variety.
\end{ej}

\subsection{The invariant theory of a symmetric embedding} Let $G$ be a reductive group and $\theta\in \Aut_2(G)$ an involution. The action of $G$ on $G/G^\theta$ by left multiplication restricts to an action of $G^\theta$. This action was widely studied by Richardson \cite{richardson}, and we state here his main result.

\begin{thm}[\cite{richardson}*{Corollary 11.5}]
Let $A\subset G$ be a maximal $\theta$-split torus and $W_\theta$ the little Weyl group. The restriction homomorphism $k[G/G^\theta]\rightarrow k[A_{G^\theta}]$ induces an isomorphism
\begin{equation*}
k[G/G^\theta]^{G^\theta} \overset{\sim}{\longrightarrow} k[A_{G^\theta}]^{W_\theta}.
\end{equation*} 
\end{thm}

We denote
\begin{equation*}
\Cc_{G/G^\theta}=(G/G^\theta) \git G^\theta \cong A_{G^\theta}/W_\theta.
\end{equation*} 
Richardson's result can be extended easily to a symmetric embedding. Indeed, suppose that $G$ is semisimple, $\theta\in \Aut_2(G)$ and let $\Sigma$ be a simple affine symmetric embedding with $O_\Sigma'$ isomorphic to $G/G^\theta$. The left action of $G^\theta$ extends to an action on $\Sigma$. Let us denote $\Cc_\Sigma=\Sigma \git G^\theta$.

\begin{prop} \label{invariant}
If $\Sigma$ is very flat, there is an isomorphism
\begin{equation*}
\Cc_\Sigma \cong \Cc_{G/G^\theta} \times \Ab_{\Sigma}.
\end{equation*} 
\end{prop}
\begin{proof}
Since every element of $\PP_+(\Sigma)$ is of the form $(\chi,  \psi(\chi) + \eta)$ for $\chi\in \XX^*(A_{G^\theta})_+$ and $\eta\in \PP_+(\Ab_\Sigma)$, 
\begin{equation*}
k[\Sigma]^{G^\theta} = \bigoplus_{(\chi,\psi(\chi)+\eta) \in \PP_+(\Sigma)} k[O_\Sigma]_{(\chi,\psi(\chi)+\eta)}^{G^\theta}= k[G/G^\theta]^{G^\theta} \otimes k[\Ab_\Sigma],
\end{equation*} 
as we wanted to show.
\end{proof}

It follows from the discussion in \cite{richardson}*{Sections 13 and 14} that when $G$ is semisimple simply-connected, the ring $k[A_{G^\theta}]^{W_\theta}$ is a polynomial algebra and thus $A_{G^\theta}\git W_\theta$ is an affine space. Moreover, since $k[A_{G^\theta}]=k[e^{\XX^*(A_{G^\theta})}]$ and the lattice $\XX^*(A_{G^\theta})$ is generated by $\varpi_1,\dots,\varpi_l$, for the $\varpi_i$ defined in Section \ref{roots}, we get an isomorphism
\begin{equation*}
k[G/G^\theta]^{G^\theta} \cong k[b_1,\dots, b_l],
\end{equation*} 
where each $b_i$ is a function in $k[G/G^\theta]$ with weight $\varpi_i$. In general, for $G$ semisimple not simply-connected, Richardson \cite{richardson}*{Section 15} characterizes which involutions satisfy that $k[A_{G^\theta}]^{W_\theta}$ is a polynomial algebra.

\subsection{Loop parametrization} \label{section_loops} Let $\OO=k[[z]]$ denote the ring of formal power series in a formal variable $z$ and $F=\mathrm{qf}(\OO)=k((z))$ the field of formal Laurent series. For any variety $\Sigma$, we let $\Sigma(F)=\mathrm{Maps}_k(\Spec F,\Sigma)$ be the set of \emph{formal loops} and $\Sigma(\OO)=\mathrm{Maps}_k(\Spec \OO,\Sigma)$ the set of \emph{positive formal loops}. Both these sets can be regarded as the spaces of $k$-points of the functors $\Sigma_F$ and $\Sigma_\OO$ sending any $k$-algebra $R$ to $\mathrm{Maps}_k(\Spec(R\otimes_k F),\Sigma)$ and $\mathrm{Maps}_k(\Spec(R\otimes_k \OO),\Sigma)$, respectively. These functors can be endowed with the structure of an \emph{ind-scheme}. If $G$ is an algebraic group, the space $G(F)$ is called the \emph{formal loop group} and $G(\OO)$ is the \emph{formal positive loop group}. The homogeneous space $\mathrm{Gr}_G=G(F)/G(\OO)$ is known as the \emph{affine Grassmannian of} $G$.

If $G$ is a reductive group and $T\subset G$ is a maximal torus, for any cocharacter $\lambda\in \XX_*(T)$ we can obtain an element $z^\lambda \in G(F)$ as the image of the formal variable $z$ under the induced morphism
\begin{equation*}
\lambda(F): F^\times = \GG_m(F) \rightarrow T(F)\subset G(F).
\end{equation*} 
The Cartan decomposition of the loop group states that
\begin{equation*}
G(F)=\bigsqcup_{\lambda \in \XX_*(T)_+} G(\OO)z^\lambda G(\OO).
\end{equation*} 
Moreover, the closure of any of these orbits $G(\OO)z^\lambda G(\OO)$ is equal to
\begin{equation*}
\overline{G(\OO)z^\lambda G(\OO)} = \bigsqcup_{\mu \in \XX_*(T)_+, \mu\leq \lambda} G(\OO)z^\mu G(\OO).
\end{equation*} 
See \cite{zhu} for references on these facts.

The results stated above for reductive groups can be easily generalized to symmetric varieties. We thus begin by taking an involution $\theta\in \Aut_2(G)$, a maximal $\theta$-split torus $A\subset G$, a maximal torus $T\subset G$ containing $A$, and a Borel subgroup $B\subset G$ contained in a minimal $\theta$-split parabolic subgroup and containing $T$. The following is well known.

\begin{prop} \label{prop:loopdecomp}
Given any symmetric subgroup $H\subset G$ associated to $\theta$, we can decompose
\begin{equation*}
	(G/H)(F) = \bigsqcup_{\lambda \in \XX_*(A_H)_-} G(\OO) z^\lambda.
\end{equation*} 
\end{prop}

\begin{rmk}
Note that $z^\lambda$ is a well defined element of $(G/H)(F)$ since there	is a natural inclusion $A_H=T/(T\cap H)\hookrightarrow G/H$.
\end{rmk}

\begin{proof}
	For completion, we recall a proof given by Nadler \cite{nadler_matsuki}.
We begin by reducing to the case $H=G_\theta$. Thus, suppose that there is an element of $z^{\XX_*(A_{G_\theta})_-}$ in the orbit of every element of $(G/G_\theta)(F)$.
Since, by definition of a symmetric subgroup, we have $H\subset G_\theta$, there is a natural projection $G/H\rightarrow G/G_\theta$ and thus $(G/H)(F)\rightarrow (G/G_\theta)(F)$. Moreover, since the inverse image of $A_{G_\theta}$ under the projection $G/H\rightarrow G/G_\theta$ is  $A_H$, the inverse image of $z^{\XX_*(A_{G_\theta})_-}$ is $z^{\XX_*(A_H)_-}$. Since the projection is $G(\OO)$-equivariant, we conclude that there is an element of $z^{\XX_*(A_H)_-}$ in the orbit of every element of $(G/H)(F)$. Uniqueness follows from the fact that anti-dominant weights of $A_H$ are anti-dominant for $T$ and from the Cartan decomposition of $G(F)$.

It remains to show that for every $\phi\in (G/G_\theta)(F)$ there exists some $\lambda\in \XX_*(A_{G_\theta})_-$ such that $z^\lambda$ is in the $G(\OO)$-orbit of $\phi$. We consider now the wonderful compactification $\overline{G/G_\theta}$. Since $\overline{G/G_\theta}$ is projective, by the valuative criterion of properness, every formal loop $\phi\in (G/G_\theta)(F)$ extends to a formal arc  $\bar{\phi}\in \overline{G/G_\theta}(\OO)$. By the local structure theorem, there exist some formal arcs $g\in G(\OO)$ and $\bar{a}\in \Ab(\OO)$ such that
\begin{equation*}
\bar{\phi} = g \bar{a}.
\end{equation*} 
A cocharacter $\lambda: \GG_m \rightarrow A$ extends to a morphism $\Aa^1 \rightarrow \Ab$ if and only if  $\langle \lambda, 2\bar{\alpha}_i \rangle \leq 0$ for every $i=1,\dots,l$, that is, if  $\lambda\in \PP_{\theta,-}^\vee= \XX_*(A_{G_\theta})_-$.
Therefore, $\bar{a}=az^{\lambda}$ for some $\lambda \in \XX_*(A_{G_\theta})_-$.
\end{proof}

\begin{rmk}
Let $H\subset G$ be a symmetric subgroup associated to $\theta$ and $f\in k[G/H]_\chi$ a function with weight $\chi \in \XX^*(A_H)_+$. Recall that for any dominant weight $\chi$,  if $V_\chi$ is the irreducible representation with highest weight $\chi$, we have an isomorphism $V_\chi^*=V_{-w_0\chi}$, for $w_0$ the longest element of the Weyl group, so $f(gt)=t^{-w_0\chi}f(g)$ for any $g\in G$ and  $t\in T$. Therefore, if $\phi=gz^{\lambda}\in G(\OO)z^\lambda$, we have
\begin{equation*}
f(\phi)=(z^\lambda)^{-w_0\chi} f(g) = z^{-\langle \lambda, w_0\chi \rangle} f(g).
\end{equation*} 
Thus, $f(\phi)$ is a Laurent series with highest pole order less or equal than $\langle \lambda,w_0\chi \rangle$.
\end{rmk}

Let us describe now the closures of the orbits $G(\OO)z^\lambda$. Recall that we can define an order in $\XX_*(A_{H})_-$ by putting $\lambda \leq \mu$ if and only if $\lambda-\mu\in -\mathbb{N}\langle \Delta_\theta^\vee\rangle$.

\begin{prop}
For any $\lambda \in \XX_*(A_{H})_-$, the closure of $G(\OO)z^\lambda$ is equal to
\begin{equation*}
\overline{G(\OO)z^\lambda} = \bigsqcup_{\mu \in \XX_*(A_H)_-,\mu\leq \lambda} G(\OO)z^\mu.
\end{equation*} 
\end{prop}

\begin{proof}
Let $f\in k[G/H]_\chi$ be a function with weight $\chi\in \XX^*(A_H)_+$. If $\phi\in G(\OO)z^\mu$ is such that $\phi$ belongs to the closure of $G(\OO)z^\lambda$ then the highest pole order of $f(\phi)$ must be less or equal than $\langle \lambda, w_0\chi\rangle$, so	
\begin{equation*}
\langle \mu, w_0 \chi \rangle \leq \langle \lambda, w_0\chi \rangle \implies \langle \lambda - \mu, w_0\chi \rangle \geq 0.
\end{equation*} 
Now, since $w_0\XX^*(A_H)_+=\XX^*(A_H)_-$, we conclude that $\lambda-\mu\in -\mathbb{N}\langle \Delta_\theta^\vee \rangle$.

Reciprocally, we want to show that for every $\mu\leq \lambda$, $G(\OO)z^\mu \in \overline{G(\OO)z^\lambda}$. It suffices to find an element $\phi\in G(\OO)z^\mu$ with $\phi \in \overline{G(\OO)z^\lambda}$. The argument is analogous to the one given for the affine Grassmannian (see \cite{zhu}). We do it for $G=\mathrm{SL}_2$ and the general result follows by considering the canonical homomorphism $\mathrm{SL}_2\rightarrow G$ associated to the root $\bar{\alpha}$ with $\lambda-\mu = \bar{\alpha}^\vee$. Given any $m\in \mathbb{N}$ we can consider the family
 \begin{equation*}
\begin{pmatrix}
	z^m & 0 \\
	z^{-m}+t^{-1}z^{-m+1} & z^{-m}
\end{pmatrix}
\in \SL_2(\OO) z^{-m} ,
\end{equation*} 
for $t\in k^\times$.
One can easily check that
\begin{equation*}
\begin{pmatrix}
	t^{-1} & 0 \\	
	0 & t
\end{pmatrix}
\begin{pmatrix}
	1 & -t z^{2m-1}\\
	0 & 1
\end{pmatrix}
\begin{pmatrix}
	z^m & 0 \\
	z^{-m}+t^{-1}z^{-m+1} & z^{-m}
\end{pmatrix}
=
\begin{pmatrix}
	- z^{m-1} & -z^{m-1} \\
	tz^{-m} + z^{-m + 1} & tz^{-m}
\end{pmatrix},
\end{equation*} 
which lies in the same orbit $\SL_2(\OO)z^{-m}$ as the original matrix. Now, the limit as $t\rightarrow 0$ clearly has to lie in the closure $\overline{\SL_2(\OO)z^{-m}}$, but this limit is
\begin{equation*}
\begin{pmatrix}
	-z^{m-1} & -z^{m-1} \\	
	z^{-m+1} & 0
\end{pmatrix}
=
\begin{pmatrix}
	-z^{2m} & -z^{2m} \\
	1 & 0
\end{pmatrix}
z^{-m+1}
\in \SL_2(\OO)z^{-(m-1)}
\end{equation*} 
We conclude that $\SL_2(\OO)z^{-(m-1)}\subset \overline{\SL_2(\OO)z^{-m}}$
\end{proof}

Now, let $G$ be semisimple simply-connected and $\theta\in \Aut_2(G)$ an involution. We finish the section by describing the loop space of the enveloping embedding $\Env(G/G^\theta)$; this is a generalization of \cite{chi}*{2.5}. We begin by noticing, as in the proof of Proposition \ref{prop:loopdecomp} that if $\Ab$ is an $A$-toric variety with weight semigroup $\PP_+(\Ab)\subset A$, then an element $a\in A(F)$ extends to an element $\bar{a}\in \Ab(\OO)$ if and only if $a\in A(\OO)z^\lambda$, for $\lambda$ in the dual semigroup
\begin{equation*}
\PP_+(\Ab)^\vee = \left\{\lambda \in \XX^*(A): \langle \lambda,\chi\rangle \in \mathbb{N}, \forall \chi \in \PP_+(\Ab)\right\}.
\end{equation*} 
In other words,
\begin{equation*}
A(F) \cap \Ab(\OO) = \bigsqcup_{\lambda \in \PP_+(\Ab)^\vee} A(\OO)z^\lambda.
\end{equation*} 
In particular, we get
\begin{equation*}
A_{G_\theta}(F)\cap \Ab_{\Env(G/G^\theta)}(\OO)=\bigsqcup_{\lambda \in \XX^*(A_{G_\theta})_-} A_{G_\theta}(\OO)z^\lambda,
\end{equation*} 
since $\XX^*(A_{G_\theta})_-=(-\mathbb{N}\langle \Delta_\theta \rangle)^\vee$.

Consider an anti-dominant cocharacter $\lambda \in \XX_*(A_{G_\theta})_-$ and define $\Env^\lambda(G/G^\theta)$ as the fibered product
\begin{center}
\begin{tikzcd}
\Env^\lambda(G/G^\theta) \ar{r} \ar{d} &A_{G_\theta} \times \Env(G/G^\theta) \ar{d} \\
\Spec \OO \ar{r}{z^{-w_0\lambda}} & \Ab_{\Env(G/G^\theta)}, 
\end{tikzcd}
\end{center}
where the vertical map on the right is the multiplication of the abelianization map $\alpha_{\Env(G/G^\theta)}$ with the natural embedding $A_{G_\theta}\hookrightarrow \Ab_{\Env(G/G^\theta)}$. Replacing $\Env(G/G^\theta)$ by the open subvariety $\Env^0(G/G^\theta)$ we define an open subvariety $\Env^{\lambda,0}(G/G^\theta)$. The above stratification induces
\begin{equation*}
\Env(G/G^\theta) (\OO) \cap (G/G^\theta)_+(F) = \bigsqcup_{\lambda \in \XX^*(A_{G_\theta})_-} \Env^\lambda(G/G^\theta) (\OO). 
\end{equation*} 
We also note that
\begin{equation*}
\Env^{\lambda,0}(G/G^\theta)(\OO) = \Env^{\lambda}(G/G^\theta)(\OO) \cap \Env^{0}(G/G^\theta)(\OO)
\end{equation*} 

\begin{prop}\label{prop:extender_loops}
For any $\phi \in (G/G^\theta)_+(F)$, we have $\phi \in \Env^\lambda(G/G^\theta)(\OO)$ if and only if the image of $\phi$ in $(G/G_\theta)(F)$ belongs to $\overline{G(\OO)z^\lambda}$. Moreover, $\phi\in \Env^{\lambda,0}(G/G^\theta)(\OO)$ if and only if the image of $\phi$ in $(G/G_\theta)(F)$ belongs to $G(\OO)z^\lambda$.
\end{prop}

\begin{proof}
It suffices to show it for $\phi=(z^\mu,z^\eta)$ with $\mu,\eta\in \XX_*(A_{G_\theta})_-$. First, note that by construction we must have $z^{\eta}=z^{-w_0\lambda},$ so we get $\eta=-w_0\lambda$. On the other hand, $\phi \in \Env(G/G^\theta)(\OO)$ if and only if $f_+ (\phi)\in \OO$ for any $f\in k[G/G^\theta]$. Thus, for any $\chi \in \XX^*(A_{G^\theta})_+$ and for any $f\in k[G/G^\theta]_\chi$ we have
\begin{equation*}
\OO \ni	f_+(z^\mu,z^\eta) = z^{\langle w_0\chi,\eta \rangle} z^{\langle \chi, \mu \rangle}.
\end{equation*} 
Therefore, 
\begin{equation*}
0 \leq \langle w_0 \chi,\eta \rangle + \langle \chi, \mu \rangle = \langle w_0 \chi, -w_0 \lambda \rangle + \langle \chi,\mu \rangle = \langle \chi, \mu-\lambda \rangle.
\end{equation*} 
We have this for any $\chi \in \XX^*(A_{G^\theta})_+$, so $\mu-\lambda \in \mathbb{N}(\Delta_\theta^\vee)$ and $\lambda\geq \mu$. Finally, $\phi\in \Env^0(G/G^\theta)(\OO)$ if and only if $f_+(\phi)\in \OO$ does not vanish, thus, if and only if $0\geq \langle \chi, \mu-\lambda \rangle$, so $\lambda=\mu$.
\end{proof}

\section{Multiplicative Higgs bundles for symmetric varieties} \label{hitchin}
\subsection{The multiplicative Hitchin map for symmetric embeddings} \label{higgssymmetric}
We begin by taking a semisimple simply-connected group $G$ over $k$ and an involution $\theta\in \Aut_2(G)$. Let $\Sigma$ be a very flat symmetric embedding with $O_\Sigma'=G/G^\theta$.

The left multiplication action of $G^\theta$ on $G/G^\theta$ extends to an action of $G^\theta$ on $\Sigma$. Moreover, if $O_\Sigma$ is a symmetric variety the form  $G_Z/H_Z$ for some torus $Z$, then the torus $Z_\Sigma=Z/Z_2$ also acts on $\Sigma$ through the action of $Z$.

We can now consider the quotient stacks $[\Sigma/G^\theta]$ and $[\Sigma/(G^\theta\times Z_\Sigma)]$. The quotient map $\chi_\Sigma:\Sigma \rightarrow \Cc_\Sigma:=\Sigma\git G^\theta$ induces a natural map $\chi_\Sigma:[\Sigma/G^\theta]\rightarrow \Cc_\Sigma$, that we also denote by $\chi_\Sigma$. We can also descend this to a map $[\Sigma/(G^\theta \times Z_\Sigma)]\rightarrow [\Cc_\Sigma/Z_\Sigma]$.

As we explained in the previous section, the abelianization map $\alpha_\Sigma:\Sigma \rightarrow \Ab_\Sigma:=\Sigma\git G$ factors through  $\Sigma \overset{\chi_\Sigma}{\rightarrow} \Cc_\Sigma \rightarrow \Ab_\Sigma$. The torus $Z_\Sigma$ clearly acts on the abelianization  $\Ab_\Sigma$, so composing with the natural map $\Ab_\Sigma \rightarrow \Spec(k)$, this sequence induces a sequence of quotient stacks
\begin{center}
\begin{tikzcd}
	\text{$[\Sigma/(G^\theta \times Z_\Sigma)]$} \ar{r}{\chi_\Sigma} & \text{$[\Cc_\Sigma/Z_\Sigma]$} \rar & \text{$[\Ab_\Sigma/Z_\Sigma]$} \rar & \B Z_\Sigma.
\end{tikzcd}
\end{center}
Here, $\B Z_\Sigma$ denotes the quotient stack $[\Spec(k)/Z_\Sigma]$, which is the classifying stack of $Z_\Sigma$-bundles.

Let $X$ be a smooth algebraic curve over $k$. Let $\MM_X(\Sigma)$, $\BB_X(\Sigma)$ and $\aA_X(\Sigma)$ denote the stacks of maps from $X$ to $[\Sigma/(G^\theta\times Z_\Sigma)]$, $[\Cc_\Sigma/Z_\Sigma]$ and $[\Ab_\Sigma/Z_\Sigma]$, respectively. We obtain a sequence
\begin{center}
\begin{tikzcd}
	\MM_X(\Sigma) \ar{r}{h_X} & \BB_X(\Sigma) \rar & \aA_X(\Sigma) \rar & \mathrm{Bun}_{Z_\Sigma}(X).
\end{tikzcd}
\end{center}
If we let $L\rightarrow X$ be a $Z_\Sigma$-bundle, the natural map $X\rightarrow \B Z_\Sigma$ associated to $L$ induces the following diagram, where all squares are Cartesian
\begin{center}
\begin{tikzcd}
	L(\text{$[\Sigma/G^\theta]$}) \ar{r}{\chi_L} \ar{d} & L(\Cc_\Sigma) \ar{d} \ar{r} & L(\Ab_\Sigma) \ar{r} \ar{d} & X \ar{d}{L} \\
	\text{$[\Sigma/(G^\theta \times Z_\Sigma)]$} \ar{r}{\chi} & \text{$[\Cc_\Sigma/Z_\Sigma]$} \ar{r} & \text{$[\Ab_{\Sigma}/Z_\Sigma]$} \ar{r} & \B Z_\Sigma.
\end{tikzcd}
\end{center}
These $L([\Sigma/G^\theta])$, $L(\Cc_\Sigma)$ and $L(\Ab_\Sigma)$ are naturally stacks over $X$ and we denote by $\MM_L(\Sigma)$, $\BB_L(\Sigma)$ and $\aA_L(\Sigma)$ their stacks of sections. We obtain a sequence
\begin{center}
\begin{tikzcd}
	\MM_L(\Sigma) \ar{r}{h_L} & \BB_L(\Sigma) \rar & \aA_L(\Sigma).
\end{tikzcd}
\end{center}
This sequence is the fibre over $L\in \mathrm{Bun}_{Z_\Sigma}(X)$ of the sequence defined above.
Fixing a section $s$ in $\aA_L(\Sigma)$, over it we obtain a morphism of stacks
\begin{equation*}
\MM_{s}(\Sigma) \overset{h_{s}}{\longrightarrow} \BB_{s}(\Sigma).
\end{equation*} 

\begin{defn}
The map $h_{s}:\MM_{s}(\Sigma)\rightarrow \BB_{s}(\Sigma)$ is the \emph{multiplicative Hitchin map} over $X$ associated to $\Sigma$ and $s$.
\end{defn}

We can give a more explicit description of all the objects taking part in this definition. The bundle $L(\Ab_\Sigma)$ is just the associated bundle over $X$ defined by the action of $Z_\Sigma$ on the $A_\Sigma$-toric variety $\Ab_\Sigma$. Since $\Ab_\Sigma$ is an affine space, $L(\Ab_\Sigma)$ is a vector bundle of rank equal to the rank of $A_\Sigma$. The bundle $L(\Cc_\Sigma)$ is also an associated bundle, this time to the action of $Z_\Sigma$ on  $\Cc_\Sigma$ and, since  $\Cc_\Sigma$ is also an affine space,  $L(\Cc_\Sigma)$ is also a vector bundle, of rank equal to the rank of $A_\Sigma$ plus the rank of $G/G^\theta$. The stacks $\aA_L(\Sigma)=H^0(X,L(\Ab_\Sigma))$ and $\BB_L(\Sigma)=H^0(X,L(\Cc_\Sigma))$ are just the spaces of sections of these vector bundles.

More precisely, $\Ab_\Sigma$ is the $A_\Sigma$-toric variety with weight semigroup $\PP_+(A_\Sigma)\subset \XX^*(A_\Sigma)\hookrightarrow \XX^*(Z_\Sigma)$. If we take $\gamma_1,\dots,\gamma_s$ to be generators of $\PP_+(A_\Sigma)$, then $L(\Ab_\Sigma)=\bigoplus_{i=1}^s L_{\gamma_i}$, for $L_{\gamma_i}$ the associated line bundle to the action of $Z_\Sigma$ on $\GG_m$ defined by $\gamma_i$. On the other hand, $k[\Cc_\Sigma]$ is generated by the same $e^{\gamma_i}$ and by some functions $b_i\in k[G/G^\theta]$ with weight $\varpi_i$, so, $L(\Cc_\Sigma)=L(\Ab_\Sigma)\oplus \bigoplus_{i=1}^l (\psi^*L)_{\varpi_i}$. Here, $\psi^*L$ is the $A_{G^\theta}$-bundle on $X$ obtained as the image of $L$ through the map $\B Z_{\Sigma}\rightarrow \B A_{G^\theta}$ induced by $\psi:\XX^*(A_{G^\theta})\rightarrow \XX^*(Z_\Sigma)$.

Now, a morphism $X\rightarrow [\Sigma/(G^\theta\times Z_\Sigma)]$ consists of a pair $(E,\varphi)$, where $E\rightarrow X$ is a $G^\theta$-bundle and $\varphi \in H^0(X,E(\Sigma))$ is a section of the associated bundle $E(\Sigma)$ defined by the action of $G^\theta$ on $\Sigma$. Such a pair is called a \emph{mutiplicative $\Sigma$-Higgs pair}.

Now the sequence $\MM_L(\Sigma)\rightarrow \BB_L(\Sigma) \rightarrow \aA_L(\Sigma)$ defining the multiplicative Hitchin map can be explicitly described as
\begin{equation*}
	(E,\varphi)\longmapsto (b(\varphi),\alpha_\Sigma(\varphi)) \longmapsto \alpha_\Sigma(\varphi),
\end{equation*} 
for $b(\varphi)=(b_1(\varphi),\dots,b_l(\varphi))$.

\subsection{Multiplicative \texorpdfstring{$(G,\theta)$}{(G,theta)}-Higgs bundles} \label{gthetahiggs}
Let $G$ be any reductive group and $\theta\in \Aut_2(G)$ an involution. As in the previous section, we let $X$ be a smooth algebraic curve over $k$.

For any positive integer $d$ we denote by $X_d=X^d/\mathfrak{S}_n$ the $d$-th symmetric product, so that elements $D\in X_d$ are effective divisors of degree $d$ on $X$. More generally, given a tuple $\bm{d}=(d_1,\dots,d_n)$ of positive integers, we denote $X_{\bm{d}}=X_{d_1}\times \dots \times X_{d_n}$ and to any element $\bm{D}=(D_1,\dots,D_n)\in X_{\bm{d}}$ we can associate the divisor $D=D_1+\dots+D_n$.

\begin{defn} \label{def:gthetahiggs}
Let $\bm{D}\in X_{\bm{d}}$. A \emph{multiplicative $(G,\theta)$-Higgs bundle with singularities in $\bm{D}$} is a pair $(E,\varphi)$, where $E\rightarrow X$ is a $G^\theta$-bundle and $\varphi$ is a section of the associated bundle of symmetric varieties $E(G/G^\theta)$ defined over the complement $X\setminus \lvert D \rvert$ of the support $\lvert D \rvert$ of the divisor $D$. 

We denote by $\MM_{\bm{d}}(G,\theta)$ the moduli stack of tuples $(\bm{D},E,\varphi)$ with $\bm{D}\in X_{\bm{d}}$ and $(E,\varphi)$ a multiplicative $(G,\theta)$-Higgs bundle with singularities in $\bm{D}$. 
\end{defn}

To any tuple $(\bm{D},E,\varphi)$ in $\MM_{\bm{d}}(G,\theta)$ and any point $x\in \lvert D \rvert$ we can associate an invariant $\inv_x(\varphi)\in \XX_*(A_{G_\theta})_-$, for $A\subset G$ a maximal $\theta$-split torus. This is constructed by considering the completion $\OO_x$ of the local ring $\Oo_{X,x}$ of $X$ at $x$, which by picking a local parameter is isomorphic to the ring $\OO$ of formal power series, and its quotient field $F_x=\mathrm{qf}(\OO_x)\cong F$. By fixing a trivialization of $E$ around the formal disk $\Spec(\OO_x)$ and restricting $\varphi$ to $\Spec(F_x)$ we get an element of $(G/G^\theta)(F_x)$ which is well defined up to the choice of the trivialization, thus, up to the action of $G(\OO_x)$. That is, we obtain a well defined element 
\begin{equation*}
\inv_x(\varphi) \in (G/G^\theta)(F_x)/G(\OO_x) \cong \XX_*(A_{G^\theta})_-.
\end{equation*} 
Globally, if $D=\sum_{x\in X}n_x x$, we get a $\XX_*(A_{G^\theta})_-$-valued divisor
\begin{equation*}
\inv(\varphi)=\sum_{x\in X} n_x\inv_x(\varphi) x.
\end{equation*} 

Now, we can impose a constraint on the invariant $\inv(\varphi)$ in order to obtain a finite-type moduli stack.
We define
\begin{equation*}
\MM_{\bm{d},\bm{\lambda}}(G,\theta) =  \left\{(\bm{D},E,\varphi)\in \MM_{\bm{d}}(G,\theta): \inv(\varphi)= \bm{\lambda}\cdot \bm{D}\right\},
\end{equation*} 
where $\bm{\lambda}=(\lambda_1,\dots,\lambda_n)$ is a tuple of anti-dominant cocharacters $\lambda_i\in \XX_*(A_{G^\theta})_-$ and $\bm{\lambda}\cdot \bm{D}=\sum_{i=1}^n \lambda_i D_i$. For any element $(\bm{D},E,\varphi)$ in $\MM_{\bm{d},\bm{\lambda}}(G,\theta)$, we say that the pair $(E,\varphi)$ is a multiplicative $(G,\theta)$-Higgs bundle with \emph{singularity type} $(\bm{D},\bm{\lambda})$.
The order on $\XX_*(A_{G^\theta})_-$ induces an order on $\XX_*(A_{G^\theta})_-$-valued divisors and thus we can also define the bigger stack
\begin{equation*}
\MM_{\bm{d},\bar{\bm{\lambda}}}(G,\theta) =  \left\{(\bm{D},E,\varphi)\in \MM_{\bm{d}}(G,\theta): \inv(\varphi)\leq \bm{\lambda}\cdot \bm{D}\right\}.
\end{equation*} 

Suppose now that we are in the situation in which $k[G/G^\theta]^{G^\theta}$ is a polynomial algebra and thus it is generated by some functions $b_1,\dots,b_l$ with weights $\varpi_1,\dots,\varpi_l$ respectively, for $\varpi_1,\dots,\varpi_l$ the fundamental dominant weights of the restricted root system $\Phi_\theta$. Consider now the bundle $\BB_{\bm{d},\bm{\lambda}}(G,\theta)\rightarrow X_{\bm{d}}$ whose fibre over $\bm{D}$ is the space of sections
\begin{equation*}
\BB_{\bm{d},\bm{\lambda}}(G,\theta)_{\bm{D}}= \bigoplus_{i=1}^l H^0(X,\Oo_X(\langle \varpi_i,w_0\bm{\lambda}\cdot \bm{D}\rangle)).
\end{equation*} 

\begin{defn}
The \emph{multiplicative Hitchin map} for $(G,\theta)$-Higgs bundles is the morphism
\begin{align*}
h_{\bm{d},\bm{\lambda}}: \MM_{\bm{d},\bm{\lambda}}(G,\theta) & \longrightarrow \BB_{\bm{d},\bm{\lambda}}(G,\theta) \\
(\bm{D},E,\varphi) & \longmapsto (b_1(\varphi),\dots,b_l(\varphi)). 
\end{align*} 
\end{defn}

We can also consider the simpler case where the $d_i$ are equal to $1$ and thus we have that $\bm{D}=\bm{x}=(x_1,\dots,x_n)$. We define
\begin{equation*}
\MM_{\bm{x},\bm{\lambda}}(G,\theta) := \MM_{\vec{1},\bm{\lambda}}(G,\theta)_{\bm{x}} = \left\{(E,\varphi) \in \MM_{\vec{1}}(G,\theta): \inv(\varphi)=\bm{\lambda}\cdot \bm{x}\right\} ,
\end{equation*} 
where $\vec{1}=(1,\dots,1)$, and the bigger stack
\begin{equation*}
\MM_{\bm{x},\bar{\bm{\lambda}}}(G,\theta) := \MM_{\vec{1},\bar{\bm{\lambda}}}(G,\theta)_{\bm{x}} = \left\{(E,\varphi) \in \MM_{\vec{1}}(G,\theta): \inv(\varphi)\leq \bm{\lambda}\cdot \bm{x}\right\} .
\end{equation*} 
When $k[G/G^\theta]$ is polynomial, we can also consider the Hitchin base
\begin{equation*}
\BB_{\bm{x},\bm{\lambda}}(G,\theta) = \BB_{\vec{1},\bm{\lambda}}(G,\theta)_{\bm{x}}=\bigoplus_{i=1}^l H^0(X,\Oo_X(\langle \varpi_i,w_0\bm{\lambda}\cdot \bm{x}\rangle)),
\end{equation*} 
and the Hitchin map
\begin{align*}
h_{\bm{x},\bm{\lambda}}: \MM_{\bm{x},\bm{\lambda}}(G,\theta) & \longrightarrow \BB_{\bm{x},\bm{\lambda}}(G,\theta) \\
(E,\varphi) & \longmapsto (b_1(\varphi),\dots,b_l(\varphi)). 
\end{align*}

\subsection{Relating the two pictures}
In this section we see how, for $G$ semisimple, the multiplicative Hitchin map for $(G,\theta)$-Higgs bundles as defined in Section \ref{gthetahiggs} can be recovered from the multiplicative Hitchin map as defined in Section \ref{higgssymmetric}.

\subsubsection*{The simply-connected case}
Let us suppose that $G$ is simply-connected. Let $\bm{\lambda}=(\lambda_1,\dots,\lambda_n)$ be a tuple of anti-dominant cocharacters $\lambda_i\in\XX_*(A_{G^\theta})_-$.  
 The natural projection $A_{G^\theta}\to A_{G_\theta}$ induces a projection $\XX_*(A_{G^\theta})_- \to \XX_*(A_{G_\theta})_-$ and thus $\bm{\lambda}$ defines a multiplicative map
\begin{align*}
\bm{\lambda}: \GG_m^n & \longrightarrow A_{G_\theta} \\
(z_1,\dots,z_n) & \longmapsto z_1^{\lambda_1}\dots z_n^{\lambda_n}
\end{align*} 
which extends to a map $\bm{\lambda}:\Aa^n \rightarrow \Ab_{\Env(G/G^\theta)}$, since $\PP_+(\Env(G/G^\theta))^\vee=\XX_*(A_{G_\theta})_-$. We can now define a very flat symmetric embedding $\Sigma^{\bm{\lambda}}$ with semisimple part $O_{\Sigma^{\bm{\lambda}}}'=G/G^\theta$ through the Cartesian diagram
\begin{center}
\begin{tikzcd}
\Sigma^{\bm{\lambda}} \ar{r} \ar{d} & \Env(G/G^\theta) \ar{d}{\alpha_{\Env(G/G^\theta)}} \\
\Aa^n \ar{r}{-w_0 \bm{\lambda}} & \Ab_{\Env(G/G^\theta)}. 
\end{tikzcd}
\end{center}
Since for any symmetric embedding the tori $A_\Sigma$ and $Z_\Sigma$ differ by a finite subgroup, by construction we get $Z_{\Sigma^{\bm{\lambda}}}=\GG_m^n$ and $\Ab_{\Sigma^{\bm{\lambda}}}=\Aa^n$, so $\mathrm{Bun}_{ Z_{\Sigma^{\bm{\lambda}}}}(X)=\mathrm{Pic}(X)^n$ and, for any tuple of line bundles $\bm{L}=(L_1,\dots,L_n)$, the fibre of $\aA_X(\Sigma^{\bm{\lambda}})\rightarrow \mathrm{Bun}_{Z_{\Sigma^{\bm{\lambda}}}}(X)$ over $\bm{L}$ is $\aA_{\bm{L}}(\Sigma^{\bm{\lambda}})=\bigoplus_{i=1}^n H^0(X,L_i)$. Replacing $\Env(G/G^\theta)$ by $\Env^0(G/G^\theta)$ we obtain an open dense subvariety $\Sigma^{\bm{\lambda},0}\subset \Sigma^{\bm{\lambda}}$, and we can also consider the open substack $\MM_X(\Sigma^{\bm{\lambda},0})$ of $\MM_X(\Sigma^{\bm{\lambda}})$ of maps from $X$ to $[\Sigma^{\bm{\lambda},0}/(G^\theta\times Z_\Sigma)]$.

Consider now $\bm{d}=(d_1,\dots,d_n)$ a tuple of positive integers and an element $\bm{D}=(D_1,\dots,D_n)\in X_{\bm{d}}$. To each $D_i$ we can associate the line bundle $\Oo_X(D_i)$ which, since $D_i$ is effective (all its coefficients are positive), has a canonical non-vanishing section $s_i \in H^0(X,L_i)$. Let us denote $\Oo_X(\bm{D})=\bigoplus_{i=1}^n \Oo_X(D_i)$ and $\bm{s}=(s_1,\dots,s_n)$.

\begin{thm}
The map
\begin{align*}
X_{\bm{d}} & \longrightarrow \aA_X(\Sigma^{\bm{\lambda}}) \\
\bm{D} & \longmapsto (\Oo_X(\bm{D}),\bm{s})
\end{align*} 
induces the following diagram, where all squares are Cartesian,
\begin{center}
\begin{tikzcd}
	\MM_{\bm{d},\bar{\bm{\lambda}}}(G,\theta) \dar \rar[hook] &	\MM_{\bm{d},\bar{\bm{\lambda}}}(G,\theta) \ar{r}{h_{\bm{d},\bm{\lambda}}} \dar & \BB_{\bm{d},\bm{\lambda}}(G,\theta) \rar\dar & X_{\bm{d}}\dar \\
							   \MM_X(\Sigma^{\bm{\lambda},0}) \rar[hook] &	\MM_X(\Sigma^{\bm{\lambda}}) \ar{r}{h_X} & \BB_X(\Sigma^{\bm{\lambda}})  \rar & \aA_X(\Sigma^{\bm{\lambda}}).
\end{tikzcd}
\end{center}
\end{thm}

\begin{proof}
From the above discussions it is clear that for any $\bm{D}\in X_{\bm{d}}$ we can identify $\BB_{\Oo_X(\bm{D})}(\Sigma)$ with the space of sections of the bundle
\begin{equation*}
\Oo_X(\bm{D}) \oplus \bigoplus_{i=1}^l (\Oo_X(\bm{D})\times_{w_0\bm{\lambda}}\GG_m)_{\varpi_i}= \Oo_X(\bm{D}) \oplus \bigoplus_{i=1}^l \Oo_X(\langle \varpi_i, w_0\bm{\lambda}\cdot \bm{D}\rangle).
\end{equation*} 
It is now clear that the rightmost square is Cartesian. 

Take any point $x\in \lvert D \rvert$ and take $\lambda_x$ the coefficient of the divisor $\bm{\lambda} \cdot \bm{D}$ corresponding to $x$. This $\lambda_x$ is an antidominant cocharacter that can be written as
\begin{equation*}
\lambda_x = \bm{m}_x\cdot \bm{\lambda},
\end{equation*} 
where $\bm{m}_x=(m_{1x},\dots,m_{nx})$ is a vector with components $m_{ix}$ for $D_i=\sum_x m_{ix}x$. Now, by mapping $z\mapsto z^{\bm{m}_x}=(z^{m_{1x}},\dots,z^{m_{nx}})$ we obtain a morphism $\bm{m}:\OO\rightarrow \Aa^n$ and the following diagram, with all squares Cartesian
\begin{center}
\begin{tikzcd}
	\Env^{\lambda_x}(G/G^\theta) \ar{r} \ar{d} & \Sigma^{\bm{\lambda}} \ar{d} \rar & \Env(G/G^\theta) \ar{d}{\alpha_{\Env(G/G^\theta)}} \\
\Spec	\OO \ar{r}{\bm{m}} & \Aa^n \ar{r}{-w_0\bm{\lambda}} & \Ab_{\Env(G/G^\theta)}.
\end{tikzcd}
\end{center}
Now, if we restrict an element $(E,\varphi)\in\MM_X(\Sigma^{\bm{\lambda}})$ to $\Spec(F_x)$ we obtain an element of $(G/G^\theta)_+(F_x)$ which must lie in $\Env^{\lambda_x}(G/G^\theta)$. Therefore, by Proposition \ref{prop:extender_loops}, the image of $\varphi|_{\Spec(F_x)}$ belongs to $\overline{G(\OO_x)z^{\lambda_x}}$ and thus $\inv_x(\varphi)\leq \lambda_x$. Moreover, if we change $\Env(G/G^\theta)$ by $\Env^0(G/G^\theta)$ in the diagram above, the element $\varphi|_{\Spec(F_x)}$ lies in $\Env^{\lambda_x,0}(G/G^\theta)$ and therefore its image belongs to $G(\Oo_x)z^{\lambda_x}$, so $\inv_x(\varphi)=\lambda_x$.
\end{proof}

\subsubsection*{The non simply-connected case} Suppose now that $G$ is a semisimple group and consider $\pi:\hat{G}\rightarrow G$ the simply-connected cover.  This map is a central isogeny and thus its kernel is a subgroup of the centre $\pi_1(G)\subset Z_{\hat{G}}$, the fundamental group of $G$. Moreover, it is a well known result of Steinberg \cite{steinberg_endomorphisms}*{Theorem 9.16} that any involution $\theta\in \Aut_2(G)$ lifts to an involution $\hat{\theta}\in \Aut_2(\hat{G})$ making the following diagram commute
\begin{center}
\begin{tikzcd}
\hat{G} \ar{r}{\hat{\theta}} \ar{d}{\pi} & \hat{G} \ar{d}{\pi} \\
G \ar{r}{\theta} & G.
\end{tikzcd}
\end{center}

The following is also well-known.
\begin{lemma}
$\pi(\hat{G}^{\hat{\theta}})=G_0^\theta$.	
\end{lemma}

\begin{proof}
The inclusion $\pi(\hat{G}^{\hat{\theta}})\subset G_0^\theta$ is clear. Indeed, for any $g\in \hat{G}^{\hat{\theta}}$, we have 
\begin{equation*}
\theta(\pi(g))=\pi(\hat{\theta}(g))=\pi(g),
\end{equation*} 
so $\pi(\hat{G}^{\hat{\theta}})\subset G^\theta$ and the inclusion follows since $\pi$ is continuous and $\hat{G}^{\hat{\theta}}$ is connected.

For the other inclusion, take $k\in G_0^\theta$ and $g\in \hat{G}$ with $\pi(g)=k$. Now,
\begin{equation*}
\pi(g\hat{\theta}(g)^{-1})=k\theta(k)^{-1}=1.
\end{equation*} 
Thus, $g\hat{\theta}(g)^{-1}\in \pi_1(G)\subset Z_{\hat{G}}$, so $g\in \hat{G}_{\hat{\theta}}$. Recall now the decomposition $G_{\hat{\theta}}=F_{\hat{\theta}} \hat{G}^{\hat{\theta}}$, which implies that there exists some $f\in F_{\hat{\theta}}$ and some $g_0\in \hat{G}^{\hat{\theta}}$ such that $g=fg_0$. Let $k_0=\pi(g_0)\in G_0^\theta$. Clearly, $\pi(f)\in F^\theta$, so we obtain a decomposition
\begin{equation*}
k=\pi(f)k_0,
\end{equation*} 
with $k,k_0\in G_0^\theta$ and  $\pi(f)\in F^\theta$. This implies that $\pi(f)=1$, and thus $k=k_0=\pi(g_0)$, and $k\in \pi(\hat{G}^{\hat{\theta}})$.
\end{proof}

It follows from the lemma above that we can define an action of $G_0^\theta$ on $\hat{G}/\hat{G}^{\hat{\theta}}$ from the left multiplication action of $\hat{G}^{\hat{\theta}}$. Indeed, for any $k\in G_0^\theta$ we just have to take any $h\in \hat{G}^{\hat{\theta}}$ with $\pi(h)=k$ and define $k\cdot (gG^\theta)=hgG^\theta$. This is well defined since, if $h'$ is such that $\pi(h')=k$, then $h'=zh$ for some $z\in G^\theta \cap Z_G$ and thus $h'gG^\theta=hgG^\theta$.

For the following, we need to introduce a new object.

\begin{defn}
A \emph{multiplicative $(G,\theta)_0$-Higgs bundle} is a pair $(E,\varphi)$, where $E\rightarrow X$ is a $G_0^\theta$-bundle and $\varphi$ is a section of the associated bundle of symmetric varieties $E(G/G^\theta)$ defined over $X'$ the complement of a finite subset of $X$.
\end{defn}

The exact sequence $1\rightarrow \pi_1(G)\rightarrow \hat{G}\rightarrow G\rightarrow 1$, induces a fibration $\Gamma\hookrightarrow \hat{G}/\hat{G}^{\hat{\theta}}\twoheadrightarrow G/G^\theta$, for some finite group $\Gamma$. For any $G_0^\theta$-bundle, we can consider the bundles $E(G/G^\theta)$ and $E(\hat{G}/\hat{G}^{\hat{\theta}})$, associated to the actions of $G_0^\theta$ and we obtain a exact sequence of sets with a distinguished element (where $H^0$ stands for sets of sections)
\begin{center}
\begin{tikzcd}
	H^0(X', E(\hat{G}/\hat{G}^{\hat{\theta}})) \rar & H^0(X', E(G/G^\theta)) \ar{r}{\delta} & H^1(X',\Gamma).
\end{tikzcd}
\end{center}
It follows that to any multiplicative $(G,\theta)_0$-Higgs bundle $(E,\varphi)$ we can associate the invariant $\delta(\varphi)\in H^1(X',\Gamma)$. Now, the section $\varphi$ comes from a section of $E(\hat{G}/\hat{G}^{\hat{\theta}})$ if and only if $\delta(\varphi)=1$. Moreover, the map $\pi:\hat{G}^{\hat {\theta}}\rightarrow G_0^\theta$ induces a map $\mathrm{Bun}_{\hat{G}^{\hat{\theta}}}\rightarrow \mathrm{Bun}_{G_0^\theta}$. We conclude the following.

\begin{prop}
Any multiplicative $(G,\theta)_0$-Higgs bundle with $\delta(\varphi)=1$ is induced from a multiplicative $(\hat{G},\hat{\theta})$-Higgs bundle from the maps $\mathrm{Bun}_{\hat{G}^{\hat{\theta}}}\rightarrow \mathrm{Bun}_{G_0^\theta}$ and $\hat{G}/\hat{G}^{\hat{\theta}}\rightarrow G/G^\theta$.
\end{prop}

The next step consists on relating multiplicative $(G,\theta)_0$-Higgs bundles with multiplicative $(G,\theta)$-Higgs bundles. We start by considering a $(G,\theta)$-Higgs bundle $(E,\varphi)$. The decomposition $G^\theta =F^\theta G_0^\theta$ allows us to define an isomorphism  $G^\theta\cong F^\theta \ltimes G_0^\theta$, and this induces a factorization of the $G^\theta$-bundle $E\rightarrow X$ as
\begin{center}
\begin{tikzcd}
E \ar{dd} \ar{dr} \\
& Y=E/G_0^\theta \ar{dl} \\
X.
\end{tikzcd}
\end{center}
Here, $Y\rightarrow X$ is a Galois étale cover with Galois group equal to $F^\theta$, while $E\rightarrow Y$ has the structure of a principal $G_0^\theta$-bundle. Moreover, the deck transformation $Y\rightarrow Y$ induced by an element $a\in F^\theta$ lifts to a map $\tilde{a}:E\rightarrow E$ and thus defines a right $F^\theta$-action on the $G_0^\theta$-bundle $E\rightarrow Y$, which clearly gives an $\Int$-twisted $F^\theta$-equivariant structure on it. By this we mean that, for any $e\in E$, $a\in F^\theta$ and $g\in G_0^\theta$, we have
\begin{equation*}
	\tilde{a}(e) \cdot g = \tilde{a}(e\cdot aga^{-1}).
\end{equation*} 
A similar argument allows us to see the bundle $E(G/G^\theta)\rightarrow X$ as a $F^\theta$-equivariant bundle on $Y$, and there is a bijective correspondence between sections of $E(G/G^\theta)\rightarrow X$ and $F^\theta$-equivariant sections of $E(G/G^\theta)\rightarrow Y$. We refer the reader to \cite{noconexo} for more details on these correspondences. We have proved the following.

\begin{prop}
Any multiplicative $(G,\theta)$-Higgs bundle is induced from a multiplicative $(G,\theta)_0$-Higgs bundle on a certain $F^\theta$-Galois cover $Y\rightarrow X$.	
\end{prop}

\section{Multiplicative Higgs bundles, involutions and fixed points} \label{invhiggs}
\subsection{Multiplicative Higgs bundles and moduli spaces}
Let $G$ be a reductive group and $X$ a smooth algebraic curve over $k$. Let $\bm{d}=(d_1,\dots,d_n)$ be a tuple of positive integers.

\begin{defn}
Let $\bm{D}\in X_{\bm{d}}$. A \emph{multiplicative $G$-Higgs bundle with singularities in $\bm{D}$} is a pair $(E,\varphi)$, where $E\rightarrow X$ is a $G$-bundle and $\varphi$ is a section of the adjoint bundle of groups $E(G)$ defined over the complement $X\setminus |D|$ of the support $|D|$ of the divisor $D$.

We denote by $\MM_{\bm{d}}(G)$ the moduli stack of tuples $(\bm{D},E,\varphi)$ with $\bm{D}\in X_{\bm{d}}$ and $(E,\varphi)$ a multiplicative $G$-Higgs bundle with singularities in $\bm{D}$.
\end{defn}

An invariant for multiplicative $G$-Higgs bundles can be defined in the same way as in Section \ref{gthetahiggs}. We simply restrict $\varphi$ to $\Spec F_x$ for any point $x\in \lvert D \rvert$ and obtain a well defined orbit  $G(\OO)z^\lambda G(\OO)$ for $\lambda\in \XX_*(T)_+$, and $T\subset G$ a maximal torus. We put $\inv_x(\varphi)=\lambda$. Thus, we can also define, for any tuple $\bm{\lambda}=(\lambda_1,\dots,\lambda_n)$ of dominant cocharacters $\lambda_i\in \XX^*(T)_+$, the finite-type moduli stack
\begin{equation*}
\MM_{\bm{d},\bm{\lambda}}(G)=\left\{(\bm{D},E,\varphi)\in \MM_{\bm{d}}(G): \inv(\varphi) = \bm{\lambda}\cdot \bm{D}\right\},
\end{equation*} 
and the bigger stack
\begin{equation*}
\MM_{\bm{d},\bar{\bm{\lambda}}}(G)=\left\{(\bm{D},E,\varphi)\in \MM_{\bm{d}}(G): \inv(\varphi) \leq \bm{\lambda}\cdot \bm{D}\right\}.
\end{equation*} 

It is clear that multiplicative $G$-Higgs bundles are a particular case of the pairs introduced in Definition \ref{def:gthetahiggs}, for the group $G\times G$ endowed with the involution $\Theta(g_1,g_2)=(g_2,g_1)$ (see Example \ref{ex:diagonal} for more details). This picture can also be related to the picture of Section \ref{higgssymmetric} by using \emph{Vinberg's enveloping monoid} \cite{vinberg}; this has been studied in the works of Bouthier and J. Chi \cites{bouthier_Springer, bouthier_fibration, bouthier-chi,chi} and in the thesis of G. Wang \cite{wang}. 

For any involution $\theta\in \Aut_2(G)$, there is a natural map 
\begin{align*}
\MM_{\bm{d}}(G,\theta) & \longrightarrow  \MM_{\bm{d}}(G) \\
(\bm{D},E,\varphi) & \longmapsto (\bm{D},E_G,\varphi_G),
\end{align*} 
sending a multiplicative $(G,\theta)$-Higgs bundle $(E,\varphi)$ to the multiplicative $G$-Higgs bundle $(E_G,\varphi_G)$, where
\begin{equation*}
E_G=E\times_{G^\theta} G
\end{equation*} 
is the natural extension of the structure group of $E$ from $G^\theta$ to $G$ and, regarding $\varphi$ as a $G^\theta$-equivariant map $f:E|_{X\setminus |D|}\rightarrow M^\theta$, the section $\varphi_G$ is determined by the map
\begin{align*}
f_G: E_G|_{X\setminus |D|} & \longrightarrow G \\
[e,g] & \longmapsto gf(e)g^{-1}.
\end{align*} 
Here, we are identifying the symmetric variety $G/G^\theta$ with the subvariety $M^\theta=\tau^\theta(G)$ so that for any $g\in G^\theta$, the image  $gf(e)g^{-1}$ stays in $M^\theta$. Moreover, recall that for any involution $\theta\in \Aut_2(G)$ we have a natural inclusion $\XX_*(A_{G^\theta})\subset \XX_*(T)$, for $A\subset G$ a maximal $\theta$-split torus and $T$ a maximal torus containing it. Now, the map $\MM_{\bm{d}}(G,\theta)\rightarrow \MM_{\bm{d}}(G)$ restricts to a map
\begin{align*}
\MM_{\bm{d},\bm{\lambda}}(G,\theta) & \longrightarrow \MM_{\bm{d},w_0\bm{\lambda}}(G). 
\end{align*} 
Indeed, we have the following lemma.

\begin{lemma}
	For any multiplicative $(G,\theta)$-Higgs bundle $(E,\varphi)$ and any singular point $x$ of it, we have
$\inv_x(\varphi_G)=w_0 \inv_x(\varphi)$.
\end{lemma}

\begin{proof}
The decomposition $(G/G^\theta)(F)=\bigsqcup_{\lambda \in \XX_*(A_{G_\theta})_-} G(\OO)z^\lambda$ can be reformulated in terms of $M^\theta$ as
\begin{equation*}
M^\theta(F)=\bigsqcup_{\lambda \in \XX_*(A_{G_\theta})_-} G(\OO)*_{\theta}z^\lambda.
\end{equation*} 
Now, the inclusion $M^\theta\subset G$ induces an inclusion $M^\theta(F)\subset G(F)$ and it is clear that each orbit $G(\OO)*_\theta z^\lambda$ is mapped inside the orbit $G(\OO)z^\lambda G(\OO)=G(\OO)z^{w_0\lambda}G(\OO)$.
\end{proof}

Note that a section $\varphi$ of the adjoint bundle $E(G)$ is simply an automorphism $\varphi:E\rightarrow E$. A \emph{morphism} of multiplicative $G$-Higgs bundles $(E_1,\varphi_1)\rightarrow (E_2,\varphi_2)$ is then a map $E_1\rightarrow E_2$ such that the following square commutes
\begin{center}
\begin{tikzcd}
E_1|_{X\setminus |D|} \ar{r} \ar{d}{\varphi_1} & E_2|_{X\setminus |D|} \ar{d}{\varphi_2} \\
 E_1|_{X\setminus |D|}\ar{r} & E_2|_{X\setminus |D|}.
\end{tikzcd}
\end{center}
We say that a multiplicative $G$-Higgs bundle $(E,\varphi)$ is \emph{simple} if its only automorphisms are those given by multiplying by elements of the centre $Z_G\subset G$.

A \emph{moduli space} for simple multiplicative $G$-Higgs bundles can be considered after an argument by Hurtubise and Markman \cite{hurtubise-markman} that we explain now. Let $\rho:G\hookrightarrow \GL(V)$ be a faithful representation of $G$ and, for $T\subset G$ a maximal torus, let
\begin{equation*}
\PP(V)=\left\{\chi \in \XX^*(T): \exists v\in V\setminus \left\{0\right\}\text{ such that } \rho(t)v=t^\chi v, \forall t\in T\right\}
\end{equation*} 
be its weight lattice. For any cocharacter $\lambda \in \XX_*(T)$ we define the number
\begin{equation*}
d_\rho(\lambda)=\min\left\{\langle w\lambda,\chi\rangle : w\in W_T, \chi\in \PP(V)\right\},
\end{equation*} 
for $W_T$ the Weyl group of $T$. Now, if $\lambda\in \XX_*(T)_+$, for any $\phi\in G(F)$ that lies in the orbit $G(\OO)z^\lambda G(\OO)$, the highest pole order in the coefficients of the matrix $\rho(\phi)\in \End V \otimes F$ is $d_\rho(\lambda)$. Thus, if we write $\bm{\lambda}\cdot \bm{D}=\sum_{x\in X}\lambda_x x$, for $\bm{\lambda}=(\lambda_1,\dots,\lambda_n)$ a tuple of dominant cocharacters and $\bm{D}\in X_{\bm{d}}$, we can consider the divisor
\begin{equation*}
	(\bm{\lambda}\cdot \bm{D})_\rho = -\sum_{x\in X} d(\lambda_x)x,
\end{equation*} 
and for any multiplicative $G$-Higgs bundle $(E,\varphi)$ with singularity type $(\bm{D},\bm{\lambda})$, the representation $\rho$ induces a section
\begin{equation*}
\rho(\varphi)\in H^0(X\setminus |D|,\End E(V) \otimes \Oo_X((\bm{\lambda}\cdot \bm{D})_\rho)).
\end{equation*} 
Indeed, if $\bm{\lambda}\geq \bm{\mu}$, then $d_\rho(\lambda)\leq d_\rho(\mu)$, so $(\bm{\lambda}\cdot \bm{D})_\rho \geq (\bm{\mu}\cdot \bm{D})_\rho$.

A pair $(E,\varphi)$, where $E$ is a $G$-bundle on $X$ and $\varphi$ is a section of $\End E(V)\otimes L$ for some representation $\rho:G\rightarrow \GL(V)$ and for $L$ a line bundle, is called an \emph{$L$-twisted $\rho$-Higgs bundle}. We denote the moduli stack of these pairs by $\Higgs_{L,\rho}(G)$. What we have just proved is that $\rho$ induces a well-defined inclusion
\begin{align*}
\MM_{\bm{d},\bar{\bm{\lambda}}}(G)_{\bm{D}} & \longrightarrow \Higgs_{\Oo_X((\bm{\lambda}\cdot \bm{D})_\rho),\rho}(G). 
\end{align*} 
The stability conditions and good moduli space theory for $L$-twisted $\rho$-Higgs bundles are well known (the reader may refer to \cite{oscar-gothen-mundet} or \cite{schmitt} for general treatments of the topic). The existence of moduli spaces of simple $L$-twisted $\rho$-Higgs bundles guarantees the existence of good moduli spaces of simple pairs $\Mm_{\bm{d},\bm{\lambda}}(G)$ and $\Mm_{\bm{d},\bar{\bm{\lambda}}}(G)$ for $\MM_{\bm{d},\bm{\lambda}}(G)$ and $\MM_{\bm{d},\bar{\bm{\lambda}}}(G)$, at least as algebraic spaces. For more details, see Remark 2.5 in \cite{hurtubise-markman}. Finally, for any involution $\theta\in \Aut_2(G)$ we can consider the space $\Mm_{\bm{d},\bm{\lambda}}(G,\theta)$ defined as the intersection of $\Mm_{\bm{d},w_0\bm{\lambda}}(G)$  with the image of the map $\MM_{\bm{d},\bm{\lambda}}(G,\theta)\rightarrow \MM_{\bm{d},w_0\bm{\lambda}}(G)$. Analogously, we can define $\Mm_{\bm{d},\bar{\bm{\lambda}}}(G,\theta)$.

For most discussions in this section, we will be interested in fixing the divisor $\bm{D}$, and thus we denote the fibre $\Mm_{\bm{D},\lambda}(G)=\Mm_{\bm{d},\lambda}(G,\theta)_{\bm{D}}$, and similarly for $\Mm_{\bm{D},\lambda}(G,\theta)$.

\subsection{Involutions on the moduli space} 
Let $\theta\in \Aut_2(G)$ be an involution. To any multiplicative $G$-Higgs bundle $(E,\varphi)$ over $X$ we can associate other two multiplicative $G$-Higgs bundles obtained from $(E,\varphi)$ and the action of the involution $\theta$. These are the pairs
\begin{equation*}
\iota^\theta_{\pm}(E,\varphi) = (\theta(E),\theta(\varphi)^{\pm 1}).
\end{equation*} 
Here, $\theta(E)$ stands for the associated principal bundle $E\times_\theta G$ or, equivalently, it has the same total space as $E$ but it is equipped with the $G$-action $e\cdot_\theta g= e\cdot \theta(g)$. The section $\theta(\varphi)$ is obtained as the composition of $\varphi$ with the natural map
\begin{align*}
E(G) & \longrightarrow \theta(E)(G) \\
[e,g] & \longmapsto [e,\theta(g)].
\end{align*} 
Equivalently, if we regard $\varphi$ as an automorphism $\varphi:E|_{X'}\rightarrow E|_{X'}$, then $\theta(\varphi)$ is set theoretically the same map as $\varphi$, but now regarded as an automorphism of $\theta(E)|_{X'}$. We can give one last interpretation of $\theta(\varphi)$ by noting that the associated $G$-equivariant maps $f_\varphi,f_{\theta(\varphi)}:E|_{X'}\rightarrow G$ are related by
\begin{equation*}
f_{\theta(\varphi)}=\theta \circ f_\varphi.
\end{equation*} 

Let $T\subset G$ be a $\theta$-stable maximal torus and $B\subset G$ a Borel subgroup contained in a minimal $\theta$-split parabolic and containing $T$. Given any dominant cocharacter $\lambda \in \XX_*(T)_+$, the associated cocharacter $\theta(\lambda)$ is not dominant in general, but there is a dominant cocharacter in its orbit under the action of the Weyl group $W_T$. We denote this element by $\theta(\lambda)_+$. The map $\lambda \mapsto \theta(\lambda)_+$ defines an involution on $\XX_*(T)_+$. One checks immediately that, if $x$ is a singular point of $(E,\varphi)$, then
\begin{equation*}
\inv_x(\theta(E),\theta(\varphi)^{\pm}) = (\pm \theta(\inv_x(E,\varphi)))_+.
\end{equation*} 

Therefore, if $\bm{\lambda}=(\lambda_1,\dots,\lambda_n)$ is a tuple of dominant coweights $\lambda_i \in \XX_*(T)$ with $\theta(\lambda_i)_+=\pm \lambda_i$, then for any $\bm{D}\in X_{\bm{d}}$, we have the following involutions on the moduli space of simple pairs
\begin{align*}
\iota^\theta_{\pm}: \Mm_{\bm{D},\bm{\lambda}}(G) & \longrightarrow \Mm_{\bm{D},\bm{\lambda}}(G) \\
[E,\varphi] & \longmapsto [\theta(E),\theta(\varphi)^{\pm 1}],
\end{align*} 
where $\iota^\theta_+$ is well defined if $\theta(\lambda_i)_+=\lambda_i$ and $\iota^\theta_-$ if $(-\theta(\lambda_i))_+=\lambda_i$. We will denote these two involutions together by writing $\iota_\eps^\theta$, for $\eps=\pm 1$.

Moreover, note that the actions of the inner automorphisms on the moduli space is trivial since, if $\alpha=\Int_g$ is an inner automorphism of $G$, then the map $e\mapsto e\cdot g^{-1}$ gives an isomorphism between any principal $G$-bundle $E$ and $E\times_{\alpha} G$, commuting with any $\varphi$. Therefore, the above involutions are well defined at the level of the outer class $a$ of $\theta$ in $\Out_2(G)$. That is, given any element $a\in \Out_2(G)$, we can define
\begin{align*}
\iota^a_{\eps}: \Mm_{\bm{D},\bm{\lambda}}(G) & \longrightarrow \Mm_{\bm{D},\bm{\lambda}}(G) \\
[E,\varphi] & \longmapsto [\theta(E),\theta(\varphi)^{\eps}],
\end{align*} 
for any representative $\theta$ of the class $a$.

\begin{rmk}
Note that it follows from the Cartan decomposition of the formal loop group $G(F)$ that if $\theta$ and $\theta'$ are two involutions representing the same element in $\Out_2(G)$, then $\theta(\lambda)_+=\theta'(\lambda)_+$.	
\end{rmk}

We are interested in studying the spaces $\Mm_{\bm{D},\bm{\lambda}}(G)^{\iota^a_{\eps}}$ of fixed points under these involutions. For this it will be useful for us to think about an isomorphism between the $G$-bundles $E$ and $\theta(E)$, for $\theta\in \Aut_2(G)$, as a \emph{$\theta$-twisted automorphism}. By this we mean an automorphism $\psi:E\rightarrow E$ of the total space $E$, such that
\begin{equation*}
\psi(e\cdot g)=\psi(e)\cdot \theta(g).
\end{equation*} 
Therefore, a multiplicative $G$-Higgs bundle $(E,\varphi)$ will be a fixed point of the involution $\iota^a_\eps$ if and only if there exists a $\theta$-twisted automorphism $\psi:E\rightarrow E$, for any element $\theta$ of the class $a\in \Out_2(G)$, such that the following diagram commutes
\begin{center}
\begin{tikzcd}
E|_{X'} \ar{r}{\psi} \ar{d}{\varphi} & E|_{X'} \ar{d}{\varphi^\eps} \\
E|_{X'} \ar{r}{\psi} & E|_{X'}.
\end{tikzcd}
\end{center}
Equivalently, if we define $f_\psi:E\rightarrow G$ as $\psi(e)=e\cdot f_\psi(e)$, then $(E,\varphi)\in \Mm_{\bm{D},\bm{\lambda}}(G)^{\iota^a_\eps}$ if and only if, for any $e\in E$, we have
\begin{equation*}
f_\psi(e)\theta(f_\varphi(e))f_\psi(e)^{-1}=f_\varphi(e)^{\eps}.
\end{equation*} 

\subsection{Fixed dominant cocharacters}
Before describing the fixed points of the involution $\iota^\theta_-$, it is important that we give a good description of the involution at the level of dominant cocharacters
\begin{align*}
\XX_*(T)_+ & \longrightarrow \XX_*(T)_+ \\
\lambda & \longmapsto (-\theta(\lambda))_+,
\end{align*} 
for $T$ a $\theta$-stable maximal torus.

As we mentioned above, in general $-\theta(\lambda)$ is not a dominant cocharacter, and that is why we need to take $(-\theta(\lambda))_+$ the dominant cocharacter in its Weyl group orbit. However, there is an exception to this, which is when there are no imaginary roots. 

Indeed, the dominant Weyl chamber in $\XX_*(T)_+ \otimes_{\ZZ} \mathbb{Q}$ is spanned by a basis of fundamental coweights $\left\{\mu_1,\dots,\mu_r\right\}$, which is the dual basis of the set of simple roots $\Delta_T=\left\{\alpha_1,\dots,\alpha_r\right\}$. Since $\theta$ is an involution, we have $\langle \chi^\theta, \theta(\lambda)\rangle = \langle \chi,\lambda \rangle$ for any $\chi\in \XX^*(T)$ and any $\lambda \in \XX_*(T)$. Therefore, $\left\{\theta(\mu_1),\dots,\theta(\mu_r)\right\}$ is the dual basis of $\left\{\alpha^\theta_1,\dots,\alpha_r^\theta\right\}$. Recall now that there is an involution $\sigma$ on the set of simple roots which are not imaginary such that $\alpha^\theta_i+\alpha_{\sigma(i)}$ is an imaginary root. Thus, if there are no imaginary roots, we conclude that $\alpha_i^\theta=-\alpha_{\sigma(i)}$ and $\theta(\mu_i)=-\mu_{\sigma(i)}$. Summing up, we have the following.

\begin{lemma}
If $\Phi_T^\theta = \varnothing$, then for any dominant cocharacter $\lambda\in \XX_*(T)_+$, the cocharacter $-\theta(\lambda)$ is also dominant, and thus $(-\theta(\lambda))_+=-\theta(\lambda)$.	
\end{lemma}

Recall from Lemma \ref{prop:quasisplit_existence} that for any involution $\theta\in \Aut_2(G)$ there exists an inner automorphism $u_\theta \in \Int(G)$ such that the involution $\theta_q:=u_\theta \circ \theta$ is quasisplit. Therefore, if we let $T_q$ be a maximal $\theta_q$-stable torus, the involution $\lambda \mapsto (-\theta_q(\lambda))_+$ is just the involution $\lambda \mapsto -\theta_q(\lambda)$. The fixed points of this involution are simply the dominant cocharacters $\XX_*(A_q)_+$, where $A_q$ is the maximal $\theta$-split torus formed by the anti-fixed points of $\theta_q$ in $T_q$.

The fixed points of the involution $\lambda \mapsto -\theta(\lambda)$ in $\XX_*(T)$ are the cocharacters $\XX_*(A)$ of the maximal $\theta$-split torus $A$ of anti-fixed points of $\theta$ in $T$. Any two maximal tori are conjugate and thus there is a canonical bijection $\XX_*(T)_+\rightarrow \XX_*(T_q)_+$. It follows from the above that under this bijection the dominant cocharacters $\XX_*(A)_+$ are mapped inside $\XX_*(A_q)_+$, but in general not every element of $\XX_*(A_q)_+$ comes from an element of $\XX_*(A)_+$. In fact, the image of $\XX_*(A)_+$ under this map is the set
\begin{equation*}
\XX_*(A_q)_+^{u_\theta}= \left\{\lambda \in \XX_*(A_q)_+: u_\theta(\lambda) = \lambda \right\}.
\end{equation*} 

Given an element $a\in \Out_2(G)$, a quasisplit representative $\theta_q$ of $a$ and a dominant cocharacter $\lambda \in \XX_*(A_q)_+$, we can consider the set
\begin{equation*}
a_\lambda = \left\{\theta \text{ representative of } a: u_\theta(\lambda)=\lambda \right\}.
\end{equation*} 
That is, the elements of $a_\lambda$ are the involutions $\theta$ for which $\lambda$ is in the image of the cocharacters of a maximal $\theta$-split torus. This set clearly descends to the clique $\cl^{-1}(a)$ and we can define
\begin{equation*}
\cl^{-1}(a)_\lambda = \left\{[\theta] \in \cl^{-1}(a): u_\theta(\lambda) = \lambda\right\}.
\end{equation*} 

\subsection{The fixed points}
We move on now to describe the fixed points of the involutions $\iota_\eps^a$, for $a\in \Out_2(G)$. A first result is clear.

\begin{prop}\label{inclusion} Let $\theta\in \Aut_2(G)$ be a representative of $a$, $T\subset G$ a maximal $\theta$-stable torus and $B\subset G$ a Borel subgroup contained in a minimal $\theta$-split parabolic and containing $T$. Let $\bm{\lambda}$ be a tuple of dominant cocharacters of $T$.
	\begin{enumerate}
\item If the $\lambda_i \in \XX_*(T^\theta)$, the subspace $\widetilde{\Mm}_{\bm{D},\bm{\lambda}}(G^\theta)\subset \Mm_{\bm{D},\bm{\lambda}}(G)$, defined as the intersection of $\Mm_{\bm{D},\bm{\lambda}}(G)$ with the image of the moduli stack of multiplicative $G^\theta$-bundles of type $(\bm{D},\bm{\lambda})$ under extension of the structure group, is contained in the fixed point subspace $\Mm_{\bm{D},\bm{\lambda}}(G)^{\iota^a_+}$.
\item Any extension $(E_G,\varphi_G)$ of a pair $(E,\varphi)$ where $E$ is a $G^\theta$-bundle and $\varphi$ is a section of $E|_{X\setminus |D|}(S^\theta)$, for the conjugation action of $G^\theta$ on $S^\theta$, is a fixed point of $\iota^a_-$. In particular, if $\bm{\lambda}=(-\theta(\bm{\lambda}))_+$, $\Mm_{\bm{D},\bm{\lambda}}(G,\theta)\subset \Mm_{\bm{D},\bm{\lambda}}(G)^{\iota^a_-}$. 
	\end{enumerate}
\end{prop}

\begin{proof}
Notice first that if $E_G=E\times_{G^\theta} G$ is the extension of a $G^\theta$ bundle $E$, then the map $\psi:E_G\rightarrow E_G$ defined as $\psi(e)=e$ for $e\in E$ and $\psi(e\cdot g)=e\cdot \theta(g)$ is clearly a well-defined $\theta$-twisted automorphism of $E_G$. Now, if for every $e\in E$ we have $\theta(f_\varphi(e))=f_\varphi(e)^\eps$, then
\begin{align*}
	f_{\varphi_G}(e\cdot g)^\eps &= g^{-1} f_\varphi(e)^\eps g = g^{-1} \theta(f_\varphi(e)) g \\ & = g^{-1}\theta(g) \theta(f_\varphi(e\cdot g)) \theta(g)^{-1}g \\ &= f_\psi(e\cdot g) \theta(f_\varphi(e\cdot g)) f_\psi(e\cdot g)^{-1}.
\end{align*} 
And thus $(E_G,\varphi_G)$ is a fixed point of $\iota^a_\eps$.
\end{proof}

The main step towards the description of the fixed points is provided by the following theorem.

\begin{thm}\label{fixedpoints}
If $(E,\varphi)$ is a simple multiplicative $G$-Higgs bundle with $(E,\varphi)\cong \iota^a_\eps(E,\varphi)$, then:
\begin{enumerate}
	\item There exists a unique $[\theta]\in \cl^{-1}(a)$ such that there is a reduction of structure group of $E$ to a $G^\theta$-bundle $E_\theta\subset E$.
	\item If we consider the corresponding $G$-equivariant map $f_\varphi:E|_{X'}\rightarrow G$, then $f_\varphi|_{E^\theta}$ takes values into $G^\theta$ if $\eps=1$, and $S^\theta$ if $\eps=-1$.
\end{enumerate}
More precisely, when $\eps=-1$, $f_\varphi|_{E_\theta}$ takes values in a single orbit $M^\theta_s\subset S^\theta$, for some $s\in S^\theta$ unique up to $\theta$-twisted conjugation.
\end{thm}

\begin{rmk}
The statement $(1)$ of the above theorem regarding the reduction of the structure group of $E$ was proven in \cite{oscar-ramanan}*{Proposition 3.9}.	
\end{rmk}

\begin{proof}
Let $\theta_0\in \Aut_2(G)$ be any representative of the class $a$. By hypothesis, there exists some $\theta_0$-twisted automorphism $\psi:E\rightarrow E$ making the following diagram commute	
\begin{center}
\begin{tikzcd}
E|_{X'} \ar{r}{\psi} \ar{d}{\varphi} & E|_{X'} \ar{d}{\varphi^\eps} \\
E|_{X'} \ar{r}{\psi} & E|_{X'}.
\end{tikzcd}
\end{center}
Here, $X'=X\setminus |D|$ denotes the complement of the singularity divisor of  $\varphi$.
Now, independently of the value of $\eps$, we have that the following diagram commutes
\begin{center}
\begin{tikzcd}
E|_{X'} \ar{r}{\psi^2} \ar{d}{\varphi} & E|_{X'} \ar{d}{\varphi} \\
E|_{X'} \ar{r}{\psi^2} & E|_{X'},
\end{tikzcd}
\end{center}
and $\psi^2(e\cdot g)=\psi(\psi(e)\cdot \theta_0(g))=\psi^2(e)\cdot g$, so $\psi^2$ is an automorphism of $(E,\varphi)$. Since by assumption $(E,\varphi)$ is simple, there exists some $z\in Z_G$ such that $\psi^2(e)=e\cdot z$ for every $e\in E$. Therefore,
\begin{equation*}
e\cdot z=\psi^2(e)=\psi(e\cdot f_\psi(e))=\psi(e)\cdot \theta_0(f_\psi(e))=e\cdot [f_\psi(e)\theta_0(f_\psi(e))].
\end{equation*} 
We conclude that $f_\psi$ maps $E$ into $S_{\theta_0}$.

Note that $f_\psi$ is $G$-equivariant for the $\theta_0$-twisted conjugation action. Therefore, it descends to a morphism $X=E/G\rightarrow S_{\theta_0}/(G\times Z_G)\cong \cl^{-1}(a)$. Now, since $\cl^{-1}(a)$ is finite and $X$ is irreducible, this maps $X$ to a single element $[\theta]\cong \cl^{-1}(a)$. Let us take $r\in S_{\theta_0}$ a representative of the class corresponding to $[\theta]$. In other words, we take $r$ such that $\theta=\Int_r \circ \theta_0$. 

Let us consider the subset $E_\theta=f_{\psi}^{-1}(r)\subset E$. If $e\in E_\theta$, then another element $e\cdot g$ in the same fibre of $E$ belongs to $E_\theta$ if and only if
\begin{equation*}
r=f_\psi(e\cdot g) = g^{-1}*_{\theta_0}f_\psi(e)=g^{-1}r\theta_0(g).
\end{equation*} 
Thus, $e\cdot g\in E_\theta$ if and only if $g=\Int_r\circ \theta_0(g)=\theta(g)$. That is, if $g\in G^\theta$. Therefore,  $E_\theta$ defines a principal $G^\theta$-bundle to which $E$ reduces. This proves $(1)$.

As we explained above, if $\psi$ determines the isomorphism between $(E,\varphi)$ and $(\theta_0(E),\theta_0(\varphi)^\eps)$, then for every $e\in E$ we must have
\begin{equation*}
f_\psi(e)\theta_0(f_\varphi(e))f_\psi(e)^{-1}=f_{\varphi}(e)^\eps.
\end{equation*} 
Thus, if $e\in E_\theta=f_\psi^{-1}(r)$, we have
\begin{equation*}
\theta(f_{\varphi}(e))=r\theta_0(f_\varphi(e))r^{-1}=f_\psi(e)\theta_0(f_\varphi(e))f_\psi(e)^{-1}=f_{\varphi}(e)^\eps,
\end{equation*} 
proving $(2)$.

Suppose now that $\eps=-1$. In that case we have a map $f_\varphi|_{E_\theta}:E_\theta|_{X'}\rightarrow S^\theta$. Since this map is $G^\theta$-equivariant, we can quotient by $G^\theta$ and obtain a map from $X'=E_{\theta}|_{X'}/G^\theta$ to the quotient of $S^\theta$ by the conjugation action of $G^\theta$. We can further quotient $S^\theta$ by the $\theta$-twisted conjugation action and obtain a well defined map $X'\rightarrow S^\theta/G$. Again, since $X'$ is irreducible and $S^\theta/G$ is finite, we conclude that $f_\varphi$ maps $E_\theta|_{X'}$ to a single $\theta$-twisted orbit $M_s^\theta$, for some element $s\in S^\theta$.
\end{proof}

\begin{rmk} 
	Notice that if there exists some $e\in E$ such that $f_\varphi(e)$ is semisimple or unipotent, then $\theta$ can be chosen so that $f_\varphi|_{E_\theta}$ takes values in the symmetric variety $M^\theta$. Indeed, in that case we can choose $r$ to be $r=f_\psi(e)$, so that $e\in E_\theta$, for $\theta=\Int_r\circ \theta_0$. We know from the above that $s=f_{\varphi}(e)\in S^\theta$ and that $f_\varphi$ maps $E_\theta|_{X'}$ to the $\theta$-twisted orbit $M_s^\theta$. However, if $s$ is semisimple or unipotent, we have by Richardson \cite{richardson}*{Lemmas 6.1-6.3} that in fact $s\in M^\theta$ and thus $M^\theta_s=M^\theta$.
\end{rmk}

When $\theta$ and $\theta'=\Int_g\circ \theta \circ \Int_g^{-1}$ are two involutions of $G$ related by the equivalence relation $\sim$, we have $\widetilde{\Mm}_{\bm{D},\bm{\lambda}}(G^\theta)=\widetilde{\Mm}_{\bm{D},\bm{\lambda}}(G^{\theta'})$ and $\Mm_{\bm{D},\bm{\lambda}}(G,\theta)=\Mm_{\bm{D},\bm{\lambda}}(G,\theta')$. Therefore, the above theorem allows us to decompose $\Mm_{\bm{D},\bm{\lambda}}(G)^{\iota^a_+}$ in the components $\widetilde{\Mm}_{\bm{D},\bm{\lambda}}(G^\theta)$, for $\theta$ the elements of the clique $\cl^{-1}(a)$. That is, we get
\begin{equation*}
\Mm_{\bm{D},\bm{\lambda}}(G)^{\iota_a^+}=\bigcup_{[\theta]\in \cl^{-1}(a)}\widetilde{\Mm}_{\bm{D},\bm{\lambda}}(G^\theta)	.
\end{equation*} 

The situation for $\eps=-1$ is a little bit more involved, as we explain now.

\begin{defn}
Let $(G,\theta,s)$ be a triple consisting of a reductive group $G$, an involution $\theta\in \Aut_2(G)$ and an element $s\in S^\theta$. A \emph{multiplicative $(G,\theta,s)$-Higgs pair} $(E,\varphi)$ on $X$ is a pair consisting of a principal $G^\theta$-bundle $E\rightarrow X$, and a section $\varphi$ of the bundle $E(M^\theta_s)$ associated to the conjugation action of $G$ on $M^\theta_s$, and defined over the complement $X'$ of a finite subset of $X$.
\end{defn}

\begin{rmk}
	Note that this is a new kind of object, different from the others defined in this paper. Indeed these are not multiplicative $(G,\theta)$-Higgs bundles since the bundle $E$ is a $G^\theta$-bundle but the variety $M^\theta_s$, although it is a symmetric variety, is equal to $M^\theta_s=M^{\theta_s}s\cong G/G^{\theta_s}$, so its stabilizer is not equal to $G^\theta$.
\end{rmk}

To a multiplicative $(G,\theta,s)$-Higgs pair we can associate an invariant $\inv(\varphi)$, which is a divisor with values in the quotient $M^\theta_s(F)/G(\Oo)$. Now, since $M^\theta_s=M^{\theta_s}s$, this quotient can in fact be identified with the semigroup $\XX_*(A_{G^{\theta_s}})_-$, and the invariant $\inv(\varphi)$ is a $\XX_*(A_{G^{\theta_s}})_-$-valued divisor. Here, $A$ is a maximal $\theta_s$-split torus.

We denote by $\MM_{\bm{d},\bm{\lambda}}(G,\theta,s)_{\bm{D}}$ the moduli stack of multiplicative $(G,\theta,s)$-Higgs pairs with singularity type $(\bm{D},\bm{\lambda})$, for $\bm{D}\in X_{\bm{d}}$, where  $\bm{\lambda}$ is a tuple of anti-dominant cocharacters in $\XX_*(A_{G^{\theta_s}})_-$.  As in the case of multiplicative $(G,\theta)$-Higgs bundles, there is a natural map
\begin{equation*}
\MM_{\bm{d},\bm{\lambda}}(G,\theta,s)_{\bm{D}} \longrightarrow \MM_{\bm{d},w_0\bm{\lambda}}(G)_{\bm{D}}.
\end{equation*} 
We denote by $\Mm_{\bm{D},\bm{\lambda}}(G,\theta,s)\subset \Mm_{\bm{D},w_0\bm{\lambda}}(G)$ the intersection of the image of this map with $\Mm_{\bm{D},w_0\bm{\lambda}}(G)$.

In this language, we can write the consequences of the results of this section as follows.

\begin{corol} \label{fixedpoints2}
Let $\theta_q$ be the quasisplit involution representing the class $a\in \Out_2(G)$ and let $\bm{\lambda}$ be a tuple of anti-dominant cocharacters $\lambda_i\in \XX_*(A_{G^{\theta_q}})_+$, for $A$ a maximal $\theta_q$-split torus. Then,
\begin{equation*}
\Mm_{\bm{D},w_0\bm{\lambda}}(G)^{\iota_a^-} = \bigcup_{[\theta]\in \cl^{-1}(a)} \bigcup_{[s]\in (S^\theta/G)_{\bm{\lambda}}} \Mm_{\bm{D},\bm{\lambda}}(G,\theta,s).
\end{equation*} 
Here, $(S^\theta/G)_{\bm{\lambda}}$ denotes the set
\begin{equation*}
(S^\theta/G)_{\bm{\lambda}} = \left\{[s]\in S^\theta/G: u_{\theta_s}(\bm{\lambda})=\bm{\lambda}\right\},
\end{equation*} 
where $\theta_s=\Int_s \circ \theta$ and $u_{\theta_s}$ is the inner automorphism such that $\theta_q=u_{\theta_s}\circ \theta_s$.
\end{corol}

\begin{proof}
Indeed, note that $\bm{\lambda}$ can only be a valid invariant for a multiplicative $(G,\theta,s)$-Higgs pair if and only if the $\lambda_i$ are in the image of $\XX_*(A^s_{G^{\theta_s}})$, for $A^s$ a maximal $\theta_s$-split torus, thus, if and only if $\theta_s\in a_{\lambda_i}$ for every $i=1,\dots,n$.
\end{proof}

\subsection{Fixed points, submanifolds and the symplectic structure}
In this section we consider the particular case where $X$ is a Calabi--Yau curve, that is, the canonical line bundle of $X$ is the trivial line bundle $K_X=\Oo_X$. A Calabi--Yau curve must be either $\Aa^1$, $\GG_m$, or an elliptic curve.

Hurtubise and Markman \cite{hurtubise-markman} showed that, assuming that $X$ is a Calabi--Yau curve, the moduli space $\Mm_{\bm{D},\bm{\lambda}}(G)$ admits an algebraic symplectic structure $\Omega$.

\begin{thm}\label{symplectic}
For any $a\in \Out_2(G)$, 
\begin{equation*}
	(\iota_{a}^\eps)^* \Omega = \eps \Omega.
\end{equation*} 
Therefore, $\Mm_{\bm{D},\bm{\lambda}}(G)^{\iota_a^+}$ is an algebraic symplectic submanifold, while $\Mm_{\bm{D},\bm{\lambda}}(G)^{\iota_a^-}$ is an algebraic Lagrangian submanifold of $\Mm_{\bm{D},\bm{\lambda}}(G)$.
\end{thm}

We recall first some details about the deformation theory of multiplicative Higgs bundles and the construction of the Hurtubise--Markman symplectic structure $\Omega$ on $\Mm_{\bm{D},\bm{\lambda}}(G)$. We start by describing the tangent bundle of $\Mm_{\bm{D},\bm{\lambda}}(G)$. Using standard arguments in deformation theory (see, for example, \cite{biswas-ramanan}), Hurtubise and Markman showed that the Zariski tangent space of $\Mm_{\bm{D},\bm{\lambda}}(G)$ at a point $[E,\varphi]$ is equal to $\HH^1(C_{[E,\varphi]})$ the first hypercohomology space of the deformation complex
\begin{center}
\begin{tikzcd}
	C_{[E,\varphi]}: & E(\g) \rar{\ad_\varphi} & \ad(E,\varphi).
\end{tikzcd}
\end{center}
Here, $\ad(E,\varphi)$ is a vector bundle on $X$ defined by the following short exact sequence 
\begin{center}
\begin{tikzcd}
	0 \rar & \left\{(a,b): a + \Ad_\varphi(b) = 0\right\} \rar[hook] & E(\g)\oplus E(\g) \rar & \ad(E,\varphi) \rar & 0,
\end{tikzcd}
\end{center}
where, if we regard $b\in E(\g)$ and $\varphi\in E(G)$, as maps $f_b:E\rightarrow \g$, $f_\varphi:E\rightarrow G$, by $\Ad_{\varphi}(b)$ we mean the map sending any $e\in E$ to $\Ad_{f_\varphi(e)}(f_b(e))$. The map $\ad_\varphi$ is defined as
\begin{equation*}
\ad_\varphi= L_\varphi - R_\varphi,
\end{equation*} 
where $L_\varphi$ and $R_\varphi$ are given by the following diagram
\begin{center}
\begin{tikzcd}
	E(\g) \dar{i_1} \arrow{rd}{L_\varphi} & \\
	E(\g)\oplus E(\g) \rar & \ad(E,\varphi) \\
	E(\g). \uar{i_2} \arrow{ru}{R_\varphi} & 
\end{tikzcd}
\end{center}
\begin{rmk}
In terms of a faithful representation $\rho:G\hookrightarrow \GL(V)$, the maps $L_\varphi$	and $R_\varphi$ correspond precisely to left and right multiplication by  $\rho(\varphi)$, and the deformation complex becomes
\begin{align*}
\ad_{\rho(\varphi)}: \End E(V_\rho) & \longrightarrow \End E(V_\rho) \otimes \Oo_X((\bm{\lambda}\cdot \bm{D})_\rho) \\
\psi & \longmapsto [\rho(\varphi),\psi].
\end{align*} 
This is precisely the deformation complex for $\Oo_X((\bm{\lambda}\cdot \bm{D})_\rho)$-twisted Higgs bundles.
\end{rmk}

We now have to give a explicit description for the tangent space $\HH^1(C_{[E,\varphi]})$. We do this by taking an acyclic resolution of the complex $C_{[E,\varphi]}$ and computing cohomology of the total complex. For example, we can take the \v{C}ech resolution associated to an acyclic étale cover of $X$, that we denote by $\left\{U_i\right\}_{i}$. We can now compute $\HH^1(C_{[E,\varphi]})$ as the quotient $Z/B$, where $Z$ consists of pairs $(\bm{s},\bm{t})$, with $\bm{s}=(s_{ij})_{i,j}$, $\bm{t}=(t_i)_i$, for the $s_{ij}\in \Gamma(U_i\cap U_j, E(\g))$, and $t_i \in \Gamma(U_i,\ad(E,\varphi))$, satisfying the equations
\begin{equation*}
\begin{cases}
s_{ij}+s_{jk}=s_{ij}, \\
t_i-t_j = \ad_\varphi(s_{ij});
\end{cases}
\end{equation*} 
and $B$ is the set of pairs $(\bm{s},\bm{t})$ of the form $\bm{s}=(r_i-r_j)_{i,j}$ and $\bm{t}=(\ad_\varphi(r_i))_i$, for some $\bm{r}=(r_i)_i$, with $r_i\in \Gamma(U_i,E(\g))$. Now one obtains a deformation of the pair $(E,\varphi)$ from a pair $(\bm{s},\bm{t})$ by considering the pair $(E,\varphi)_{(\bm{s},\bm{t})}$ over $X\times \Spec(k[\delta])$, for $k[\delta]=k[t]/(t^2)$, determined by 
\begin{equation*}
	\begin{cases}
\bm{g}^{\bm{s}} = \bm{g}(1 + \delta \bm{s}), \\
\bm{\phi}^{\bm{t}} = \bm{\phi}(1 + \delta \bm{t}).
	\end{cases}
\end{equation*} 
Here, $\bm{g}=(g_{ij})_{i,j}$ are the transition functions of $E$ and $\bm{\phi}=(\phi_i)_{i}$ is determined by restricting $\varphi$ to the $U_i$.

With this explicit description of the tangent space we can finally compute the differential of the involution $\iota_a^\eps$. Indeed, $(\iota^\eps_a)_*(\bm{s},\bm{t})$ is such that
\begin{equation*}
\iota_a^\eps(\bm{g}^{\bm{s}},\bm\phi^{\bm{t}}) = (\theta(\bm{g})^{(\iota^\eps_a)_*(\bm{s})}, (\theta(\bm{\phi})^{\eps})^{(\iota^\eps_a)_*(\bm{t})}).
\end{equation*} 
Now, 
\begin{equation*}
\iota_a^\eps(\bm{g}^{\bm{s}},\bm\phi^{\bm{t}}) = (\theta(\bm{g})(1+\delta \theta(\bm{s})),\theta(\bm{\phi})^\eps(1+\delta \theta(\bm{t}))^\eps)=(\theta(\bm{g})^{\theta(\bm{s})},(\theta(\bm{\phi})^{\eps})^{\eps \theta(\bm{t})}),
\end{equation*} 
where, for the last equality, we have used that, in $k[\delta]$,
\begin{equation*}
	(1+a\delta )^\eps=1+\eps a\delta.
\end{equation*} 
We conclude that
\begin{equation*}
	(\iota_a^\eps)_*(\bm{s},\bm{t})=(\theta(\bm{s}),\eps\theta(\bm{t})).
\end{equation*} 

Consider now the dual complex of the deformation complex
\begin{center}
\begin{tikzcd}
	C^*_{[E,\varphi]}: & \ad(E,\varphi)^* \rar{\ad_\varphi^*} & E(\g^*).
\end{tikzcd}
\end{center}
Grothendieck--Serre duality gives a perfect pairing
\begin{align*}
\HH^1(C_{[E,\varphi]}) \times \HH^1(C^*_{[E,\varphi]}\otimes K_X)& \longrightarrow H^1(X,K_X) \cong k. 
\end{align*} 
However, since we assumed that $K_X=\Oo_X$, we get a perfect pairing between $\HH^1(C_{[E,\varphi]})$ and $\HH^1(C^*_{[E,\varphi]})$, so we can identify the cotangent space of the moduli space at $[E,\varphi]$ with $\HH^1(C^*_{[E,\varphi]})$. 

An invariant bilinear form on $\g$ gives a natural isomorphism between $\g$ and its dual $\g^*$. Under this isomorphism, the complex $C^*_{[E,\varphi]}$ can be identified with
\begin{center}
\begin{tikzcd}
	C^*_{[E,\varphi]}: & \ad(E,\varphi^{-1}) \rar{-\ad_\varphi} & E(\g).
\end{tikzcd}
\end{center}
We can then describe the cotangent space $\HH^1(C^*_{[E,\varphi]})$ as the quotient $Z^*/B^*$, where $Z^*$ consists of pairs $(\bm{\sigma},\bm{\tau})$, with $\bm{\sigma}=(\sigma_{ij})_{i,j}$, $\bm{\tau}=(\tau_{i})_i$, $\sigma_{ij}\in \Gamma(U_i\cap U_j,\ad(E,\varphi^{-1}))$, $\tau_i\in \Gamma(U_i,E(\g))$, satisfying
\begin{equation*}
\begin{cases}
\sigma_{ij}+\sigma_{jk}=\sigma_{ij}, \\
\tau_i-\tau_j = -\ad_\varphi(\sigma_{ij});
\end{cases}
\end{equation*} 
and $B^*$ is the set of pairs $(\bm{\sigma},\bm{\tau})$ of the form $\bm{\sigma}=(\eta_i-\eta_j)_{i,j}$ and $\bm{\tau}=(-\ad_\varphi(\eta_i))_i$, for some $\bm{\eta}=(\eta_i)_i$, with $\eta_i\in \Gamma(U_i,\ad(E,\varphi^{-1}))$. One can now check that, under this description, the Grothendieck--Serre duality pairing is given explicitly by
\begin{align*}
\Phi:\HH^1(C_{[E,\varphi]}) \times \HH^1(C^*_{[E,\varphi]}\otimes K_X) & \longrightarrow H^1(X,K_X) \cong k \\
([\bm{s},\bm{t}],[\bm{\sigma},\bm{\tau}]) & \longmapsto  \Phi((\bm{s},\bm{t}),(\bm{\sigma},\bm{\tau}))= \langle \bm{s},\bm{\tau} \rangle- \langle \bm{\sigma},\bm{t}\rangle,
\end{align*} 
where $\langle-,-\rangle$ denotes the duality pairing.

Finally, note that we can define a morphism of complexes $\Psi:C_{[E,\varphi]} \rightarrow C^*_{[E,\varphi]}$, as in the following diagram
\begin{center}
\begin{tikzcd}
	C_{[E,\varphi]} \dar{\Psi} : &	E(\g) \rar{\ad_\varphi} \dar{-L_{\varphi^{-1}}} & \ad(E,\varphi) \dar{L_{\varphi^{-1}}} \\
	C_{[E,\varphi]}^*: &	\ad(E,\varphi^{-1}) \rar{-\ad_\varphi} & E(\g).
\end{tikzcd}
\end{center}
The adjoint of this morphism is given by
\begin{center}
\begin{tikzcd}
		C_{[E,\varphi]} \dar{\Psi^\dagger} : & E(\g) \rar{\ad_\varphi} \dar{-R_{\varphi^{-1}}} & \ad(E,\varphi) \dar{R_{\varphi^{-1}}} \\
	C_{[E,\varphi]}^*: & \ad(E,\varphi^{-1}) \rar{-\ad_\varphi} & E(\g).
\end{tikzcd}
\end{center}
Now, these two maps are homotopic, since the following diagram commutes
\begin{center}
\begin{tikzcd}
	E(\g) \rar{\ad_\varphi} \dar{-\ad_{\varphi^{-1}}} & \ad(E,\varphi) \dar{\ad_{\varphi^{-1}}} \ar{ld}{h} \\
	\ad(E,\varphi^{-1}) \rar{-\ad_\varphi} & E(\g),
\end{tikzcd}
\end{center}
for $h=L_{\varphi^{-1}} \circ R_{\varphi^{-1}}$. Therefore, $\Psi$ and $\Psi^\dagger$ define the same map in hypercohomology. Moreover, this map is an isomorphism. We can now define the \emph{Hurtubise--Markman symplectic form} as
\begin{align*}
\Omega: \HH^1(C_{[E,\varphi]}) \times \HH^1(C_{[E,\varphi]}) & \longrightarrow k \\
(\bm{v},\bm{w}) & \longmapsto \Phi(\bm{v}, \Psi(\bm{w})).
\end{align*} 
This is obviously non-degenerate and it is clearly a $2$-form since
\begin{align*}
	\Omega((\bm{s},\bm{t}),(\bm{s}',\bm{t}')) &= \langle \bm s, \Psi(\bm t')\rangle- \langle \Psi(\bm s'),\bm t\rangle = \langle \Psi^\dagger(\bm s), \bm t'\rangle- \langle \bm s',\Psi^\dagger(\bm t)\rangle\\ &= \langle \Psi(\bm s), \bm t'\rangle- \langle \bm s',\Psi(\bm t)\rangle=-\Omega((\bm s',\bm t'),(\bm{s},\bm t)).
\end{align*} 
Hurtubise and Markman \cite{hurtubise-markman}*{\S 5} show that it is closed.

We now have all the ingredients for the proof of Theorem \ref{symplectic}.

\begin{proof}[Proof of Theorem \ref{symplectic}]
	\begin{align*}
		(\iota_{a}^\eps)^* \Omega((\bm{s},\bm{t}),(\bm{s}',\bm{t}')) &= \Omega((\theta(\bm{s}),\eps \theta(\bm{t})),(\theta(\bm{s}'),\eps \theta(\bm{t}'))) \\
									     &= \langle \theta(\bm s),\eps \Psi^\theta(\theta(\bm t'))\rangle- \langle \Psi^\theta(\theta(\bm s')),\eps \theta(\bm t)\rangle \\
									     &= \langle\theta(\bm s),\eps \theta(\Psi(\bm t'))\rangle- \langle\theta(\Psi(\bm s')),\eps \theta(\bm t)\rangle \\
									     &= \eps [\langle\bm s, \Psi(\bm t')\rangle- \langle\Psi(\bm s'),\bm t\rangle ] \\
									     &= \eps \Omega((\bm{s},\bm{t}),(\bm{s}',\bm{t}')),
	\end{align*} 
since the bilinear form can be taken to be invariant under automorphisms of $\g$.
\end{proof}

\appendix
\section{Root systems} \label{appendix}
We recall the basics of root systems.

Let $(\EE,(\cdot,\cdot))$ be a finite-dimensional Euclidean vector space. Recall that a (crystallographic) \emph{root system} in $\EE$ is a finite subset $\Phi\subset \EE$, whose elements are called \emph{roots}, such that:
\begin{enumerate}
	\item The roots span $\EE$.
	\item For any two roots $\alpha, \beta$, the vector $s_\alpha(\beta)=\beta-2\frac{(\alpha,\beta)}{(\alpha,\alpha)}\alpha$ is in $\Phi$.
	\item For any two roots $\alpha,\beta$, the number $2\frac{(\alpha,\beta)}{(\alpha,\alpha)}$ is an integer.
\end{enumerate}
We say that a root system is \emph{reduced} if the only scalar multiples of a root $\alpha$ that belong to $\Phi$ are $\alpha$ and $-\alpha$. In any case, note that even if $\Phi$ is nonreduced, the only multiples of a root $\alpha$ that can belong to $\Phi$ are $\pm \alpha$, $\pm \tfrac{1}{2} \alpha$ or $\pm 2 \alpha$, since
 \begin{equation*}
2\frac{(\alpha,\lambda \alpha)}{(\alpha, \alpha)}= 2\lambda \ \ \text{ and }\ \ 2\frac{(\alpha,\lambda \alpha)}{(\lambda \alpha, \lambda \alpha)}= 2/\lambda
\end{equation*} 
must be integers. 

Given a root system $\Phi$, one can fix a subset $\Phi^+$ of \emph{positive roots} which is closed under the sum and such that for each $\alpha \in \Phi$ either $\alpha$ or $-\alpha$ belong to $\Phi^+$ (but not both). When $\Phi$ is reduced, the indecomposable elements of $\Phi^+$ form the set of \emph{simple roots} $\Delta$, and every root $\alpha\in \Phi$ can be written as a linear combination of elements in $\Delta$ with integer coefficients. Fixing a set of simple roots $\Delta$ is equivalent to fixing the positive roots $\Phi^+$. This choice is unique up to the action of the \emph{Weyl group} of $\Phi$, which is the finite group $W$ generated by the reflections $s_\alpha$. This group naturally acts on $\EE$, and its fundamental domains are the \emph{Weyl chambers}, which are the connected components of the complement of the union of the hyperplanes perpendicular to each root $\alpha \in \Phi$. Given a choice of simple roots $\Delta$, the corresponding \emph{dominant Weyl chamber} is the one defined as 
\begin{equation*}
C^+_\Delta = \left\{v\in \EE: (v,\alpha)\geq 0, \forall \alpha \in \Delta \right\}.
\end{equation*} 
One can also consider the \emph{anti-dominant Weyl chamber}
\begin{equation*}
C^-_\Delta = \left\{v\in \EE: (v,\alpha)\leq 0, \forall \alpha \in \Delta \right\},
\end{equation*} 
which is related to $C^+_\Delta$ by the \emph{longest element} of $W$, which is the element $w_0$ of maximal length as a word in the $s_\alpha$, for $\alpha \in \Delta$.

To any root system $\Phi$ in $(\EE,(\cdot,\cdot))$, one can associate its \emph{dual root system}
\begin{equation*}
\Phi^\vee = \left\{\alpha^\vee\in \EE^*: \alpha \in \Phi\right\},
\end{equation*} 
where $\alpha^\vee$ is the \emph{coroot} of $\alpha$, defined as
\begin{equation*}
\langle v,\alpha^\vee \rangle =2\frac{(v,\alpha)}{(\alpha,\alpha)}
\end{equation*} 
for any $v\in \EE$. Here,  $\langle \cdot ,\cdot \rangle$ denotes the duality pairing of  $\EE$ and $\EE^*$. From the definition of a root system, it is clear that this duality pairing restricts to a pairing $\langle \cdot, \cdot \rangle: \Phi \times \Phi^\vee \rightarrow \ZZ$. When $\Phi$ is reduced, the \emph{simple coroots} $\Delta^\vee$, which are the duals of the simple roots, give the simple roots for the root system $\Phi^\vee$.

The \emph{root and weight lattices} of a root system $\Phi$ are, respectively
\begin{align*}
	\Rr(\Phi)&= \ZZ\langle \Phi\rangle \subset \EE\\	
\PP(\Phi)&= \left\{v\in \EE: \langle v, \alpha^\vee \rangle \in \ZZ, \forall \alpha \in \Phi\right\}.
\end{align*}
We can also consider the \emph{coroot and coweight lattices}, respectively,
\begin{align*}
	\Rr^\vee(\Phi)&=\Rr(\Phi^\vee)= \ZZ\langle\Phi^\vee\rangle \subset \EE^*\\	
\PP^\vee(\Phi)&= \PP(\Phi^\vee)= \left\{v \in \EE^*: \langle \alpha, v \rangle \in \ZZ, \forall \alpha \in \Phi\right\}.
\end{align*}
Note that the pairing $\langle \cdot, \cdot \rangle$ defines a perfect pairing between the root lattice and the coweight lattice, and between the weight lattice and the coroot lattice.

The elements of the intersection $\PP_+(\Delta)=\PP(\Phi)\cap C_\Delta^+$ are called the \emph{dominant weights} of $\Phi$. This intersection is a semigroup spanned by some elements $\omega_1,\dots,\omega_n$ called the \emph{fundamental dominant weights}. That is, we have
\begin{align*}
	\PP_+(\Delta)&=\mathbb{N}\langle \omega_1,\dots,\omega_n\rangle, \\	
	\PP(\Phi)&=\mathbb{Z}\langle \omega_1,\dots,\omega_n\rangle. 
\end{align*}
Dually, and assuming that $\Delta^\vee$ defines a set of simple coroots, we can consider the intersection 
$$\PP^\vee_+(\Delta)=\PP^\vee(\Phi)\cap C_{\Delta^\vee}^+=\left\{\lambda \in \PP^\vee(\Phi): \langle \lambda, \alpha \rangle \geq 0, \forall \alpha \in \Delta \right\}.$$
This is the set of \emph{dominant coweights} of $\Phi$. Again, this set is a semigroup spanned by some elements  $\lambda_1,\dots,\lambda_n$ called the \emph{fundamental dominant coweights}, and
\begin{align*}
	\PP^\vee_+(\Delta)&=\mathbb{N}\langle \lambda_1,\dots,\lambda_n\rangle, \\	
	\PP(\Phi)&=\mathbb{Z}\langle \lambda_1,\dots,\lambda_n\rangle. 
\end{align*}

A choice of the simple roots $\Delta$ also determines an order on $\EE$, and thus also on $\PP(\Phi)$ and $\Rr(\Phi)$, given by
\begin{equation*}
v\geq v' \text{ if and only if } v-v' \in \mathbb{N}\langle \Delta \rangle. 
\end{equation*} 

Note that, if $\Phi$ is a nonreduced root system, and $\alpha$ is a root with $2\alpha \in \Phi$ then $(2\alpha)^\vee=\alpha^\vee/2$.
Thus, when $\Phi$ is nonreduced, in order for the above definitions to work properly, we need to define the simple roots as the union of the indecomposable elements with their positive multiples that belong to $\Phi$. Given that definition, the dual simple roots $\Delta^\vee$ will be simple roots for $\Phi^\vee$ and we will be able to define Weyl chambers and fundamental and dominant weights and coweights just like in the reduced case. Note that this change in the definition does not change the Weyl group, since the reflection associated to $\alpha$ is the same that the one associated to $2\alpha$.

Finally, in this paper we consider the \emph{multiplicative invariants} of a root system. By this we mean the ring $k[e^{\PP(\Phi)}]^W$ of $W$-invariants of the group algebra of the weight lattice $k[e^{\PP(\Phi)}]$. We write $e^{\PP(\Phi)}$ in order to regard the weight lattice as a multiplicative abelian group, rather than an additive one. Now, given any element $a\in k[e^{\PP(\Phi)}]$, which is of the form
\begin{equation*}
a=\sum_{\omega \in \PP(\Phi)} a_{\omega}e^{\omega},
\end{equation*} 
we define the weights of $a$ to be those $\omega$ such that $a_{\omega}\neq 0$. The maximal elements among these weights are called the \emph{highest weights} of $a$. The main result here is the following.

 \begin{prop}[\cite{bourbaki}*{VI.3.4 Theorem 1}]
Let $\Phi$ be a reduced root system, with $\Delta\subset \Phi$ a choice of simple roots. Let $\omega_1,\dots,\omega_n$ be the corresponding fundamental dominant weights of $\Phi$ and, for each $i$, let $a_i\in k[e^{\PP(\Phi)}]^W$ be a $W$-invariant element with unique highest weight $\omega_i$. Then, there is an isomorphism
\begin{equation*}
k[e^{\PP(\Phi)}]^W \cong k[a_1,\dots,a_n].
\end{equation*} 
\end{prop}

\end{document}